\documentclass[10pt,a4paper]{amsart}	
\usepackage{amsfonts}			
\usepackage{amsmath}			
\usepackage{amsthm}			
\usepackage{amssymb}			
\usepackage{graphicx}			
\usepackage{enumerate}			
\usepackage{mathrsfs}			
\usepackage{comment}			
\usepackage{hyperref}
\usepackage{pstricks}			
\usepackage[all,knot,arc]{xy}		
\usepackage{rotating}
\usepackage{array}
\usepackage{stmaryrd}
\usepackage{StyleEng}
\usepackage{bbm}
\usepackage{wrapfig}
\usepackage{subcaption}
\usepackage{MnSymbol}

\newcommand{\nk}{\mathbbm{k}}

\graphicspath{{}}

\newcommand{\executeiffilenewer}[3]{%
\ifnum\pdfs

trcmp{\pdffilemoddate{#1}}%
{\pdffilemoddate{#2}}>0%
{\immediate\write18{#3}}\fi%
}

\newcommand{\trop}[1]{\text{``}{#1}\text{''}}

\newcommand{\fps}[1]{\left\llbracket #1 \right\rrbracket}

\newcommand{\subrelcpct}{\Subset}	
\newcommand{\monid}[1]{\mc{P}(#1)}	
\newcommand{\Newton}[1]{\mc{N}(#1)}	
\newcommand{\partialNewton}[1]{\Delta\mc{N}(#1)}	
\newcommand{\VertNewton}[1]{\mc{V\!N}(#1)}
\newcommand{\nf}[1]{\wt{#1}}		
\newcommand{\mon}[1]{\mc{M}(#1)}	

\renewcommand{\setminus}{\,\backslash\,}

\title{Tropical Hopf manifolds and contracting germs}
\author[M. Ruggiero]{Matteo Ruggiero}
\address{IMJ, Universit\'e Paris $7$, B\^atiment Sophie Germain, Case $7012$, $75205$ Paris Cedex $13$, France.}
\email{ruggiero@math.univ-paris-diderot.fr}
\author[K. Shaw]{Kristin Shaw}
\address{Kristin Shaw, 
Technische Universit\"at
Berlin, MA 6-2, 10623 Berlin, Germany.}
\email{shaw@math.tu-berlin.de} 
\address{}
\email{}
\date{\today}

\makeindex

\begin{document}

\begin{abstract}

Classical Hopf manifolds are compact complex manifolds whose universal covering is $\nC^d \setminus \{0\}$. We investigate the tropical analogues of Hopf manifolds, and relate their geometry to tropical  contracting germs. To do this we develop a procedure called monomialization which transforms non-degenerate tropical germs into morphisms, up to tropical modification.  A link is provided between tropical  Hopf manifolds  and the analytification of Hopf manifolds over a non-archimedean field.  
We conclude by computing the tropical Picard group and $(p,q)$-homology groups.

\end{abstract}

\maketitle

\pagestyle{plain}				


\section*{Introduction}

This paper deals with a tropical analogue of the classical construction of Hopf manifolds.
In complex geometry,  Hopf manifolds are defined by the property 
that their universal covering space is $\nC^d \setminus \{0\}$.
Equivalently, they are obtained as a quotient of $\nC^d \setminus \{0\}$ by a free and properly discontinuous action of a group $G$.
When the group $G$ is infinite cyclic, the resulting Hopf manifold is called  \emph{primary}.  In this case $G$ is generated by a contracting automorphism $f:\nC^d \to \nC^d$ fixing the origin.
A complex primary Hopf manifold of dimension $n$ is diffeomorphic to $S^{2n-1} \times S^1$. 
All Hopf manifolds are finite quotients of \emph{primary} Hopf manifolds.

The construction of (primary) Hopf manifolds is \emph{local}, meaning 
up to biholomorphism the variety $X(f)$ associated to a contracting automorphism fixing the origin $f:(\nC^d,0) \to (\nC^d,0)$  depends only on the conjugacy class of the germ of $f$ at the origin.
Hence the classification of primary Hopf manifolds coincides with the classification of germs of contracting automorphisms of $(\nC^d,0)$.
This type of classification is known as Poincar\'e-Dulac normal forms, see \cite{sternberg:localcontractions,rosay-rudin:holomorphicmaps,berteloot:methodeschangementechelles}.
These normal forms allow a detailed study of  the dynamical features of contracting automorphisms, and also the geometry of Hopf manifolds.

Hopf manifolds are among the easiest examples of non-K\"ahler (hence not projective algebraic) compact manifolds.
For this and other reasons, analogous constructions have been studied in rigid geometry, 
see \cite{voskuil:nonarchhopf}.

\smallskip

Here we consider a natural construction for tropical Hopf manifolds.
Denote by $\nT = \nR \cup \{ \infty\}$ the tropical semi-field with min-plus conventions. 
To simplify notation, $\infty$ is also used to denote $(\infty, \dots, \infty) \in \nT^n$.

The first example of a  tropical Hopf manifold, is the  quotient of $\nT^d\setminus\{\infty\}$ by the translation $\tau_a: (\nT^d, \infty) \to (\nT^d, \infty)$ by a vector $a$, with strictly positive entries. Then for a neighborhood $U$ of $\infty$ we have $\tau_a(U)$ is strictly contained in $U$.

This case will be referred to as \emph{diagonal} tropical Hopf manifolds.
Monomial Hopf manifolds are quotients of $\nT^d \setminus \{\infty\}$ by a \emph{contracting automorphism} $F: \nT^d \to \nT^d$. A contracting automorphism $\tF: \nT^n \to \nT^n$ is the extension of an  integer affine map $\nR^n \to \nR^n$ which is a permutation of coordinates composed with a translation, such that $F(U) \subset U$ for any neighborhood $U$ of $\infty$.

More generally a  tropical Hopf data is a pair $(X, F)$ where $X \subseteq \nT^n$ is a neighborhood of $\infty$ of a  tropical subvariety 
obtained from $\nT^d$ via a  tropical modification, (we call $X$ a germ of a tropical subvariety)
and $F: \nT^n \to \nT^n$ is a contracting automorphism of $\nT^n$ such that $F(X) \subset  X$. 
A tropical Hopf data $(\tX, \tF)$ produces a tropical Hopf manifold $S(\tX, \tF)$ by considering the quotient of $\tX \setminus \{\infty\}$ by the action of $\tF$. 

Call a subset $\tV \subseteq \nT^n$ a \emph{cone} if there is a vector with strictly positive entries $w \in (\nR^*_{+})^n$ such that if $p \in X$  then $p + \lambda w \in \tX$ for any $\lambda \in \nT$. Then it is possible to characterize tropical Hopf data as follows.

\begin{introthm}\label{thm:introgeneralHopf}
If $(\tX, \tF)$ is a tropical Hopf data then $X$ is a neighborhood of a cone tropical variety $\tV$. 
In particular, the map $\tF$ is a global automorphism of the tropical variety $\tV$. 

Also, the tropical Hopf manifold obtained from $(\tX, \tF)$  is the finite quotient of a tropical modification of a diagonal tropical Hopf manifold.
\end{introthm}

From the dynamical point of view, we are also interested in the relation between tropical Hopf manifolds and contracting automorphisms defined over a field $\nK$ with non-archimedean valuation.
Assume we have a contracting automorphism $\af:(\nA_\nK^d,0) \to (\nA_\nK^d,0)$, and denote by $\tf:(\nT^d, \infty) \to (\nT^d,\infty)$ its tropicalization. We shall call such a tropical map a \emph{weakly non-degenerate contracting} germ.
In general,   $\tf$ 
need  not be an automorphism of $\nT^d$. 
In this case, we may  ask if there is a tropical  modification $\pi:\tX \to \nT^d$ so that $\tf$ \emph{lifts} to a map $\tF:\tX \to \tX$ which is an automorphism. 

In general we see that this is not possible, at least with only a finite number of modifications, see Example \ref{ex:noncompactHopf}.
Nonetheless, we show that it is always possible to lift a map in a non-dynamical way.
In other words, we show that there is a (canonical) way to represent any (weakly non-degenerate) tropical germ $\tf:(\nT^d, \infty) \to (\nT^d, \infty)$ as a morphism  $\tF: \tX \to \tY$ where $\tX, \tY$ are  modifications of $\nT^d$ and $\tF$ is an isomorphism in a neighborhood of $\infty$.
See Definition \ref{def:modification} for modifications, strict transforms and the exceptional locus.

\begin{introthm}\label{thm:intromonomialization}
Let $\tf:(\nT^d,\infty)\to (\nT^d,\infty)$ be a tropical
weakly non-degenerate germ.
Then there exists two regular tropical modifications $\pi:\tX \to (\nT^d,\infty)$ and $\eta:\tY \to (\nT^d,\infty)$, and a tropical isomorphism $\tF:\tX \to \tY$ such that $\tf \circ \pi = \eta \circ \tF$ on the strict transform $\on{St}(\pi)$ in a suitable neighborhood of $\infty$.
\end{introthm}

If the two modifications $\pi$ and $\eta$ coincide, and $\tF$ is a contracting automorphism, then $(\tX,\tF)$ is a (virtually regular) Hopf manifold (see Example \ref{ex:weight2monom}).
However, in general it is not possible to choose modifications $\pi$ and $\eta$ so that they coincide. In this case we obtain the following result. 

\begin{introthm}\label{thm:intrononcompact}
Let $\tf:\nT^d \to \nT^d$ be a tropical weakly non-degenerate germ at $\infty$ obtained as the tropicalization of a global automorphism $\af:\nA^d \to \nA^d$.

Then there exists a (possibly infinite) sequence $(\xi^{(j)})_{j \geq 0}$ of regular tropical modifications $\xi^{(j)}:W^{(j+1)} \to W^{(j)}$, with $W^{(0)}=\nT^d$, satisfying the following properties.
Denote by $\mu^{(n)}=\xi^{(0)} \circ \cdots \circ \xi^{(n-1)}:W^{(n)} \to \nT^d$ the composition of the first $n$ modifications, and by $\mu^{(\infty)}:W^{(\infty)} \to \nT^d$ the inverse limit of $(\mu^{(n)})$.
There exists an automorphism $\thh^{(\infty)}: \tW^{(\infty)} \to \tW^{(\infty)}$ such that $\tf \circ \mu^{(\infty)} = \mu^{(\infty)} \circ \thh^{(\infty)}$ on the inverse limit $\displaystyle \lim_{\substack{\longleftarrow\\ n}} \on{St}\big(\mu^{(n)}\big)$
of the strict transforms of $\nT^d$ by $\mu^{(n)}$. 
\end{introthm}

If $\tf$ is contracting and  the sequence of modifications is finite, then  the data given by Theorem \ref{thm:intrononcompact} gives rise to a tropical Hopf manifold. If it is not the case, we can still consider the quotient of (a pointed neighborhood of $\infty$ of) $W^{(\infty)}$ by the action of $\thh^{(\infty)}$, obtaining what we call a \emph{non-compact tropical Hopf manifold}.

In dimension $2$, we give a classification of contracting germs $\tf:\nT^2 \to \nT^2$ for which the monomialization procedure from Theorem \ref{thm:intromonomialization} produces a tropical Hopf data of dimension $2$. 

\smallskip

As mentioned above, Hopf surfaces have been considered also over non-archimedean fields.
In order to make a connection between these surfaces and the tropical surfaces, we begin by looking for \emph{realizations} of the Hopf data $(\tX, \tF)$.
Let $\nK$ be an algebraically closed field, and $\nu_0$ a non-trivial valuation on $\nK$.
As an example, consider $\nK=\nC\langle t \rangle$ the field of Puiseux series over the complex numbers,
and $\nu_0$ the $t$-adic valuation.
We denote  by $\nA^n$ the affine space of dimension $n$ over the field $\nK$.

A realization of $(\tX, \tF)$ is a pair $(\aX, \aF)$ such that 
\begin{enumerate}[(a)]
\item $\aX$ is the image of an embedding $i:(\nA^d, 0) \to (\nA^n, 0)$;
\item $\aF:(\nA^n,0) \to (\nA^n,0)$ is a contracting automorphism satisfying $\aF(\aX) = \aX$;
\item  $\trp(\aX) = \tX$ and $\trp(\aF) = \tF$;
\end{enumerate}
in this case we say that $(\tX, \tF)$ is \emph{realizable} by $(\aX, \aF)$, see Definition \ref{def:realizeHopf}. 
It turns out that realizations of compact tropical Hopf manifolds do not exist except in very special cases.
For example, we prove the following in dimension $2$. 

\begin{introthm}\label{thm:introrealHopf}
Let $(\tX, \tF)$ be data for a tropical Hopf surface  where $\tX \subset \nT^n$ is not homeomorphic to $\nT^2$ and $\tF: \nT^n \to \nT^n$ is a contracting automorphism.
Suppose $(\tX, \tF)$ is realizable by a Hopf surface defined over $\nK$ by a germ $\af$ defined over $\nK$. Then $\af$ is conjugate to $(\ax, \ay) \mapsto (\alpha \ax, \beta \ay)$ where $\nu_0(\alpha), \nu_0(\beta) > 0$ and $\alpha^k = \beta^l$ for some $k, l \in \nN^*$. 
\end{introthm}

A topological classification of the compact $2$-dimensional tropical Hopf data which are realizable
is given in Theorem \ref{thm:listfinitetrophopf}.

In dimension $2$, we can also go in the other direction, meaning from any analytic Hopf data to a tropical one. 
Starting from a germ $\af: (\nA^d, 0) \to (\nA^d, 0)$ we consider a sequence of embeddings $i_k: (\nA^d, 0) \to (\nA^{n_k}, 0)$.
In general, it may be that none of these embeddings are invariant under a monomial map on $\nA^{n_k}$ which induces $\af$. 
The next theorem gives a way to go from the analytic Hopf surfaces to tropical ones even in the non-compact case.

\begin{introthm}\label{thm:introantotrop}
Let $\nK$ be a non-trivially valued algebraically closed field, and $\af:(\nA_\nK^2,0) \to (\nA_\nK^2,0)$ a germ of a contracting automorphism.
Denote by $\aX(\af)$ the analytic Hopf surface associated to $\af$.
Then $\aX(\af)$  retracts to a tropical Hopf surface, which is non-compact if $\af$ is a resonant germ.
\end{introthm}

For any embedding $i: \nA^d \to \nA^n$ one obtains a map $\rho_i: \nA^d_{an} \to \trp(i(\nA^d))$, following \cite{payne:analitlimittrop}, where $\nA^d_{an}$ denotes the analytification of affine space of dimension $d$ over $\nK$ in the sense of Berkovich.
Under suitable conditions on the tropicalization, this map admits a canonical section from a subset $\tU \subset \trp(i(\nA^d))$ back to $\nA^d_{an}$ (see  \cite{gubler-rabinoff-werner:tropicalskeletons}). 
To prove the above theorem, we first construct a possibly non-compact tropical Hopf surface as a limit of the tropicalizations of an infinite sequence of embeddings.
Then we combine the two results above and take the dynamics into account.

Lastly, we consider both the tropical Picard   and $(p, q)$-homology groups  of diagonal tropical Hopf manifolds. 
These invariants are analogous to the classical situation, in the sense described below. 

\begin{introthm}
Let $S$ be a diagonal tropical Hopf manifold, then $\on{Pic}(S) \cong \nT^*$. Moreover, every tropical line bundle on $S$ has a section. 
\end{introthm}

However, it turns out every line bundle on a  tropical Hopf manifold always has a section. This is not the case for classical Hopf manifolds \cite[Theorem 3.1]{dabrowski:modulihopfsurfaces}

Upon computing the  $(p, q)$-homology groups of diagonal tropical Hopf manifolds we see that their ranks are the same as the ranks of the Dolbeault cohomology of a Hopf manifold, \cite[Theorem 4(c)]{ise:geomhopfmfld}. Therefore, we obtain a simple counter-example to Conjecture 5.3 from \cite{mikhalkin-zharkov:eigenwave} which says that $H^{p, q}(X) \cong H^{q, p}(X)$ for a tropical manifold $X$.

\medskip

The paper is organized as follows.
In Section \ref{sec:basics}, we give some basics on tropical geometry.
Section \ref{sec:tropHopf} deals with the construction of tropical Hopf manifolds in the compact case, proving Theorem \ref{thm:introgeneralHopf}.
In Section \ref{sec:Hopfandgerms} we provide a connection  between tropical germs and tropical Hopf varieties, and prove the monomialization results Theorem \ref{thm:intromonomialization} and Theorem \ref{thm:intrononcompact}. We specialize the situation at dimension $2$ giving an explicit classification of tropical germs which can be lifted to automorphisms.
In Section \ref{sec:analytification} we study the link between tropical and analytic Hopf varieties, proving Theorem \ref{thm:introrealHopf} and Theorem \ref{thm:sectionTrop}.
Finally in Section \ref{sec:invariants} we compute some geometrical invariants for tropical Hopf manifolds.

\begin{ack}
The authors would like to thank Joe Rabinoff  and Bernard Teissier for fruitful discussions on the links between tropical geometry and valuation spaces.

The first author is supported by the ERC-starting grant project ``Nonarcomp'' no.307856.
The second author is supported by the Alexander von Humboldt Foundation. 
This work begun while the second author was visiting the \'Ecole Polytechnique, Paris. 
\end{ack}

\section{Basics}\label{sec:basics}

\subsection{Tropical algebra}

\mbox{}

The tropical semi-field is $\nT =  \nR \cup \{ \infty \}$ equipped with two operations: 
addition and minimum, which is denoted $x \wedge y := \min\{x, y\}$. 
This is isomorphic to the max-plus algebra $\nR \cup \{-\infty\}$ 
via the map $x \mapsto -x$. 
  Here we use the minimum convention because we work with fields with a non-archimedean valuation. 

A tropical polynomial function
$P: \nT^n \rightarrow \nT$ 
is defined using the above operations,
$$
P(\tx_1, \dots, \tx_n) =
 \bigwedge_{I \in \Delta} \big(a_I + \scalprod{I,\tx}\big),
$$
where $\Delta$ is a finite subset of $\nN^n$, and $\scalprod{, }$ denotes the standard inner product on $\nR^n$ extended to $\nT^n$ with the usual conventions. These are concave piecewise integer affine functions.

Let $\nK$ be a field equipped with a non-archimedean valuation $\nu_0$.
In the rest of the paper, we will consider the following examples.
\begin{itemize}
\item $\nK=\nC((t))$ the field of Laurent series over $\nC$, endowed with the $t$-adic valuation $\nu_0$ defined by $\nu_0(t)=1$.
\item Its algebraic closure $\nK=\nC\langle t \rangle$ the field of Puiseux series over $\nC$.
\item The completion of $\nC\langle t \rangle$ with respect to the $t$-adic valuation (which turns out to be also algebraically closed).
\item Let $\nk$ be any algebraically closed field endowed with the trivial valuation. We consider $\nK=\nk((t^\nR))$ (resp., $\nK=\nk((t^\nQ))$) the field of transfinite series over $\nk$ with real (resp., rational) exponents (sometimes also called \emph{Hahn series}).
They are generalized series of the form $\sum_{k \in \nR} a_k t^k$, where $a_k \in \nk$, and the set $\{k \in \nR\ |\ a_k \neq 0\}$ is well ordered (see \cite[Section 2]{teissier:amibesnonarch} and references therein).
\end{itemize}

Notice that, up to taking a field extension, we may always assume that the value group of $\nu_0$ is $\nR$.

Given 
$\aP \in \nK[\ax_1, \dots, \ax_n]$, we define its \emph{tropicalization} $\tP = \trp(\aP) \in \nT[\tx_1, \ldots, \tx_n]$ as follows :
$$
\text{if } \qquad \aP(\ax) = \sum_{I \in \Delta} \aalpha_I \ax^I, \qquad \text{then } \qquad \tP(\tx) 
 =  \bigwedge_{I \in \Delta} \big(\nu_0(\aalpha_I) + \scalprod{I,\tx}\big).
$$

\subsection{Tropical varieties}\label{subsection:tropvar}

The space $\nT=(-\infty, +\infty]$ is equipped with the topology whose open sets are generated by the euclidean topology on $\nR$ and the sets of the form $(a,+\infty]$ with $a \in \nR$. 
The tropical affine space of dimension $n$ is denoted by $\nT^n$, and we use simply $\infty$ to denote the point $(\infty, \dots , \infty)$.
The space $\nT^n$ is stratified in the following way: a point $\tx \in \nT^n$ is of \emph{sedentarity} $I \subseteq \{1, \dots, n\}$ if $\tx_i = \infty$ if and only if $i \in I$.
The subset of points of $\nT^n$ of sedentarity $I$, is denoted $\nR^{n}_I$ 
and can be identified with $\nR^{n-\abs{I}}$. 
Denote the closure of $\nR^n_I$ in $\nT^n$ by $\nT^n_I$. 

Introductions to tropical varieties in $\nR^n$  can be found in \cite{maclagan-sturmfels:book} or \cite{mikhalkin:tropapp} and tropical varieties in $\nT^n$ in \cite{brugalle-itenberg-mikhalkin-shaw:gokova}.

\begin{defi}\label{def:tropsubvariety}
A tropical subvariety of $\nR^n$ is a pure dimensional finite rational polyhedral complex $\tV \subset \nR^n$ equipped with positive integer weights on its top dimensional facets and satisfying the balancing condition (see \cite[Section 3.4]{maclagan-sturmfels:book}). 
A tropical subvariety $V$ of $\nT^n$ of sedentarity $I$ is the closure in $\nT^n$ of a tropical subvariety $V^o$ of $\nR^{n}_I$.  
A  tropical subvariety of $ \nT^n$ is a union of tropical subvarieties of $\nT^n$  of possibly different sedentarities.  
\end{defi}

Here we will use tropical stable intersection with  linear spaces to define a notion of non-singular tropical subvariety in $\nT^n$ (see \cite{mikhalkin:tropapp}, \cite{maclagan-sturmfels:book}). 

Let $\tV \subset \nT^n$ be a tropical variety and $L$ a tropical linear space of complementary dimension and empty sedentarity. For a point $x$ of empty sedentarity let $w_x(\tV \cdot L)$ denote the weight of the point $x$ in the tropical stable intersection of $\tV$ and $L$.
Also define the total intersection number to be
$$
\abs{\tV \cdot L} := \sum_{x \in \nR^n} w_x(\tV \cdot L).
$$
In what follows, we say a linear space in  $\nA^n$ is generic if upon compactifying $\nA^n$ to $\nK \nP^n$ the closure of the linear space intersects all intersections of coordinate hyperplanes in the expected dimension. 
Similarly, a generic linear space through $0$ in $\nA^n$  if upon compactifying $\nA^n$ to $\nK \nP^n$ the closure of the linear space intersects all intersections of coordinate hyperplanes in the expected dimension, with the exception of $0 \in \nA^n \subset \nK \nP^n$.

\begin{defi}\label{def:nonsingular}
Let $V \subset \nT^n$  be a tropical variety of sedentarity $\emptyset$  of codimension $k$  and $x \in V$ a point of sedentarity $\emptyset$. Let $L_x \subset \nT^n$ be the tropicalization of a generic hyperplane in $\nA^n$ with vertex at $x$ and $L_x^k $ denote the $k$th self-intersection.  Define  the local degree
at $\tx$ to be
$$
\text{deg}_{\tV}(x)= \min \{w_x(V \cdot \phi(L_x^k)) \ | \ \phi \in GL_n(\nZ) \}.
$$ 

The local degree of $V$ at $\infty$ is 
$$\text{deg}_{\tV}(\infty)= |\tV \cdot L^k | - |V \cdot L_{\infty}^k|$$
where $L \subset \nT^n$ is the tropicalization of a generic hyperplane in $\nA^n$ 
and $L_{\infty} \subset \nT^n$ is the tropicalization of a generic hyperplane in $\nA^n$ passing through $0$. 

A tropical variety $\tV \subset \nT^n$ is non-singular at a point $\tx$ of sedentarity $\emptyset$ or $x = \infty$  if 	 
$\text{deg}_{\tV}(x) = 1$.
A tropical variety $\tV \subset \nT^n$  is non-singular if $\text{codim}_{\tV}(\tV \cap \nR^n_I) = |I| $ for all $I \subsetneq \{1, \dots, n\}$ and  it is also non-singular at $\infty$ and all points of empty sedentarity. 
\end{defi}

\begin{rmk}\label{rem:nonsingular}
The condition for $\tV \subset \nT^n$ to be tropically non-singular implies that the weights on its top dimensional faces are $1$, but this is far from being sufficient. 
In fact,  a  tropical variety $\tV \subset \nT^n$ is non-singular at a point $\tx$ of sedentarity $\emptyset$ if there exists a neighborhood $\tU$ of $\tx$ in $\tV$ which coincides with a neighborhood of a tropical linear space up to the action of an element of $\text{GL}_n(\nZ)$. 
 
\end{rmk}

\begin{ex}
It is not clear how to properly define the  notion of local degree at points of sedentarity $\emptyset \subsetneq I \subsetneq [n]$.  
One suggestion is to take the minimum of the intersection multiplicities with linear spaces intersecting the tropical variety at that point. But this definition is inadequate. 

Take for example $\tV \subset \nT^3$ to be the union of two affine planes $\tV_1$, $\tV_2$ defined by $\tx_1 = 0$ and $\tx_2 = \tx_3$ respectively. These are two (non-generic) tropical linear spaces, hence at any point $\tx \in \tV_i$ we ought to have $\text{deg}_{\tV_i}(x) = 1$. But then at the point $(0, \infty, \infty) \in V_1 \cap V_2$ the local degree $\text{deg}_{\tV}(x) = 2$, however the minimum positive weight of $x$ in the stable intersection of $\tV$ with a linear space can be seen  to be $1$.
\end{ex}

\subsection{Tropicalization of affine varieties}

Let $\nA^n$ be the 
$n$ dimensional 
affine space of a field $\nK$ equipped with a non-archimedean valuation $\nu_0$. Let $\nu_0 :\nA^n \to \nT^n$ also denote the coordinate-wise valuation map, with $\nu_0(0) = \infty$.
Given a subvariety $\aV \subset \aA^n$, 
the set $\overline{\nu_0(\aV)}$ can be equipped with the structure of a finite rational polyhedral complex of real dimension $\dim(\aV)$ and it can be equipped with 
weights on the top dimensional facets making it a tropical variety in the above sense, see \cite{maclagan-sturmfels:book} or Definition \ref{def:tropaffine} below. 

The weight of a top dimensional face is given by the \emph{multiplicity of the tropicalization} (Definition \ref{def:tropmult}) at any point $\tx$ in the relative interior of the top dimensional face.
This multiplicity can be described as follows. 
Let $\mc{I} \subset \nK[\ax]$ be the defining ideal of $\aV$ where $\ax = (\ax_1, \dots, \ax_n)$. 
For any point $\tx \in \overline{\nu_0(\aV)} \cap \nR^n$ one can associate a so-called \emph{initial ideal} of $\mc{I}$ at $\tx$, denoted $\on{In}_\tx(\mc{I})$, (see \cite[Chapter 3.4]{maclagan-sturmfels:book}). 
If $\tx$ has sedentarity $I$ then $\on{In}_\tx(\mc{I}) \subset \nK[\ax_j^{\pm} | j \nin I]$. In particular, if $\mc{I}$ is an ideal contained in $\langle \ax_i, i \in I \rangle$, then for any $\tx$ of sedentarity containing $I$, we have $\on{In}_\tx(\mc{I})=0$.

\begin{defi}\label{def:tropmult}
Let $\aV \subset \nA^n$ be a subvariety and let $\tx \in \overline{\nu_0(\aV)} \subset \nT^n$ be of sedentarity $I$. Suppose $\aV=\aV(\mc{I})$ is defined by some ideal $\mc{I}$.
The \emph{multiplicity of the tropicalization} at $\tx$, denoted $m_{\aV}(\tx)$, is the number of irreducible components of the special fiber of the variety in $(\nK^*)^{n-|I|}$ defined by $\on{In}_\tx(\mc{I}) \subset \nK[\ax_j^{\pm} | j \nin I]$ counted with multiplicity. 
\end{defi}

In particular, if the initial ideal at $x$ is zero then $m_{\aV}(x) =1$. 
We can finally define the tropicalization of an affine variety.

\begin{defi}\label{def:tropaffine}
Given a subvariety $\aV \subset \nA^n$, its \emph{tropicalization} $\trp(\aV) \subset \nT^n$, is the set $\overline{\nu_0(\aV)}$ with the weight of a top 
dimensional face $E$ given by  $m_{\aV}(x)$ for $x$ in the interior of $E$. 
\end{defi}
 
The weighted rational polyhedral complex $\trp(\aV)$ satisfies the balancing condition, therefore it is a tropical subvariety of $\nT^n$.   However, not every tropical variety is obtained in this way. Determining which ones are is known as the ``realizability'' or ``approximation'' problem, and has been 
studied in various cases such as,  tropical curves in $\nR^n$ \cite{speyer:curves}, tropical linear spaces \cite{speyer:troplinspace} and curves in tropical surfaces, see \cite{bogart-katz:curves}, \cite{brugalle-shaw:obstruct}.

For $\aV \subset \nA^n$, if $\trp(\aV)$ is non-singular as described in Definition \ref{def:nonsingular}, then $m_{\aV}(F) =1$ for each top dimensional face of $\trp(\aV)$. The next proposition states that the same is true for the multiplicity of faces of $\trp(\aV)$ of higher codimension as well. 

\begin{prop}\label{prop:multiplicities}
Let $\aV \subset \nA^n$ be a tropical subvariety and suppose that $\trp(\aV)$ is non-singular at $\tx$ a point of empty sedentarity, then $m_\aV(\tx) = 1$.  
If  $\aV \subset \nA^n$ is non-singular then $m_{\aV}(\tx) =1$ for all $\tx \in \tV$. 

\end{prop}

\begin{proof}
Let $\mc{I}$ be the defining ideal of $\aV$ and suppose that $\trp(\aV)$ is non-singular at a point $\tx \in \nT^n$ of sedentarity $\emptyset$.
Let $\aV_{\!\tx} \subset (\nK^*)^{n-\abs{I}}$ be the variety defined by the initial ideal $\on{In}_\tx(\mc{I}) \subset \nK[\ax_j^{\pm}.$ . 
Its tropicalization $\trp(\aV_{\!\tx})$ is a fan tropical variety and a  neighborhood of $\tx \in \trp(\aV)$ coincides, as a weighted fan, up to translation with a neighborhood of the origin of the fan $\trp(\aV_x)$. Since $\trp(\aV)$ is non-singular at $x$, by Remark \ref{rem:nonsingular}, it coincides with weights with a tropical linear space up to some integer affine map $\phi$. 
Therefore, the initial ideal $\on{In}_\tx(\mc{I}) $ is defined by degree one equations in   the  monomials  $x^{\alpha _1}, \dots,  x^{\alpha} _{n}$ for some $\alpha_k \in \nZ^n$. The integer vectors $\alpha_1, \dots,  \alpha_n$ together with  the standard basis vectors $e_i$   for $i \in I$ form a $\nZ^n$ basis. Up to the coordinate change given by $\phi$ this ideal is a linear ideal in $x_1, \dots , x_n$.  Therefore, the initial ideal at $x$ has exactly one component counted with multiplicity and  $m_{\aV}(\tx) = 1$. This proves the first part of the proposition.

For the second part,  by the assumption that $\tV$ is non-singular, it intersects $\nR^n_I$ properly for all $I \neq \{1, \dots, n\}$. In this case, if a point $\tx$ of sedentarity $I$ is contained in the relative interior of a face $A$ of $\tV$, then there is a unique face  $A_{\emptyset}$ of sedentarity $\emptyset$ of $\tV$ such that $A = A_{\emptyset} \cap \nT^n_I$. Let $\tx_{\emptyset}$ be a point in the relative interior of $A_{\emptyset}$. Then the initial ideal of $\tx$ is given by specializing the initial ideal of the  point $\tx_{\emptyset}$ by setting $\ax_i = 0$ for  all $i \in I$. Since by above $\on{In}_{\tx_{\emptyset}}(\mc{I})$ is a linear ideal, it follows that $\tx$ has multiplicity $1$ as well.
If $\tV$ contains  $\infty$ the multiplicity of the initial ideal at $\infty$ is $1$ by convention. 
Thus the proposition is proved. 
\end{proof}

Given an affine variety  $\aX$ over $\nK$, a tropicalization depends on the choice of embedding to $\nA^n$.
Because of this we sometimes refine our previous notation to make the embedding explicit.

\begin{defi}
Let $i:\aX \to \nA^n$ be an affine embedding.
The \emph{tropicalization of $\aX$ with respect to $i$} is the polyhedral complex
$$
\mc{T}_i(\aX):=\trp(i(\aX))=\overline{\nu_0(i(\aX))},
$$
equipped with positive integer weights $m_{i(\aX)}(F)$  on the top dimensional faces $F$ as in Definition \ref{def:tropaffine}.   
\end{defi}

\subsection{Tropical morphisms}\label{sec:tropmorph}
Morphisms between tropical varieties $V \subset \nT^n$ and $U \subset \nT^m$ are extensions of tropical monomial maps $\nR^n \to \nR^m$, in other words  integer affine functions. Such maps are tropicalizations of toric equivariant maps $(\nK^*)^n \to (\nK^*)^n$. 
If an integral affine map $\tF : \nR^n \to \nR^m$ extends to $\nT^n \to \nT^m$  means that $\tF$ preserves the sedentarity of the points in $\nT^n$ in the sense that $s(\tx) \leq s(\tf(\tx))$.
In particular, if $n = m$ and $f$ is invertible and linear then $f$ must be a permutation of the standard basis vectors $e_1, \dots , e_n$ composed with a translation.

Notice that  tropical polynomial maps are not in general  morphisms of tropical varieties. 
The image of a tropical variety under a tropical polynomial map may not satisfy the balancing condition. 
 Our proposed solution to this is a monomialization procedure (see Section \ref{sec:monomialization}), relying on tropical modifications. 

\subsection{Tropical modifications}
Tropical modifications were first introduced in \cite{mikhalkin:tropapp}, another introduction can be found in \cite{brugalle-itenberg-mikhalkin-shaw:gokova}. 
Given a tropical subvariety $\tV \subset \nT^n$ and a tropical polynomial function $\tP: \nT^n \to \nT$, a \emph{regular tropical modification} is a map $\pi_\tP : \wt{\tV} \to \tV$, where $\wt{\tV} \subset \nT^{n+1}$ is a tropical variety to be described below and $\pi_\tP$ is simply the linear projection of $\nT^{n+1} \to \nT^n$ with kernel generated by $e_{n+1}$.
To obtain the tropical variety $\wt{\tV}$, start with the graph  $\Graph{\tP}(\tV)$ of the function $\tP$ restricted to $\tV$. The graph inherits weights on its facets from $V$. However, in general, the graph is not a tropical variety; it may not be balanced in the $e_{n+1}$ direction since $\tP$ is only piecewise integer affine.
At every codimension~$1$ face of $\Graph{\tP}(\tV)$ which fails to satisfy the balancing condition, add a new top dimensional face containing the $e_{n+1}$ direction, equipped with the unique positive integer weight so that the balancing condition is now satisfied. 
The divisor of the function $\tP$ restricted to $\tV$, $\on{div}_{\tV}(\tP) \subset \tV$, is the image under $\pi_\tP$ of these newly added faces equipped with the weights they were assigned in $\wt{\tV}$. 

\begin{defi}\label{def:modification}
Given a tropical variety $\tV \subset \nT^n$ and a tropical polynomial function $\tP: \nT^n \to \nT$, the map $\pi_\tP: \wt{\tV} \to \tV$ where $\pi_\tP$ and $\wt{\tV}$ are described above is the \emph{regular modification of $\tV$ along the function $\tP$}.
The divisor $\on{div}_{\tV}(\tP) \subset \tV$ is a tropical variety, of codimension $1$ in $V$, whose support is the points of $\tV$ where the map $\pi_\tP: \wt{\tV} \to \tV$ is not $1-1$. The weight function on the facets of $\on{div}_{\tV}(\tP)$ is induced from the weights of $\wt{\tV}$.

We denote by $\Exc{\pi_\tP}=\pi_\tP^{-1}(\on{div}_{\tV}(\tP))$ the \emph{exceptional locus} of $\pi_\tP$, which is the locus where $\pi_\tP$ is not $1$-to-$1$; and by $\Strict{\pi_\tP}=\Graph{\tP}$ the \emph{strict transform} of $\pi_\tP$, which can be seen also as the closure of $\wt{\tV} \setminus \Exc{\pi_\tP}$ in $\nT^{n+1}$. 
\end{defi}

Given two divisors $\on{div}_{\tV}(\tP_1)$ and $\on{div}_{\tV}(\tP_2)$, then $\on{div}_{\tV}(\tP_1) + \on{div}_{\tV}(\tP_2) = \on{div}_{\tV}(\tP_1+\tP_2)$. In terms of polyhedral complexes $\on{div}_{\tV}(\tP_1+\tP_2)$ is supported on the union of $\on{div}_{\tV}(\tP_1)$ and $\on{div}_{\tV}(\tP_2)$ with refinements and addition of weight functions where necessary.
 
Tropical division corresponds to usual subtraction, so a tropical rational function is the difference of two tropical polynomial functions, $\tP = \tP_1 - \tP_2$. The divisor of a rational function is a tropical cycle
defined as $\on{div}_V(\tP) = \on{div}_V(\tP_1) - \on{div}_V(\tP_2)$.  
As a cycle $\on{div}_V(\tP)$ may be effective even if $\tP_1 -\tP_2$ does not coincide with a tropical polynomial. In this case,  we can again define the modification of $V$ along the rational function $\tP$, see \cite[Example 2.28]{shaw:tropicalintersectionproduct}. These will be referred to as rational tropical modifications.   

Notice that $\pi_\tP$ is $1-1$  outside of the exceptional locus $\Exc{\pi_\tP}$, and the divisor $\on{div}_V(\tP)$ determines the polyhedral complex $\wt{\tV} \subset \nT^{n+1}$ up to translation in the $e_{n+1}$ direction.
Therefore it is common to say that we consider the modification along the divisor $\on{div}_V(\tP)$.
In general, we refer to a modification as any composition of regular or rational modifications, and to a regular modification as the composition of modifications along tropical polynomial functions. 

A modification of a tropical variety should  be thought of as taking the graph of a function restricted to a tropical variety. In the tropical world this operation changes the topology of the space, although the modification is still related to the original space via a retraction map. A concrete relation between tropical modifications along regular functions and graphs of functions restricted to varieties over $\nK$ is given by the next proposition. 

\begin{prop}\label{lem:realGraph}
Let $V=\trp(\aV) \subset \nT^n$ be the tropicalization of a variety $\aV \subset \nA^n$ and $\tP: \nT^n \to \nT$ be a tropical polynomial.
Denote by $\pi: \wt{V} \to V$ the regular modification along the function $\tP|_V$. Then for a generic choice of polynomial $\aP \in \nK[\ax_1, \dots , \ax_n]$  such that $\trp(\aP) = \tP$ we have that $\trp(\Graph{\aP|_\aV}) = \wt{\tV}$, where $\Graph{\aP|_\aV} \subset \nA^{n+1}$ denotes the graph of the restriction of $\aP$ to $\aV$. 
\end{prop}

\begin{proof}
Consider the cylinder $\aV'$ over $\aV$ in $\nA^{n+1}$, in other words $\aV'$  it is defined by the same ideal as $\aV$ but considered over $\nK[\ax_1, \dots, \ax_{n+1}]$ instead of the polynomial ring in only $n$ variables.
Now take a function $\aP \in \nK[\ax_1, \dots , \ax_n]$  tropicalizing to $\tP$, and consider the hypersurface $\hyp{\ax_{n+1}  - \aP} \subset \nA^{n+1}$.
It is clear that the graph of $\aP|_\aV$ satisfies $\Graph{\aP|_\aV} = \aV' \cap \hyp{\ax_{n+1}  - \aP}$.
Similarly, let $\pi: \wt{\tV} \to \tV$ be the regular tropical modification along the function $\tP$, then $\wt{\tV} = \trp(\aV') \cdot  \trp(\hyp{\tx_{n+1}  - \tP})$ where ``$\cdot$" denotes again the tropical stable intersection.  By \cite[Theorem 5.3.3]{osserman-payne:tropint} the tropical stable intersection $\trp(\aV) \cdot  \trp(\aU)$ is equal to $\trp(\aV \cap \aU^{\prime})$ for $\aU^{\prime}$ a generic toric translate of $\aU$.
For hypersurfaces $\hyp{\aP}$ this boils down to making a change of coefficients of $\aP$ so that we still obtain the same tropical hypersurface.
Thus the cycle $\wt{\tV}$ which is the result of the modification $\pi$ is the tropicalization of the graph $\Graph{\aP|_\aV}$ for a generic $\aP$ tropicalizing to $\tP$.
\end{proof}
 
By repeatedly applying the above proposition we have the following:

\begin{cor}\label{cor:realMultiGraph}
For a generic choice of  functions $\aP_1, \dots , \aP_r \in \nK[\ax_1, \dots , \ax_n]$  tropicalizing to tropical polynomial functions $\tP_1, \dots , \tP_r : \nT^n \to \nT$ respectively, let $\Graph{\aP|_\nA^n} \subset \nA^{n+r}$ denote the graph of $\nA^n$ along the $r$ functions $\aP_k$, $k=1, \ldots, r$. Then $\trp(\Graph{\aP|_{\nA^n}}) = \wt{\tV}$ where $\pi: \wt{\tV} \to \nT^n$ is the composition of the regular modifications along the functions $\tP_1, \dots , \tP_r$. 
\end{cor}

The above corollary will be used when we consider the monomialization of tropical germs $\nT^d \to \nT^d$ in Section \ref{sec:monomialization}.

\subsection{Abstract tropical spaces}

Tropical Hopf manifolds are not affine. To make sense of these, we need abstract tropical spaces, first introduced in \cite{mikhalkin:tropapp}.

\begin{defi}\label{def:abstractmfld}
A \emph{$d$-dimensional tropical space} $\tX$ is a Hausdorff topological space equipped with an atlas of charts $\{(U_\alpha , \phi_\alpha)\}$,
with $\phi_\alpha:\tU_\alpha \to V_{\alpha} \subseteq \nT^{N_\alpha}$ open embeddings, such that the following conditions hold.
\begin{enumerate}[(a)]
\item For every $\alpha$, the set $\tV_\alpha$ is a $d$-dimensional tropical variety in $\nT^{N_\alpha}$;
\item For any $\alpha, \beta$, the overlapping map $\psi_\alpha^\beta = \phi_\beta \circ \phi_\alpha^{-1}:\phi_\alpha(\tU_\alpha \cap \tU_\beta)\to \phi_\beta(\tU_\alpha \cap \tU_\beta)$ is the restriction of an integer affine linear map $\nT^{N_\alpha} \to \nT^{N_\beta}$ (i.e., a continuous extension of an integer affine linear map $\nR^{N_\alpha} \to \nR^{N_\beta}$). Moreover, the weights of $V_{\alpha}$ and $V_{\beta}$ must agree on the overlaps;
\item $X$ is of \emph{finite type}, i.e., there exists a \emph{finite} covering $(W_k)$ of $X$, subordinate to the covering $(U_\alpha)$ (i.e., for any $k$ there exists $\alpha$ such that $W_k \subseteq U_\alpha$), so that $\overline{\phi_\alpha(W_k)} \subset \phi_\alpha(U_\alpha) \subset \nT^{N_\alpha}$. 
\end{enumerate}

The tropical space $\tX$ is \emph{non-singular}, and sometimes called a \emph{tropical manifold}, if the $d$-dimensional tropical models $\tV_\alpha$ in condition (a) are non-singular.
\end{defi}

\begin{rmk}
In Section \ref{ssec:noncompact}, we consider ``non-compact'' Hopf manifolds, which are tropical manifolds in a weaker sense.
In particular, they admit atlases that do not satisfy the finite-type condition (c) of Definition \ref{def:abstractmfld}, but a \emph{local finite-type} condition, where the covering $(W_k)$ is only required to be locally finite. 

\end{rmk}

\section{Tropical Hopf manifolds}\label{sec:tropHopf}

This section defines tropical analogues of the classical construction of Hopf manifolds, introduced in \cite{hopf:hopfmfld}. See also \cite[Section V.18]{barth-hulek-peters-vanderven:compactcomplexsurfaces} for Hopf surfaces, and \cite{ise:geomhopfmfld,kato:subvarietieshopf} for Hopf manifolds in higher dimensions.

\subsection{Monomial case}\label{ssec:monoHopf}

As a first example,  consider the following construction.
Let $\tF:\nT^d \to \nT^d$ be a linear (invertible) monomial map, i.e., a map of the form
$$
\tF(\tx_1, \ldots, \tx_d)=(\tx_{\sigma(1)}+a_{\sigma(1)}, \ldots, \tx_{\sigma(d)}+a_{\sigma(d)}),
$$
where $\sigma$ is a permutation of $\{1, \ldots, d\}$ and $a_j \in \nR$.
Denote by $s$ the order of $\sigma$.

Assume that $\tF$ is contracting, meaning that for any neighborhood $U$ of $\infty \in \nT^d$ there exists a neighborhood $V \subseteq U$ such that $\tF(V) \subrelcpct V$ (i.e., $\tf(V)$ is relatively compact inside $V$, or equivalently $\overline{\tF(V)} \subset V$).
In this case, $\tF$ is contracting if and only if
$$
a_k^{s}:=\sum_{h=1}^s a_{\sigma^h(k)} > 0
$$
for all $k=1, \ldots, d$. One can check that the $s$-th iterate of $\tF$ has the form
$$
\tF^s(\tx_1, \ldots, \tx_d)=(\tx_1 + a_1^{s}, \ldots, \tx_d + a_d^{s}).
$$

Set $\nK=\nk((t^\nR))$ the field of transfinite series with real exponents, endowed with the $t$-adic valuation $\nu_0$ and norm $| \cdot |_0 = e^{-\nu_0(\cdot)}$.  Let  $\nA^d=\nA_\nK^d$ be  the affine space over $\nK$.
Notice that a map $\tF$ as above is the tropicalization of a linear monomial  map $\aF: \nA^d \to \nA^d$ given by
$$\aF(\ax_1, \dots , \ax_d) = (\alpha_{\sigma(1)}\ax_{\sigma(1)}, \dots , \alpha_{\sigma(d)}\ax_{\sigma(d)}),
$$
where $\nu_0(\alpha_i) = a_i$. 
The map $\aF$ is diagonalizable and moreover, $\abs{\alpha_k^{s}}_0=\abs{\prod_{h=1}^s \alpha_{\sigma^h(k)}}_0 < 1$ implies that
the eigenvalues of $\aF$ all satisfy $\abs{\lambda}_0 < 1$, which is the usual condition to be a contracting map.

We now construct a fundamental domain for the action of $F$ on $\nT^d$. 
Let $L \subset \nT^d$ be the closure in $\nT^d$ of a $\tF$-invariant affine line in $\nR^d$. 
Such a line is easily found by tropicalizing 
an eigenvector of $\aF$. 
Namely, the line $L$ is of the form,
$$
L=\left\{\frac{\tx_1+\beta_1}{a_1^{s}} = \cdots = \frac{\tx_d + \beta_d}{a_d^{s}}\right\},
$$
for suitable $\beta_1, \ldots, \beta_d \in \nR$.
Set
\begin{equation*}
U_k = \left\{\frac{\tx_k + \beta_k}{a_k^{s}} \leq \frac{\tx_j + \beta_j}{a_j^{s}} \text{ for all } j \neq k\right\} \subset \nT^d.
\end{equation*}
Notice that $\bigcup_k U_k = \nT^d$, $\bigcap_k U_k = L$, and $F(U_k) = U_{\sigma^{-1}(k)}$.

We now pick values $b_k^- < b_k^+ \in \nR$ satisfying $b^-=(b_1^-, \ldots, b_d^-) \in L$ (for example $b_k^-=-\beta_k$), and $b_k^+ = b_{\sigma(k)}^- + a_{\sigma(k)}$, so that $b^+=(b_1^+, \ldots, b_d^+) = \tF(b^-) \in L$.
Set
\begin{align*}
W_k &= \{(\tx_1, \ldots, \tx_d) \in U_k\ |\ b_k^- \leq \tx_k \leq b_k^+\},\\
W_k^+ &= \{(\tx_1, \ldots, \tx_d) \in U_k\ |\ \tx_k = b_k^+\},\\
W_k^- &= \{(\tx_1, \ldots, \tx_d) \in U_k\ |\ \tx_k = b_k^-\}.
\end{align*}
It is easy to check that $\tF(W_k^+)=W_{\sigma^{-1}(k)}^-$, and $W:=\bigcup_k W_k$ is a fundamental domain for $\tF$.

\begin{defi}
The couple $(\nT^d,\tF)$ is called a \emph{tropical (monomial) Hopf data}. The tropical  manifold $S(\nT^d,\tF)$ defined as $\nT^d \setminus \{\infty\}/\langle \tF \rangle$, or equivalently by $W/\langle \tF\rangle $, is called the \emph{tropical Hopf manifold} associated to the tropical Hopf data $(\nT^d,F)$.
If the map $\tF : \nT^d \to \nT^d$ is simply a translation, then we say $(\nT^d, \tF)$ is a tropical \emph{diagonal} Hopf data.
\end{defi}

\begin{ex}
Set $\tG(\tx_1,\tx_2) = (\tx_1+1, \tx_2+\sqrt{2})$. This translation gives rise to a tropical diagonal Hopf data $(\nT^2,\tG)$, depicted in Figure \ref{fig:diagonalhopf}.
The quotient of $\nT^2$ by this action is diffeomorphic to a compact cylinder. 

Consider now the map $\tF(\tx_1,\tx_2) = (\tx_2+3, \tx_1-1)$. It is obtained as the composition of a translation by the vector $(-1,3)$ followed by the map $(\tx_1,\tx_2) \mapsto (\tx_2,\tx_1)$. The tropical monomial Hopf data $(\nT^2,\tF)$ is depicted in Figure \ref{fig:nondiagonalhopf}. The quotient of $\nT^2$ by this action is diffeomorphic to a compact M\"obius band. 

Notice that if we iterate $\tF$ two times, we get $\tF^2(\tx_1,\tx_2)=(\tx_1+2, \tx_2+2)$, which is a contracting diagonal map.
It follows that the diagonal Hopf surface $S(\nT^2,\tF^2)$ is a $2$-to-$1$ covering of the Hopf surface $S(\nT^2,\tF)$, see Figure \ref{fig:coveringhopf}.
\end{ex}

\begin{figure}
\begin{subfigure}{.32\columnwidth}
\begin{minipage}[t]{\columnwidth}
	\def\svgwidth{\columnwidth}
	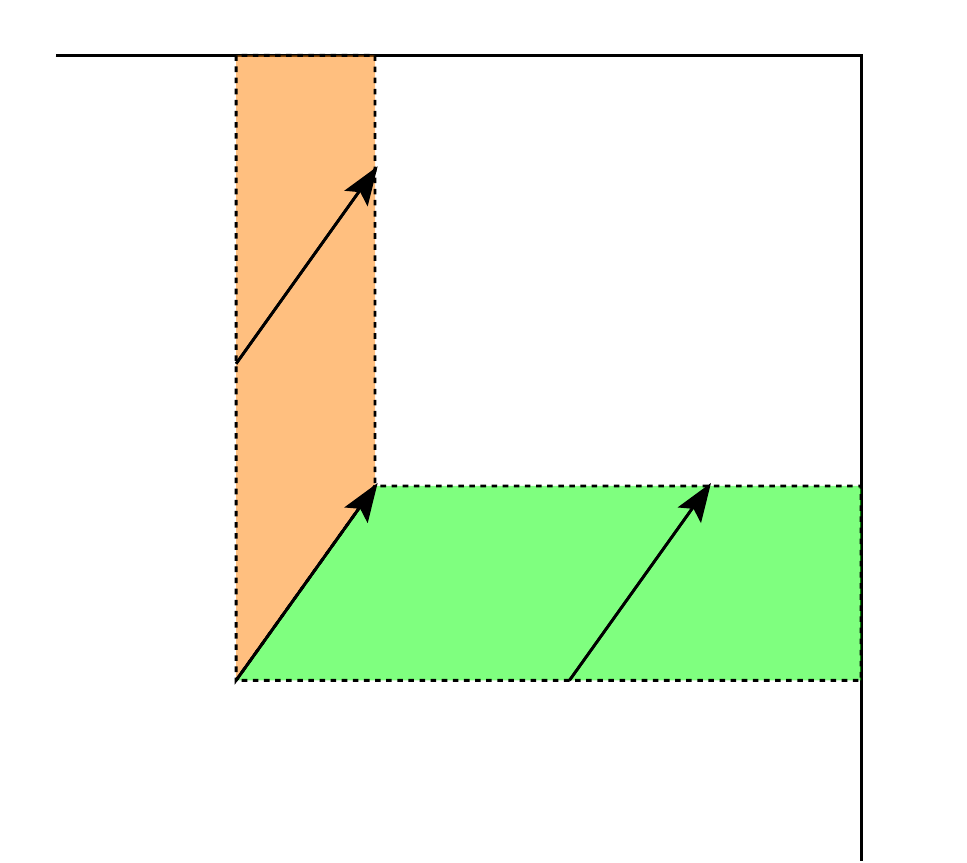
\end{minipage}
\caption{diagonal}
\label{fig:diagonalhopf}
\end{subfigure}
\begin{subfigure}{.32\columnwidth}
\begin{minipage}[t]{1\columnwidth}
	\def\svgwidth{\columnwidth}
	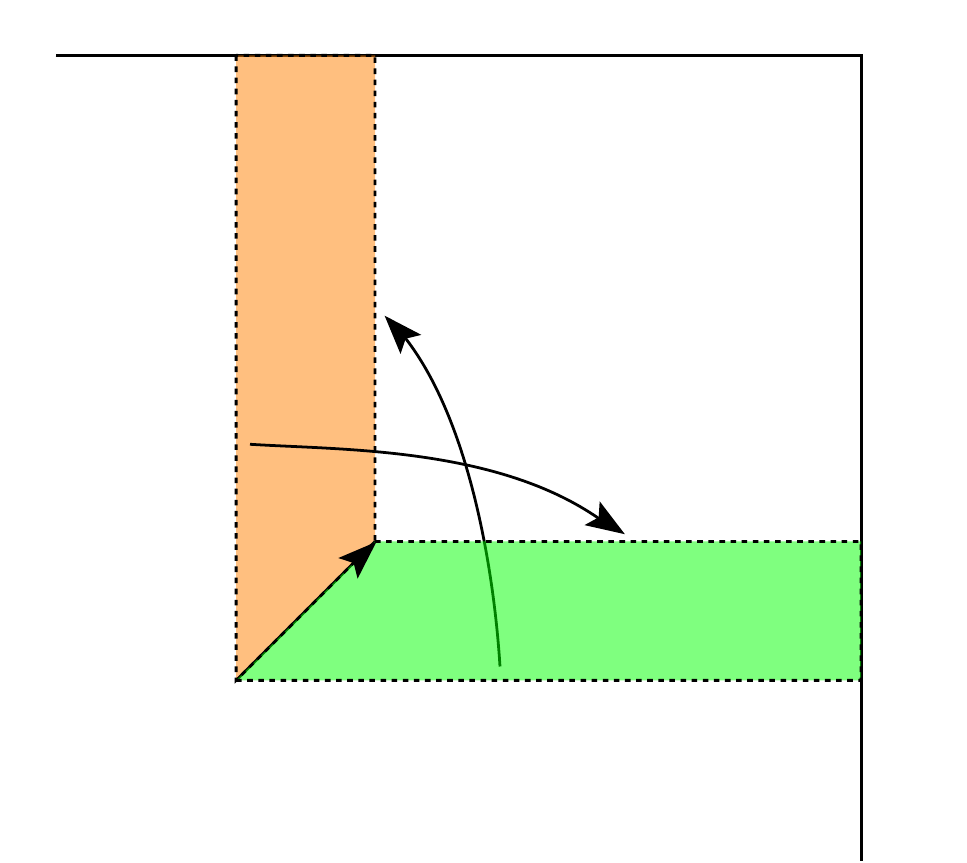
\end{minipage}
\caption{non-diagonal}
\label{fig:nondiagonalhopf}
\end{subfigure}
\begin{subfigure}{.32\columnwidth}
\begin{minipage}[t]{1\columnwidth}
	\def\svgwidth{\columnwidth}
	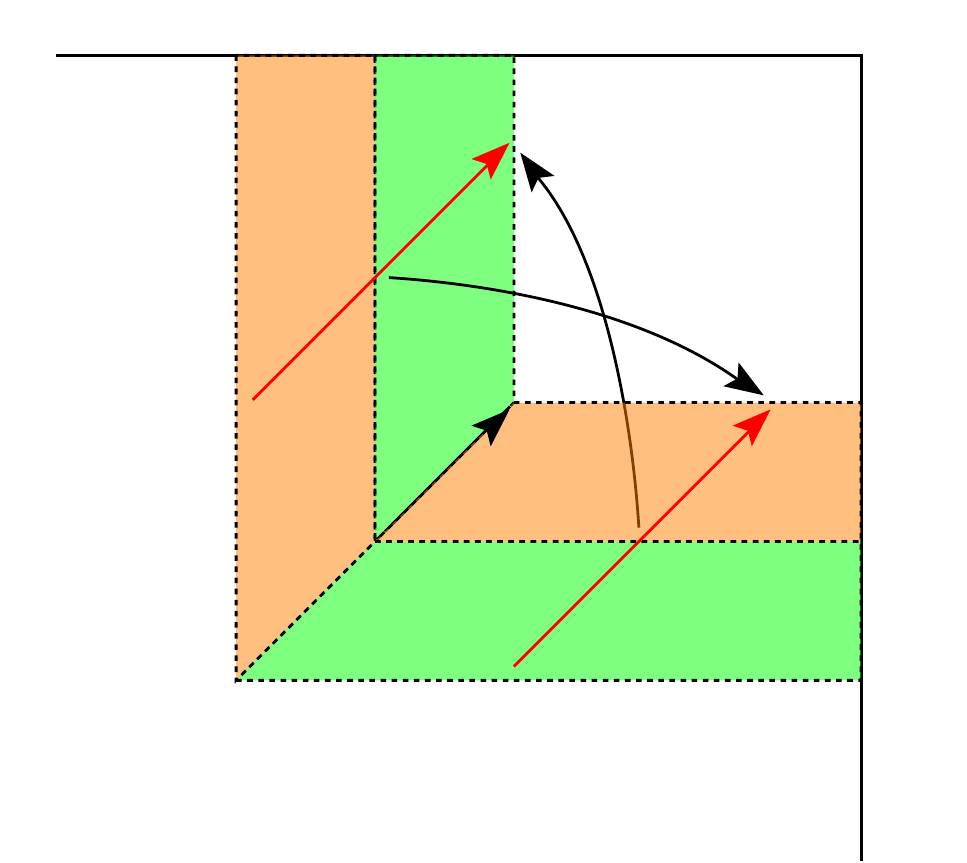
\end{minipage}
\caption{$2$-to-$1$ covering}
\label{fig:coveringhopf}
\end{subfigure}
\caption{Tropical monomial Hopf surfaces.}\label{fig:monomialhopf}
\end{figure}

\subsection{General case}

We now generalize the construction above, by starting with a \emph{non-singular germ of a  tropical subvariety} $X$ of 
$\nT^n$ 
which is invariant under the action of a monomial map $\tF: \nT^n \to \nT^n$.  
Recall the definition of non-singular tropical varieties from  Section \ref{subsection:tropvar}.
By a non-singular germ of a  tropical subvariety, we mean the intersection of a  tropical subvariety $V \subset \nT^n$ with a sufficiently small neighborhood of $\infty$ in which $V$ is non-singular. 

The tropical Hopf manifolds we construct are all submanifolds of tropical monomial Hopf manifolds. 
In  complex geometry, every Hopf manifold appears as a  submanifold of a primary Hopf manifold (see \cite[Proposition 1]{kato:subvarietieshopf}).

\begin{defi}\label{def:tropHopfdata}
A \emph{tropical Hopf data} is a couple $(\tX,\tF)$ where $\tX \subseteq \nT^n$ is 
a non-singular germ of a tropical variety obtained from a modification  $\pi: \tX \to \nT^d$   
and $\tF:(\nT^n,\infty) \to (\nT^n,\infty)$ is a linear monomial contracting map such that $\tF(\tX) \subset \tX$.
\end{defi}

\begin{defi}
Let $(\tX,\tF)$ be a tropical Hopf data, and $W$ the fundamental domain for $\tF:\nT^n \to \nT^n$ constructed as in Subsection \ref{ssec:monoHopf}.
Then $W \cap \tX$ is a fundamental domain for $\tF|_\tX$, and the \emph{tropical Hopf manifold} $S(\tX,\tF)$ associated to $(\tX, \tF)$ is   the quotient of $W \cap \tX$ by the action of $\tF|_\tX$.
\end{defi}

\begin{rmk}
In the classical setting, one can consider contracting germs $\af:(\aX,0) \to (\aX,0)$ at a normal isolated singularity $(\aX,0)$.
The existence of such a dynamical data gives restrictions on the geometry of the singularity.
Indeed, in dimension $2$ it has been proved in \cite{favre-ruggiero:normsurfsingcontrauto} that a normal surface singularity $(\aX,0)$ admitting a contracting automorphism is quasi-homogeneous.
Up to finite quotients, quasi-homogeneous surface singularities are obtained by contracting the zero-section of a negative degree line bundle over a curve $E$.
The induced ``singular Hopf manifold'' associated to such dynamical data can be a Hopf surface, a Kodaira surface, or of Kodaira dimension $1$ according to the genus of $E$ (see \cite{kato:cptcplxmanifoldsGSS}, \cite[Section V.5]{barth-hulek-peters-vanderven:compactcomplexsurfaces}).

The 
tropical constructions in this section can also be easily generalized 
to this singular setting.
If needed in the following, we shall refer to \emph{regular} tropical Hopf data as the ones as described by Definition \ref{def:tropHopfdata}, and to \emph{singular} tropical Hopf data as a couple $(\tX,\tF)$ as in Definition \ref{def:tropHopfdata}, but without the assumption of non-singularity of $\tX$.

Moreover, sometimes the tropicalization $\tX$ of a variety $\aX$ is not non-singular even if $\aX$ is non-singular.
The singularity of the tropicalization is given by a ``bad'' choice of the embedding defining it (see \cite{cueto-markwig:repairtropplanecurve} for techniques to avoid this problem for planar curves).
We refer to the case when a tropically singular $\tX$ in Definition \ref{def:tropHopfdata} comes from the tropicalization of an embedding as the \emph{virtually regular} case (see Example \ref{ex:weight2monom}).
\end{rmk}

\begin{ex}\label{ex:permutation}
Consider the non-singular tropical surface $\tX \subset \nT^n$ consisting of a single $1$-dimensional face which is the line $a(1, \dots, 1) \in \nT^n$, and $n$ faces of dimension $2$ equipped with weight $1$. These faces are defined by 
$$
A_k  := \{\tx_k \geq \tx_i = \tx_j\ |\ i \neq j \neq k \neq i\}$$ for $k = 1, \ldots, n$. The set $\tX$  is invariant under translations by $(1, \dots , 1)$.  

Consider the linear monomial map
$$
F(\tx_1, \ldots, \tx_n)=(\tx_{\sigma(1)}+a, \ldots, \tx_{\sigma(n)} + a),
$$
where $a \in (0, \infty)$ and $\sigma$ is a permutation of $\{1, \ldots, n\}$.
Then $\tF(\tX)=\tX$, moreover $\tF(A_k)=A_{\sigma^{-1}(k)}$, and $\tF$ is a contracting contracting linear monomial map. 

A fundamental domain for $\tF|_\tX$ is given by
$$
W=\bigcup_{k=1}^d W_k, \qquad W_k = \{\tx \in A_k\ |\ 0 \leq \tx_j \leq a \text{ for all } j \neq k\}.
$$
The tropical Hopf surface associated to $(X, F)$ is given by taking the quotient of $W$ (or of $\tX \setminus \{\infty\}$) by the action of $\tF$.

Notice that $\tF^s: \tX \to \tX$ is just a translation by $s(a, \dots, a)$. Therefore the tropical Hopf surface above is a quotient of $(\tX \setminus \{\infty\}) / \langle \tF^s\rangle$, which is a  tropical modification of a standard diagonal Hopf surface in the following sense.
\end{ex}

\begin{defi}\label{def:modData}
Let $(\tX, \tF)$ and $(\wt{\tX}, \wt{\tF})$ be tropical Hopf data such that $\wt{\tX} \subset \nT^{d+k}$, $\tX \subset \nT^d$ and  $\pi: \wt{\tX} \to \tX$ is a tropical modification.
If $\pi \circ \wt{\tF} = \tF \circ \pi$, then $(\wt{\tX}, \wt{\tF})$ is called a \emph{modification} of the Hopf data $(\tX, \tF)$.
\end{defi}

\begin{rmk}
If $(\wt{\tX}, \wt{\tF})$ is a modification of the Hopf data $(\tX, \tF)$ then $S(\wt{\tX},\wt{\tF})$  is a modification of the abstract tropical Hopf manifold $S(\tX, \tF)$ in the sense of \cite[Section 3.4]{mikhalkin:tropapp}.
\end{rmk} 

\begin{prop}\label{prop:finitequotient}
Any  tropical Hopf manifold is obtained as a finite quotient of a modification of a tropical diagonal Hopf manifold.
\end{prop}
\begin{proof}
Let $(\tX,\tF)$ be any tropical Hopf data.
By definition, $\tF= \sigma \circ \tau$, where $\tau$ is a translation and $\sigma$ is a permutation.
If $s$ is the order of $\sigma$, then $\tF^s$ is a translation.
Notice that the tropical Hopf manifold $S(\tX,\tF^s)$ covers $s$-to-$1$ the manifold $S(\tX,\tF)$.
Hence up to finite quotients we can assume that $\tF$ is simply a translation.
Recall there is a  modification $\pi:\tX \to \nT^d$, given by the definition of the tropical Hopf data.
Then by Definition \ref{def:modData}, the Hopf data $(\tX, \tF^s)$ is a modification of $(\nT^d, \tf)$ where $\tf$ is a translation in 
$\nT^d$.  
\end{proof}

\subsection{Explicit construction}

In this section we give an explicit construction of (possibly singular) tropical Hopf data $(\tX,\tF)$.
In particular, we describe the geometry of the total spaces $\tX$ that arise in the Hopf data. 

\begin{defi}\label{def:conealongw}
Let $\tV \subset \nT^n$ be a (possibly singular) tropical  variety of dimension $d$. 
Let $w=(w_1, \ldots, w_n) \in (\nR_+^*)^{n}$ be a vector of strictly positive entries.
We say that $\tV$ is a \emph{cone} along $w$ if
$$
p \in \tV \quad \Rightarrow \quad p + \lambda w \in \tV \quad \forall \lambda \in \nT.
$$
\end{defi}

\begin{rmk}
Firstly, notice that if $\tV$ is a cone along $w$, it will be a cone along $kw$ for any $k > 0$.
Secondly, notice that a tropical  variety $\tV$ could be a cone along different directions.
A trivial example is given by $V = \nT^n$,  then $\nT^n$ is a cone along any direction $w$.

The non-uniqueness of the cone directions 
implies that $V$ contains a linearity space of dimension $\geq 2$. However when the linearity space is $1$-dimensional then by the rationality of $V$,  we can assume up to taking a scalar multiple of $w$, that it is a primitive vector of strictly positive integer entries.
\end{rmk}

Notice that if $\tV$ is a cone along the weight $w$, then the map $\tF(\tx)=\tx+\lambda w$ with $\lambda > 0$ defines a contracting map which leaves $\tX$ invariant, so that $(V,\tF)$ is a (singular) Hopf data. The next proposition also implies that, up to replacing $\tF$ by an iterate, all Hopf data are of this form.
Together Propositions \ref{prop:Hopfdatacone} and \ref{prop:finitequotient} complete the proof of Theorem \ref{thm:introgeneralHopf}.

\begin{prop}\label{prop:Hopfdatacone}
Let $(\tX,\tF)$ be a (singular) Hopf data, then $\tX$ is a neighborhood of a cone tropical variety $\tV$.
\end{prop}
\begin{proof}
If $(\tX, \tF)$ is a  possibly singular tropical Hopf data then $\tX \subset \nT^{n}$ is a  neighborhood of $\infty$ of a $d$-dimensional tropical variety $V$ and $\tF=\sigma \circ \tau$, where $\sigma$ is a permutation of the coordinates $(\tx_1, \ldots, \tx_n)$ of $\nT^{n}$ and $\tau$ is a translation by a vector $a =  (a_1, \dots, a_n) \in \nR^{n}$.

Let $s$ be the order of $\sigma$, 
then $\tF^s$ is a translation by a vector $w=(w_1, \ldots, w_n)$, where $w_k = \sum_{h=1}^s \sigma^h(a_k)$.
Since $\tF$ is contracting, we infer $w_k > 0$ for all $k$.

It follows that for any $p \in \tX$, the point $p_h := p+hw \in \tX$ for any $h \in \nN$.
Since $\tX$ has a finite number of faces, it follows that there exists $m>0$ such that the points $\{p_{mh}\}$ belong to the same face $E$ of $\tX$. It follows that $p+\lambda w \in E \subset \tX$ for all $\lambda \in (0, \infty]$. Then  $\tX$ is a neighborhood of a  tropical variety $V$ which is a cone along $w$ given by allowing $\lambda \in \nT^n$. 
\end{proof}

It follows from Proposition \ref{prop:Hopfdatacone}, that if $\tX$ is the germ of a cone tropical variety $V$ and $(\tX, \tF)$ is a tropical Hopf data then the map $\tF$ is also a global automorphism of  $V$. 
Therefore throughout we assume that $\tX$ is a tropical variety in $\nT^n$ not just a germ and that $\tF$ acts on $\tX$ by automorphism. 

\begin{rmk}\label{rmk:conesasweightedprojspaces}
Let $\tX \subseteq \nT^n$ be a cone along $w \in (\nN^*)^{n}$ a primitive vector with positive integer entries.
We may consider the equivalence relation $\sim_w$ on $\nT^n \setminus \{\infty\}$ defined by
$$
(\tx_1, \ldots, \tx_n) \sim_w (\lambda w_1 + \tx_1, \ldots, \lambda w_n + \tx_n) \text{ for all } \lambda \in \nR.
$$
Quotienting by this equivalence relation we get
$$
\nP_w(\tX):=\tX \setminus \{\infty\}/\sim_w \ \subseteq \ \nT^{n} \setminus \{\infty\}/\sim_w =: \nT\nP_w^{n-1}.
$$
The latter space is the tropical analogue of the weighted projective space
$$
\nK\nP_w^{n-1} = \nA_\nK^{n} \setminus \{0\}/\sim_w, \quad (\ax_1, \ldots, \ax_n) \sim_w (\alambda^{w_1} \ax_0, \ldots, \alambda^{w_n}\ax_n) \text{ for all } \alambda \in \nK^*.
$$
Classically, weighted projective spaces are endowed with a natural orbifold structure. They are compactifications of $\nA^{n-1}$ whenever $\min\{w_i\}=1$.

As in the complex setting, the tropical weighted projective space is not a smooth tropical manifold. It is rather endowed with a ``tropical orbifold structure'' (which have been introduced to study moduli spaces of tropical curves, see e.g. \cite{abramovich-caporaso-payne:tropmodulispacecurves}).
In this case, the orbifold structure is a generalization of Definition \ref{def:abstractmfld}, where the overlapping maps are   \emph{rational} linear maps, i.e.~in $\text{GL}_n(\nQ)$, instead of integer ones.

The action $\aPhi:\nA^n \to \nA^n$ given by $\aPhi(\ax_1, \ldots, \ax_n)=(\ax_1^{w_1}, \ldots, \ax_n^{w_n})$ descends to the quotient and gives  a map $\nK\nP^{n-1} \to \nK \nP_w^{n-1}$.
Tropicalizing this situation, we get the linear map $\tPhi: \nT^n \to \nT^n$ given by the diagonal matrix with entries $w=(w_1, \ldots , w_n)$.
The map $\tPhi$ sends cones along $\one=(1, \ldots, 1)$ to cones along $w$.
The inverse of $\tPhi$ is a linear map defined over $\nQ$ and of course sends cones along $w$ to cones along $\one$. 
The maps $\tPhi$ and $\tPhi^{-1}$ preserve weights and the balancing condition. 
\end{rmk}

\begin{rmk}\label{rmk:metrictree}
Let $(\tX,\tF)$ be a regular Hopf data. By Proposition \ref{prop:Hopfdatacone}, the tropical variety $\tX$ is a cone along some direction $w$.
Consider the space $T=\nP_w(\tX)$ as described in Remark \ref{rmk:conesasweightedprojspaces}.

In dimension $d=2$, the quotient $T$ is a tree. The metric on  $\nT^d$ 
induces (not in a canonical way) a metric on $T$, which allows to distinguish between bounded and unbounded edges (called sometimes \emph{leaves}).

Analogous statements can be derived in any dimension $d \geq 3$. In this case $T$ is a complex of dimension $d-1$, homotopic to a point.
\end{rmk}

\section{Tropical Hopf manifolds and contracting germs}\label{sec:Hopfandgerms}

In this section we relate the geometry of tropical Hopf manifolds with  the dynamical features of tropical germs.

\subsection{Tropical germs}

\begin{defi}
Let $\tx=(\tx_1, \ldots, \tx_d)$ be a $d$-tuple of coordinates.
A formal series in $\nT\fps{\tx}$ is a formal tropical sum 
$$
\tphi(\tx) = \trop{\sum_{I \in \nN^d} \tphi_I \ax^I} = \bigwedge_{I \in \nN^d} (\tphi_I + \scalprod{I,\tx}),
$$
where $\tphi_I \in \nT$ for any multi-index $I \in \nN^d$.
\end{defi}
\begin{defi}
A formal power series $\tphi \in \nT\fps{\tx}$ is \emph{convergent}  at $\infty$ if there exists a neighborhood $U$ of $\infty$ such that $\tphi|_U$ is determined by a finite number of tropical monomials. We say that $\phi$ is convergent on $U$. 

Let $\tf:(\nT^d, \infty) \to (\nT^c, \infty)$ be a $c$-tuple of formal power series in $d$-variables. Denote by $\ty_1, \ldots, \ty_c$ the standard coordinates in $\nT^c$.
Then $\tf$ is called \emph{convergent} if $\ty_j \circ \tf:(\nT^d, \infty) \to (\nT, \infty)$ is convergent for every $j=1, \ldots, c$.
\end{defi}

\begin{ex}
Assume we are in dimension $d=1$.
The formal power series
$\displaystyle
\tphi_1(\tx)
=
\bigwedge_{n \geq 1} (na+n\tx)
$
is convergent, since $\tphi_1(\tx)=a+\tx$ on $\{\tx > -a\}$.
In particular, as a germ, $\tphi_1$ coincides with the translation by $a$.

Analogously, $\tphi_2(\tx)=\displaystyle\bigwedge_{n \geq 1} (-\log n + n \tx)$ is convergent.
Indeed, for all $m > n$, we have that $-\log m + m \tx > -\log n + n \tx$ if and only if $\tx > \frac{\log m - \log n}{m-n}$.
Since this quantity is bounded away from $\infty$, we infer that $\tphi_2$ is convergent.

On the other hand, $\tphi_3(\tx)=\displaystyle\bigwedge_{n \geq 1} (-n^2 + n \tx)$ is not convergent.
Indeed, for all $m > n \geq 0$, we have $m(\tx-m) \geq n(\tx-n)$ if and only if $\tx \geq m+n$.
In particular $\tphi_3(\tx)\leq n(\tx-n)$ whenever $\tx < n$.
Hence $\tphi_3(\tx) \to -\infty$ for all $\tx \in \nR$, and $\tphi_3$ is not convergent.
\end{ex}

\begin{prop}\label{prop:convergent}
A tropical formal power series $\tphi(\tx)=\displaystyle\bigwedge_I (\tphi_I + \scalprod{\tx,I})$ is convergent if and only if there exists an \emph{analytic} germ $\aphi:\nA_\nK^d \to \nK$ such that $\tphi = \trp(\aphi)$. 
In this case all formal germs tropicalizing to $\tphi$ are convergent.
\end{prop}
\begin{proof}
Denote by $\nu_0$ the valuation in $\nK$, and by $\abs{\,\cdot\,}_0=e^{-\nu_0(\cdot)}$ the norm associated to $\nu_0$.
Since $\abs{\,\cdot\,}$ is a non-archimedean norm, $\aphi=\displaystyle\sum_I \aphi_I \ax^I$ is convergent if and only if there exists a real number $\alpha > 0$ such that $\abs{\aphi_I}_0 \alpha^{\abs{I}} \to 0$ for $\abs{I} \to +\infty$, where $I=(i_1, \ldots, i_d)\in \nN^d$ is a multi-index and $\abs{I}=i_1 + \cdots + i_d$.
Notice that in this case the convergence radius $r$ of the formal power series $\aphi$ at $0$ satisfies $r \geq \alpha$.

Suppose $\tphi$ is the tropicalization of a convergent germ $\aphi$, satisfying the above condition. Then in the region $\{\tx \in \nT^d\ |\ \tx_j \geq a \text{ for } j = 1, \ldots, d\}$, where  
$a = -log(\alpha)$,
the germ $\tphi$ is given by a finite collection of monomials.
To obtain the converse, it suffices to remark that the convergence condition for $\aphi$ depends only on the values $\tphi_I=\nu_0(\aphi_I)$.
\end{proof}

When working over a valued field, typically analytic germs are classified up to change of coordinates.
Depending on the type of  regularity desired for the change of coordinates, there are several classifications of this kind that can be found in the literature (see \cite{milnor:dyn1cplxvar} and references therein for dynamics in $1$ complex variable, \cite{sternberg:localcontractions,rosay-rudin:holomorphicmaps,berteloot:methodeschangementechelles} for contracting automorphisms, \cite{hubbard-papadopol:supfixpnt,favre:rigidgerms,buff-epstein-koch:bottchercoordinates,ruggiero:rigidgerms} for superattracting germs, \cite{abate:residualindexdynholmaps, abate-tovena:formalnormalformsholotangid} for tangent to the identity germs, and \cite{herman-yoccoz:smalldivisornonarchi, lindahl:siegellinprimechar, ruggiero:superattrdim1charp, jenkins-spallone:padictangentid} for non-archimedean dynamics).

Over the tropical semi-field $\nT$,  the only change of coordinates that act as automorphisms on a neighborhood of $\infty \in \nT^d$ are permutations of coordinates composed with a translation.
On the other hand, not all tropical monomials contribute to the value of a tropical germ in a suitably small neighborhood if $\infty$, giving us a non-trivial equivalence relation on tropical germs. 
In the following, we are interested in the behavior of tropical maps $\tf : (\nT^d, \infty) \to (\nT^c, \infty)$ on suitably small neighborhoods of $\infty$.

\begin{defi}
We refer to \emph{tropical germs} as equivalence classes of convergent tropical maps $\tf : (\nT^d, \infty) \to (\nT^c, \infty)$, where we say that $\tf$ and $\tg$ are equivalent if for any $U,V$ where $\tf$ and $\tg$ are convergent, there exists a neighborhood $W \subset U \cap V$ of $\infty$ so that $\tf|_W=\tg|_W$.
\end{defi}

Before stating a precise result, we need to fix some standard notation.

\begin{defi}
Let $I=(i_1, \ldots, i_d),J=(j_1, \ldots, j_d) \in \nN^d$ be multi-indices.
We say that $I \preccurlyeq J$ if $i_h \leq j_h$ for all $h= 1, \ldots, d$. In other words, $I \preccurlyeq J$  if and only if $J-I \in \nN^d$.
We set $$\monid{I}=\{J \in \nN^d\ |\ I \preccurlyeq J\}.$$
\end{defi}
Notice that $\preccurlyeq$ defines a partial order on $\nN^d$.
It naturally extends to a partial order on $\nR_+^d$, where $\nR_+ = [0,+\infty)$.

\begin{defi}
Let $\tx=(\tx_1, \ldots, \tx_d)$ be coordinates in $\nT^d$, and let 
$
\tphi(\tx)= \displaystyle\bigwedge_{I\in \nN^d} (\tphi_I +\scalprod{I,\tx}) \in \nT\fps{\tx}
$ be a formal power series.
Define $$
\mon{\tphi}:=\{I \in \nN\ |\ \tphi_I \neq \infty\}
$$
the set of exponents that appear with a non-trivial coefficient in the expression of $\tphi$.

The 
Newton polyhedron $\Newton{\tphi}$ of $\tphi$ is defined as the convex hull
$$
\Newton{\tphi} = \on{Conv}\left(\bigcup_{I \in \mon{\tphi}} \monid{I}\right).
$$
The
Newton diagram, denoted  $\partialNewton{\tphi}$, consists of the compact faces of the  Newton polyhedron. 
We finally set $\VertNewton{\tphi}$ as the (finite) set $\partialNewton{\tphi} \cap \mon{\tphi}$.
\end{defi}

\begin{ex}\label{ex:phinewton}
Consider the polynomial
\begin{equation}\label{eqn:phinewton}
\tphi(\tx_1,\tx_2)=(5\tx_1+\tx_2)\wedge(a+3\tx_1+2\tx_2)\wedge(\tx_1+3\tx_2) \wedge (4\tx_1+2\tx_2) \wedge (4\tx_1+5\tx_2),
\end{equation}
where $a \in \nR$ (see Figure \ref{fig:newton}).

In this case $\mon{\tphi}=\{(5,1),(3,2),(1,3),(4,2),(4,5)\}$, corresponding to the points in red, blue and green in the picture.
For $I=(4,5)$, the set $\monid{I}$ is highlighted in blue, while $\Newton{\tphi}$ is highlighted in red.
The segment between $(5,1)$ and $(1,3)$ highlighted in red is $\partialNewton{\tphi}$, while $\VertNewton{\tphi}=\{(5,1),(3,2),(1,3)\}$.
\end{ex}

\begin{figure}
\begin{minipage}[t]{.35\columnwidth}
	\def\svgwidth{\columnwidth}
	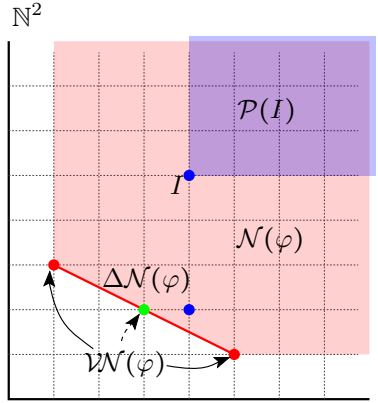
\end{minipage}
\caption{Newton polyhedron for $\phi$ from Example \ref{ex:phinewton} } \label{fig:newton}
\end{figure}

Notice that the points in $\Delta \Newton{\tphi}$ are the points of $\Newton{\tphi}$ which are minimal with respect to the order defined by $\preccurlyeq$. 

\begin{prop}\label{prop:normalform}
Let $\tx=(\tx_1, \ldots, \tx_d)$ be coordinates in $\nT^d$, and let
$$
\tphi(\tx)= \displaystyle\bigwedge_{I\in \nN^d} (\tphi_I +\scalprod{I,\tx}) : (\nT^d,\infty) \to (\nT,\infty)
$$
be a \emph{convergent} power series.
Then there exists a neighborhood $U$ of $\infty \in \nT^d$ such that
\begin{equation}\label{eqn:normalform}
\tphi(\tx) \equiv \nf{\tphi}(\tx):=\hspace{-2mm}\bigwedge_{I \in \VertNewton{\tphi}}\hspace{-3mm} \big(\tphi_I + \scalprod{I,\tx}\big),
\end{equation}
when restricted to $U$.
\end{prop}

\begin{proof}
Since $\tphi$ is convergent, there is a neighborhood $U$ of $\infty$ in which $\tphi$ is determined by just a finite number of monomials. Therefore, we may assume $\tphi$ is a tropical polynomial. 
The tropical hypersurface $V_{\tphi} = \on{div}(\tphi) \subset \nT^d$ defined by $\tphi$ 
 is a finite rational polyhedral complex of dimension $d-1$, it is dual to the regular subdivision $\mc{S}(\tphi)$ of the Newton polytope of $\tphi$
induced by its coefficients. 
The connected components of the  complement of $V_{\tphi}$ are in bijection with the integer points in the dual subdivision $\mc{S}(\tphi)$. 
On a connected component of the complement, $\tphi$ is an integer affine function corresponding to a single monomial of $\tphi$. 

In addition, we can choose a neighborhood $\wt{U}$ of $\infty$ such that the only connected components of the complement which intersect $\wt{U}$  correspond  to  
monomials which are in the Newton diagram $\partialNewton{\tphi}$. Therefore, on $\wt{U}$ the function $\tphi$ coincides with $\wt{\tphi}$ from the statement of the proposition. 
\end{proof}

\begin{ex}
Consider again $\tphi$ defined by \eqref{eqn:phinewton}, see Example \ref{ex:phinewton}.
Proposition \ref{prop:normalform} says that there exists a neighborhood $U$ of $\infty$ where $\tphi$ coincides with
$$
\nf{\tphi}(\tx_1,\tx_2)=(5\tx_1+\tx_2)\wedge(a+3\tx_1+2\tx_2)\wedge(\tx_1+3\tx_2).
$$
The monomials that appear in the normal form are drawn in red and green in Figure \ref{fig:newton}.

The red points are the corner locus of $\partialNewton{\tphi}$, and for any neighborhood $U$ of $\infty$ there always exists a point $\tx \in U$ where the minimum in $\tphi$ is attained only in the monomials corresponding to one of these points.

The green point is the point in $\partialNewton{\tphi}$ which is not a  corner. In this case, the monomial associated may not matter in the neighborhood of $\infty$, depending of its coefficient. In this case, notice (for $\tx_j \geq 0$) that:
$$
a+3\tx_1+2\tx_2 = a+\tx_1+\tx_2+(2\tx_1+\tx_2) \geq a+\tx_1+\tx_2+(4\tx_1 \wedge 2\tx_2) = a+(5\tx_1+\tx_2)\wedge(\tx_1+3\tx_2).
$$
Hence, if $a \geq 0$, we can erase this monomial and $\tphi$ coincides with $(\tx_1,\tx_2)\mapsto(5\tx_1+\tx_2)\wedge(\tx_1+3\tx_2)$, while if $a < 0$, we cannot erase this monomial: for $(\tx_1,\tx_2)=(r,2r)$ we get $5\tx_1+\tx_2 = \tx_1+3\tx_2 = 7r > a+7r = a+3\tx_1+2\tx_2$, and $(r,2r) \to \infty$ when $r \to +\infty$.
\end{ex}

As a direct consequence of Proposition \ref{prop:normalform}, we get
\begin{cor}
Let $\tx=(\tx_1, \ldots, \tx_d)$ be coordinates in $\nT^d$, and let
 $\tf=(\tf_1, \ldots, \tf_d) : (\nT^d,\infty) \to (\nT^d,\infty)$ be a 
convergent 
germ.
Then there exists a neighborhood $U$ of $\infty$ such that $\tf|_U$ coincides with
$$
\nf{\tf}(\tx)=\big(\nf{\tf_1}(\tx), \ldots, \nf{\tf_d}(\tx)\big),
$$
where $\nf{\tf_k}$ is a tropical polynomial with monomials whose exponents are in $\VertNewton{f_k}$.
\end{cor}

\subsection{Properties of tropical and analytic germs}

A tropical polynomial function $\tf:\nT^d \to \nT^c$ describes the behavior of a \emph{generic} algebraic map $\af:\nA^d \to \nA^c$ tropicalizing to it.
The same holds for tropical germs.
In some cases, if a property holds for a generic map (resp., germ) tropicalizing to $\tphi$ then it holds for \emph{all} maps (resp., germs) tropicalizing to it.
For example, by Proposition \ref{prop:convergent} if a tropical germ is convergent then every analytic germ tropicalizing to it is convergent. Another example of this is the  contracting property for  germs, see Proposition \ref{prop:contracting}.

Other times, we need to consider generic lifts.
For example, let $\tf,\tg:(\nT^d,\infty) \to (\nT^d,\infty)$ be two tropical germs.
Then for \emph{generic} lifts $\af, \ag : (\nA^d,0) \to (\nA^d,0)$ of $\tf$ and $\tg$ respectively, we have $\tg \circ \tf= \trp(\ag \circ \af)$.
To prove this, it suffices to consider the case where $\tg:(\nT^d,\infty) \to (\nT,\infty)$.
Take $\af$ and $\ag$ so that $\trp(\af)=\tf$ and $\trp(\ag)=\tg$.
Then the coefficients of $\ag \circ \af$ are some power series $\aphi_I$ on the coefficients of $\ag$ and $\af$.
For a generic choice of $\ag$ and $\af$, the valuation of such coefficients is constant, given by the minimum of the valuations of the monomials in $\aphi_I$.

For specific choices of $\ag$ and $\af$, we may have cancellations, and $\trp(\ag \circ \af) \neq \tg \circ \tf$, as the following example shows. 

\begin{ex}\label{ex:composition}
Let $\tf,\tg:(\nT^2,\infty)\to(\nT^2,\infty)$ be defined by
$$
\tf(\tx_1, \tx_2)=(\tx_1+1, \tx_1 \wedge (\tx_2+2)),\qquad \tg(\ty_1, \ty_2)=(\ty_1-1, (\ty_1-3) \wedge (\ty_2-2)).
$$
Then $\tg \circ \tf(\tx_1, \tx_2)=(\tx_1, (\tx_1-1)\wedge \tx_2)$.
Consider the lifts
$$
\af(\ax_1,\ax_2)=(t \ax_1, \alpha \ax_1+t^2\ax_2), \qquad \ag(\ay_1,\ay_2)=(t^{-1}\ay_1, t^{-3}\ay_1 + t^{-2} \ay_2)),
$$
with $\alpha \in \nC^*$.
Then
$$
\ag \circ \af(\ax_1, \ax_2)=(\ax_1, (1+\alpha)t^{-2}\ax_1 + \ax_2). 
$$
If $\alpha = -1$, then $\trp(\ag \circ \af) \neq \tg \circ \tf$, while the equality holds for $\alpha \neq 1$.
\end{ex}

We now introduce a few properties for tropical germs, and compare them to their analytic  counterparts.

\begin{defi}
Let $\tf:(\nT^d, \infty) \to (\nT^d,\infty)$ be a (formal) tropical germ. 
We say that
\begin{itemize}
\item $\tf$ is \emph{non-degenerate} if every (formal) germ $\af:(\nA^d,0) \to (\nA^d,0)$ which tropicalizes to $\tf$ is invertible.
\item $\tf$ is \emph{weakly non-degenerate} if there exists a (formal) invertible germ $\af:(\nA^d,0) \to (\nA^d,0)$ which tropicalizes to $\tf$.
\end{itemize}
\end{defi}

These two conditions depend only on the linear part of a germ $\tf$, as shown below.
For a reference on determinants and singularity of tropical matrices, see \cite{izhakian:troparithmatrixalgebra}. 

\begin{defi}\label{def:matrixnondeg}
Let $\tA=(a_{i,j}) \in \nsqmat{d}{\nT}$ be a $d \times d$ matrix with entries in $\nT$.
We recall that the \emph{tropical determinant} of $\tA$ is given by the expression
$$
\on{det}_{\nT}(\tA) 
= \bigwedge_{\sigma \in S_d} \sum_{i=1}^d a_{i,\sigma(i)}.
$$
We say that
\begin{itemize}
\item $\tA$ is \emph{non-degenerate} if there is a unique permutation $\sigma \in S_d$ so that $\on{det}_{\nT}(\tA)=\displaystyle\sum_{i=1}^d a_{i,\sigma(i)}$.
\item $\tA$ is \emph{weakly non-degenerate} if $\on{det}_{\nT}(\tA) \neq \infty$.
\end{itemize}
\end{defi}

Of course, if a map is  non-degenerate it is also weakly non-degenerate.
The definitions of non-degenerate and weakly non-degenerate given for germs and linear maps are coherent, as the following proposition states.
\begin{prop}
Let $\tf:(\nT^d, \infty) \to (\nT^d,\infty)$ be a (formal) tropical germ. 
Then $\tf$ is non-degenerate (resp., weakly non-degenerate) if and only if its linear part is.
\end{prop}
\begin{proof}
Let $\tA$ be the linear part of $\tf$, and let $\af$ be any lift of $\tf$.
Denote by $\aA$ the linear part of $\af$. Since $\trp(\af)=\tf$, we have $\trp(\aA)=\tA$.
By properties of valuations, $\nu_0(\on{det}(\aA))\geq \on{det}_\nT(\tA)$, where the equality holds when the minimum in the expression for $\on{\det}_\nT(\tA)$ is obtained by a unique monomial.
The statement of the proposition follows directly.
\end{proof}

\begin{rmk}
Notice that as soon as a tropical germ $\tf$ admits a lift $\af:(\nA^d,0) \to (\nA^d,0)$ which is invertible, the same property holds for a generic lift. Indeed, the vanishing of the determinant of the linear part of a germ $\af$ is a codimension $1$ condition in the space of germs tropicalizing to $\tf$. 
\end{rmk}

We will also need the following definition.

\begin{defi}
Let $\tA=(a_{i,j}) \in \nsqmat{d}{\nT}$ be a $d \times d$ matrix with entries in $\nT$.
We say that $\tA$ is \emph{positive} (resp., \emph{negative}) if whenever $\det_{\nT}(\tA) = \sum_{i=1}^d a_{i, \sigma(i)}$ then $\on{sign}(\sigma)$ is positive (resp., negative).
Otherwise, its sign is undetermined.
\end{defi}

\begin{ex}
The linear tropical germ $\tf:(\nT^2,\infty) \to (\nT^2, \infty)$ given by
$$
\tf(\tx_1, \tx_2)=(\tx_1 \wedge \tx_2, \tx_1 \wedge \tx_2)
$$
is weakly non-degenerate, but not non-degenerate.
Indeed, its linear part is 
$\left(\begin{array}{cc}0 & 0 \\ 0 & 0\end{array}\right)$,
 and the tropical determinant expression
$\det_{\nT}(\tf)=(0 + 0) \wedge (0 + 0)$ attains its minimum in two monomials. 
Another way to see it is that its lift $\af:(\nA_\nK^2,0) \to (\nA_\nK^2,0)$ given by
$$
\af(\ax_1, \ax_2)=(\ax_1+\ax_2, \ax_1 - \ax_2),
$$
is invertible.
However, the lift $(\ax_1,\ax_2)\mapsto(\ax_1+\ax_2,\ax_1+\ax_2)$ of $f$ is not invertible.
\end{ex}

\begin{ex}
The tropical germ $\tf:(\nT^2,\infty) \to (\nT^2, \infty)$ given by
$$
\tf(\tx_1, \tx_2)=(\tx_1 \wedge 2\tx_2, \tx_1 \wedge 3\tx_2)
$$
is not weakly non-degenerate.
Indeed, any  lift $\af:(\nA^2,0) \to (\nA^2,0)$ is of the form
$$
\af(\ax_1, \ax_2)=(\alpha \ax_1 + \beta \ax_2^2, \gamma \ax_1 + \delta \ax_2^3), 
$$
which is not invertible for any choice of $\alpha, \beta, \gamma, \delta$ (whose evaluations by $\nu_0$ equal $0$).

Notice that the linear part $(\tx_1, \tx_1)$ of $\tf$, denoted by $\tA$, has determinant $\det_\nT(\tA) = \infty$. 
\end{ex}

On the one hand, we have seen that the condition for a tropical germ $\tf$ to be weakly non-degenerate corresponds to the fact that a generic lift $\af$ of $\tf$ is invertible.
On the other hand, a weakly non-degenerate germ may  not be invertible as a map on the tropical space. Nevertheless we can define formal inverses of a tropical  germ $\tf$, by taking the tropicalization of the inverse $\af^{-1}$ of  lifts $\af$ of $\tf$.

\begin{defi}\label{def:inverse}
Let $\tf:(\nT^d,\infty) \to (\nT^d, \infty)$ be a weakly non-degenerate tropical germ.
We define a \emph{formal tropical inverse} of $\tf$ to be $\tg$ with $\tg_i = \trp(\ag_i)$ where $\ag = \af^{-1}$ and $\af: (\nA^d,0) \to (\nA^d,0)$ is a germ at $0$ satisfying $\trp(\af)=\tf$.

The \emph{canonical tropical inverse} of $\tf$ is obtained as the tropicalization $\tg=\trp(\ag)$ of the inverse $\ag=\af^{-1}$ of a \emph{generic} map $\af$ tropicalizing to $\tf$.
We denote the canonical tropical inverse of $\tf$ by $\tf^{-1}$.
Let $\pi: \tX\to \nT^d$ be the regular tropical modification along the tropical functions $\tf_1, \dots \tf_d$.
Then $\tf^{-1}$ is also obtained as $\trp(\af^{-1})$, where $\af: (\nA^d,0) \to (\nA^d,0)$ is such that $\trp(\Graph{\af}) = \tX$, where $\Graph{\af}$ denotes the graph of $\af$.
\end{defi}

Notice that formal tropical inverses need not be unique, while the canonical formal inverse is, since the tropicalizations of the inverse of a generic map $\af$ tropicalizing to $\tf$ coincide.

Moreover, a tropical germ $\tf: (\nT^d, \infty) \to (\nT^d, \infty)$ has a unique formal inverse if and only if the divisors $\on{div}(\tf_i)$ intersect properly in some neighborhood of $\infty$.
In the case of linear maps there is a third equivalent notion concerning the minors of the associated tropical matrix.
The proof of the next proposition is elementary and is omitted. 
\begin{prop}
Let $\tf: \nT^d \to \nT^d$ be a tropical linear map defined by a matrix $\tA$. The following are equivalent:
\begin{enumerate}[(a)]
\item any subset of the divisors of the functions $\tf_i$ intersect properly in a neighborhood of $\infty$;
\item for every minor of $\tA$, the minimum in the expression of the tropical determinant is attained only once;
\item the map $\tf$ has a unique formal tropical inverse.
\end{enumerate}
\end{prop}

Of course, the above conditions on the matrix $\tA$ of a linear map to have a unique formal inverse are stricter than the non-degeneracy condition (see Example \ref{ex:notgenericmatrix}), with the exception of dimension $2$. 
In this case they coincide.

\begin{rmk}
For a  tropical linear map $\tf: \nT^2 \to \nT^2$ the canonical formal inverse from Definition \ref{def:inverse} is actually a \emph{pseudo-inverse} in the sense of semi-groups.
This means that $\tf \circ \tg \circ \tf = \tf$ and $\tg \circ \tf \circ \tg = \tg$.
\end{rmk}

We end this section by describing contracting germs.

\begin{defi}\label{def:contracting}
A convergent germ $\tf:(\nT^d, \infty) \to (\nT^d,\infty)$ at $\infty$ is \emph{contracting} if for any neighborhood $U$ of $\infty$ there exists a open neighborhood $V\subseteq U$ of $\infty$ such that $\tf(V) \subrelcpct V$.
\end{defi}

In the classical  setting, a germ $\af:(\nA^d,0) \to (\nA^d,0)$ is contracting (the same definition as above) if and only if the eigenvalues of the linear part of $\af$ have positive valuation.
The next proposition will show the analogous property for tropical germs, and also provide a link between contracting tropical germs and contracting properties of their lifts.

\begin{prop}\label{prop:contracting}
Let $\tf:(\nT^d, \infty) \to (\nT^d,\infty)$ be a convergent germ. Denote by $\tA$ the $d \times d$ matrix with entries in $\nT$ representing the linear part of $\tf$. The following conditions are equivalent:
\begin{enumerate}[(a)]
\item $\tf$ is contracting;
\item $\tlambda > 0$ for all tropical eigenvalues $\tlambda$ of $\tA$;
\item generic lifts of $\tf$ are contracting; 
\item all lifts of $\tf$ are contracting.
\end{enumerate}
\end{prop}
\begin{proof}
Let $\tf$ be a tropical convergent germ, and $\af$ any analytic germ tropicalizing to $\tf$.
First notice that if $\af$ is a \emph{generic} lift of $\tf$, then $\nu_0 \circ \af(\ax) = \tf \circ \nu_0(\ax)$ for all $\ax \in \nA^d$ close to $\infty$.
The equivalence between (a) and (c) easily follows.

Notice that (d) implies (c).
The value $\nu_0 \circ \af(\ax)$ is constant for a generic choice of $\af$, and could increase for special $\af$.
It follows that if $\tf$ is contracting, then every lift $\af$ of $\tf$ is contracting, and (a), (c), (d) are all equivalent.

Tropical eigenvalues are the corner locus of the tropical characteristic polynomial $P_\tA(\lambda)$ of $\tA$, which is defined as the tropical determinant $\tA \wedge \lambda I$, see \cite{akian-bapat-gaubert:minpluseigenvalperturb}, \cite{cuninghamegreen:charmaxpoly}.

We have that $P_\tA$ is the tropicalization of the characteristic polynomial $\alg{P}_\aA$ for a generic lift $\aA$ of $\tA$, or equivalently for the linear part of a generic lift $\af$ of $\tf$.
Moreover, tropical eigenvalues of $\tA$ correspond to the valuation of eigenvalues of a generic choice of $\aA$.
This gives the equivalence between (b) and (c).
\end{proof}

\begin{rmk}
Given a tropical germ $\tf:(\nT^d,\infty) \to (\nT^d,\infty)$, the existence of a contracting lift $\af:(\nA^d,0) \to (\nA^d,0)$ does not imply that $\tf$ is itself contracting.
Consider for example the tropical germ in dimension $2$ given by $\tf(\tx_1, \tx_2)=(\tx_1 \wedge \tx_2, \tx_1 \wedge \tx_2)$.
This is a weakly non-degenerate germ, which is not contracting.
However, the germ $\af(\ax_1, \ax_2)=(\ax_1 + \ax_2, -\ax_1 + (t-1) \ax_2)$ tropicalizing to $\tf$ is contracting.
The reason is that for this specific choice there are cancellations of lowest order terms when computing the eigenvalues of the linear map $\af$. 
This also explains the meaning of ``generic" in Proposition \ref{prop:contracting}. 
\end{rmk}

\begin{rmk}
Given a tropical germ $\tf = (\tf_1, \ldots, \tf_d) : (\nT^d,\infty) \to (\nT^d,\infty)$, we can consider the induced local dynamical system.
When $\tf$ is contracting, forward orbits are defined in a whole neighborhood of $\infty$, since there is a base of neighborhood $\tU$ of $\infty$ so that $\tf^n(\tU) \subset \tU$ for all $n \in \nN$.
This  corresponds to the analogous property for any (generic) lift $\af$ of $\tf$.
Notice that backward orbits are not necessarily  defined everywhere, unless $\tf$ is the tropicalization of a global automorphism $\af:\nA^d \to \nA^d$.
\end{rmk}

\subsection{Tropical monomialization of weakly invertible germs}\label{sec:monomialization}

In this section, we consider how to lift a weakly invertible tropical germ $\tf: \nT^d \to \nT^d$ to an isomorphism by passing to other models of the tropical affine space. 
In other words, we seek modifications $\pi : \tX \subset \nT^n \to \nT^d$, $\eta: \tY \subset \nT^m \to \nT^d$ and a monomial map $\tF: \nT^n \to \nT^m$ such that $\tF(\tX) = \tY $ and $\eta \circ \tF= \tf \circ \pi$ on the strict transform $\Strict{\pi} \subset X$ in a suitable neighborhood of $\infty$.
Let us first start with an example. 

\begin{ex}
Let $\tf:(\nT^2, \infty) \to (\nT^2, \infty)$ be given by the linear map
$
\tf(\tx_1,\tx_2)=(\tx_1 \wedge \tx_2, \tx_2).
$

Let $A_1:=\{(\tx_1,\tx_2) \in \nT^2\ |\ \tx_1 \leq \tx_2\}$ and $A_2 = \{(\tx_1,\tx_2) \in \nT^2\ |\ \tx_2 \leq \tx_1\}$.
Then $\tf|_{A_1}(\tx_1,\tx_2)=(\tx_1,\tx_2)$ while $\tf|_{A_2}(\tx_1,\tx_2)=(\tx_2,\tx_2)$.

Consider the modification $\pi:\tX \to \nT^2$ along $\tx_1=\tx_2$.
Then $\tX \subset \nT^3$ is an instance of Example \ref{ex:permutation} (with $\eta = \pi$). 
The linear monomial map $\tF(\tx_1,\tx_2,\tu)=(\tu,\tx_2,\tx_1)$ leaves $\tX$ invariant, and satisfies $\pi \circ \tF = \tf \circ \pi$ on the strict transform $\Strict{\pi} \subset X$. 
\end{ex}

In general, it is not possible to find a modification $\pi:\tX \to \nT^d$ as above satisfying $\tF(\tX) = \tX$.
Nevertheless, we can find two different modifications $\pi:\tX \to \nT^d$ and $\eta:\tY \to \nT^d$ so that $\tf$ lifts to an isomorphism $\tF:\tX \to \tY$, such that $\eta \circ \tF= \tf \circ \pi$ on the strict transform $\Strict{\pi}$ in a suitable neighborhood of $\infty$. 

\begin{ex}\label{ex:lineartriang}
Consider the linear map $\tf(\tx_1,\tx_2)=(a_1+\tx_1, (b_1+\tx_1) \wedge (b_2+\tx_2))$, with $a_1>b_2>0$.
Let $\f $ be a generic lift of $\tf$, and $\g = \f^{-1}$, then $g = \trp(\g)$ is given by 
$$
\tg(\ty_1,\ty_2)=\big(\ty_1-a_1, (\ty_1+b_1-a_1-b_2) \wedge (\ty_2-b_2)\big).
$$
The modifications $\pi$ and $\eta$ given by Theorem \ref{thm:intromonomialization} are along $\tx_2=\tx_1+b_1-b_2$ and $\ty_2=\ty_1+b_1-a_1$ respectively.
Figure \ref{fig:lineartriang} shows the dynamics of $\tf$ on $\nT^2$, and the dynamics defined by the monomialization $F:\tX \to \tY$.
Notice how the image of the modified line (in red) differs from the projection by $\eta$ of the image by $F$ of its lift (the red-shaded area).
\end{ex}

\begin{figure}
\begin{minipage}[t]{.55\columnwidth}
	\def\svgwidth{\columnwidth}
	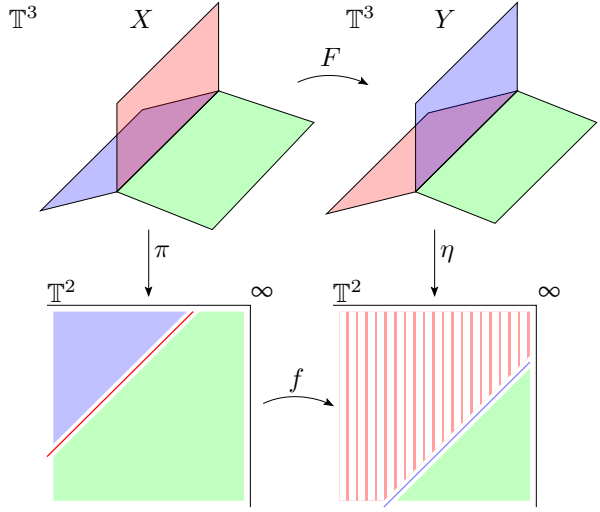
\end{minipage}
\caption{Monomialization for $\tf(\tx_1,\tx_2)=(a_1+\tx_1, (b_1+\tx_1) \wedge (b_2+\tx_2))$, with $a_1>b_2>0$.} \label{fig:lineartriang} 
\end{figure}

\begin{proof}[Proof of Theorem \ref{thm:intromonomialization}]
Given a tropical germ  $\tf: (\nT^d, \infty) \to (\nT^d, \infty)$, by Proposition \ref{prop:normalform} we may assume that, in a neighborhood $\tU$ of $\infty$, it is given by $d$ tropical polynomials which we call $\tf_1, \dots, \tf_d$.
Consider the sequence of regular modifications $\pi_i: X_i \to X_{i-1}$ where $X_{0} = \nT^d$ and $\pi_i$ is the modification along the function $\tf_i$, considered as a function $\nT^{d+i-1} \to \nT$ which is constant on all but the first $d$ variables. 
The space $\tX  \subset \nT^{2d}$ that results from this sequence of regular modifications is a balanced tropical cycle containing the strict transform $\Strict{\pi}$. 

Denote by $(\tx, \ty)$ the coordinates on $\nT^{2d}$, and consider the monomial map $\tF(\tx, \ty) = (\ty, \tx)$. 
We claim that in a suitable neighborhood $U$ of $\infty$, the tropical cycle $\tY := \tF(\tX)$ is obtained from $\nT^d$ via a tropical modification along functions $\tg_1, \dots, \tg_d$, where $\tg=(\tg_1, \ldots, \tg_d): \nT^d \to \nT^d$ is the canonical formal inverse of $\tf$.

By Corollary \ref{cor:realMultiGraph} we may find an invertible germ $\af = (\af_1, \dots , \af_d): (\nA^d,0) \to (\nA^d,0)$, defined in some neighborhood $\aU$ of $0$, satisfying $\tf = \trp(\af)$ and $\trp(\Graph{\af}) = \tX$.
Set $\ag = \af^{-1}$, and $\aF:\nA^{2d} \to \nA^{2d}$ defined by $\aF(\ax, \ay)=(\ay,\ax)$. Then on $\aU$ we have
$$
\aF \circ (\aid, \af) = (\af, \aid) = (\aid, \ag) \circ \af.
$$
Notice that the tropicalization of $(\aid,\af)$ is just the map $(\on{id},\tf): \tx \mapsto (\tx, \tf(\tx))$, which by construction satisfies $\pi \circ (\on{id},\tf)= \on{id}$.
Moreover, the restriction of $\pi:X \to \nT^d$ to $X \setminus \Exc{\pi}$ 
is invertible, and its inverse coincides with $(\on{id},\tf)$.
Upon tropicalizing  the situation, and we get
$$
\tF \circ (\on{id},\tf) = (\on{id},\tg) \circ \tf, \qquad \text{and hence }\ \eta \circ \tF = \tf \circ \pi \text{ on } 
X \setminus \Exc{\pi},
$$
for a suitable map $\eta:\tY \to (\nT^d,0)$.
By continuity, this relation extends to $\Strict{\pi}$.
Again by Corollary \ref{cor:realMultiGraph}, the modification $\eta$ must also be along the functions $\tg_i = \trp(\ag_i)$, and the theorem is proved. 
\end{proof}

\begin{ex}\label{ex:monomializationmadness}
Consider the germ
\begin{equation}
\tf(\tx_1,\tx_2)=\big((1 + \tx_1) \wedge p \tx_2, q \tx_1 \wedge (1+\tx_2)\big),
\end{equation}
where $p,q \in \nN$ satisfy $pq \geq 2$.
We shall show that this germ admits a monomialization locally at $\infty$, but not a global monomialization. This is because  there are no global automorphisms $\af:\nA^2 \to \nA^2$ tropicalizing to $\tf$.

We modify along the lines given by $f_1$ and $f_2$, obtaining a modification $\pi:\tX \to \nT^2$.
Then the strict transform of $\pi$ is given by
$$
\Strict{\pi}=\big\{(\tx_1, \tx_2, \tu_1, \tu_2) \in \nT^4\ |\ \tu_1 = (1 + \tx_1) \wedge p \tx_2, \tu_2=q \tx_1 \wedge (1+\tx_2)\big\}.
$$
The canonical formal inverse of $\tf$ is given by
$$
\tg(\ty_1, \ty_2) = \Big((\ty_1-1) \wedge \big(p \ty_2 -(p+1)\big), \big(q \ty_1-(q+1)\big) \wedge (\ty_2-1)\Big),
$$
Let $\eta:\tY \to \nT^2$ be the modification along $g_1, g_2$.
The strict transform of $\eta$ is given by
$$
\Strict{\eta}=\big\{(\ty_1, \ty_2, \tv_1, \tv_2) \in \nT^4\ |\ \tv_1 + 1 = \ty_1 \wedge (p \ty_2 - p), \tv_2 +1 = (q \ty_1-q) \wedge \ty_2 \big\}.
$$
Set $\tF(\tx_1, \tx_2, \tu_1, \tu_2)=(\tu_1, \tu_2, \tx_1, \tx_2)$.
We give a name to all facets of 
$\Strict{\pi}$ and $\Strict{\eta}$: 

\begin{minipage}{0.42 \columnwidth}
\begin{align*}
A_{++}&=\{\tu_1 = \tx_1+1 \leq p\tx_2, \tu_2 = \tx_2 + 1 \leq q\tx_1\},\\
A_{+-}&=\{\tu_1 = \tx_1+1 \leq p\tx_2, \tu_2 = q\tx_1 \leq \tx_2 + 1\},\\
A_{-+}&=\{\tu_1 = p\tx_2 \leq \tx_1+1, \tu_2 = \tx_2 + 1 \leq q\tx_1\},\\
A_{--}&=\{\tu_1 = p\tx_2 \leq \tx_1+1, \tu_2 = q\tx_1 \leq \tx_2 + 1\},\\
\end{align*}
\end{minipage}
\begin{minipage}{0.58 \columnwidth}
\begin{align*}
B_{++}&=\{\tv_1 +1 = \ty_1 \leq p\ty_2-p, \tv_2 +1 = \ty_2 \leq q\ty_1-q\},\\
B_{+-}&=\{\tv_1 +1 = \ty_1 \leq p\ty_2-p, \tv_2 +1 = q\ty_1-q \leq \ty_2\},\\
B_{-+}&=\{\tv_1 +1 = p\ty_2-p \leq \ty_1, \tv_2 +1 = \ty_2 \leq q\ty_1-q\},\\
B_{--}&=\{\tv_1 +1 = p\ty_2-p \leq \ty_1, \tv_2 +1 = q\ty_1-q \leq \ty_2\}.\\
\end{align*}
\end{minipage}
These are the facets of $\Strict{\pi}$ and $\Strict{\eta}$.
We denote by $A_{+0}$ the facets in $\tX$ that contains $A_{++} \cap A_{+-}$, and analogously for $A_{-0}$, $A_{0+}$, $A_{0-}$.
Finally, we denote by $A_{00}$ the last facet of $\tX$, generated by the directions along $\tu_1$ and $\tu_2$.
We use analogous notations for the facets in $\tY$ (see Figure \ref{fig:monomializationmadness}).

Let $c=(c_1,c_2)=\left(\frac{1+p}{pq-1}, \frac{1+q}{pq-1}\right)$. This is the intersection point of the lines along which $\pi$ is a modification of $\nT^2$.
Let $C=\{(\tx_1, \tx_2) \in \nT^2\ |\ \tx_1 \geq c_1, \tx_2 \geq c_2\}$. The image of facets in $\pi^{-1}(C)$ is depicted in Figure \ref{fig:monomializationmadness}.  By 
 construction, we have $\eta \circ \tF = \tf \circ \pi$ on $\Strict{\pi}$.
We claim that $\tF(\tX \cap \pi^{-1}(C)) \subset \tY$.

Notice that on $A_{++}$ we have $(\ty_1, \ty_2, \tv_1, \tv_2)=\tF(\tx_1, \tx_2, \tx_1+1, \tx_2+1) = (\tx_1+1, \tx_2+1, \tx_1, \tx_2)$.
It follows easily that $\tF(A_{++}) = B_{++}$.

On $A_{+-}$ we have $(\ty_1, \ty_2, \tv_1, \tv_2)=\tF(\tx_1, \tx_2, \tx_1+1, q \tx_1) = (\tx_1+1, q \tx_1, \tx_1, \tx_2)$.
In particular $\ty_2=q\ty_1-q$ and $\tv_1+1=\ty_1$.
Moreover, $q\tx_1 \leq \tx_2+1$ implies $\ty_2 \leq \tv_2+1$.
Finally, $p\ty_2-p \geq \ty_1 \Leftrightarrow pq\tx_1-p \geq \tx_1+1$, which holds if and only if $(\tx_1, \tx_2) \in C$.
It follows that $\tF(A_{+-} \cap \pi^{-1}(C)) = B_{+0}$, and analogously we get $\tF(A_{-+} \cap \pi^{-1}(C)) = B_{0+}$.

On $A_{+0}$ we have $(\ty_1, \ty_2, \tv_1, \tv_2)=\tF(\tx_1, q\tx_1-1, \tx_1+1, \tu_2) = (\tx_1+1, \tu_2, \tx_1, q\tx_1-1)$.
It is easy to check that $\tF(A_{+0}) \subset B_{+-}$ in this case.
We may argue analogously for $A_{0+}$, and we are done.

As stated above, the monomialization holds only on the preimage $\pi^{-1}(C)$, and in fact globally $\tF(\tX)$ is not contained in $\tY$.
\end{ex}

\begin{figure}
\begin{minipage}[t]{.8\columnwidth}
	\def\svgwidth{\columnwidth}
	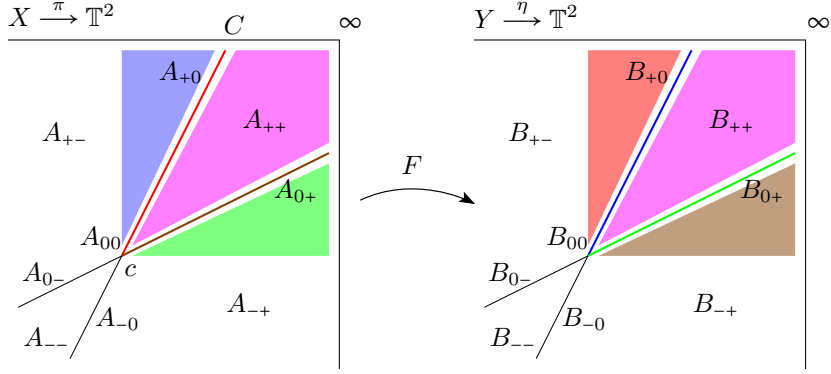
\end{minipage}
\caption{Monomialization for $\tf(\tx_1,\tx_2)=((\tx_1+1) \wedge p\tx_2, q\tx_1 \wedge (\tx_2+1))$, with $pq \geq 2$.} \label{fig:monomializationmadness} 
\end{figure}

\begin{rmk}\label{rmk:uniquenessMono}
The monomialization procedure is not unique. One reason is just that the function defining a divisor is only defined up to multiplicative constant.
So in general we could have considered $\alpha, \beta \in (\nK^*)^d$, and $\aF_{\alpha, \beta}(\ax, \ay)=(\alpha^{-1} \ay, \beta \ax)$. Here we use the following notation: if $\beta=(\beta_1, \ldots, \beta_d)$ and $\ax=(\ax_1, \ldots, \ax_d)$, then $\beta \ax = (\beta_1 \ax_1, \ldots, \beta_d \ax_d)$. Analogously for $\alpha^{-1}=(\alpha_1^{-1}, \ldots, \alpha_d^{-1})$. We now have
$$
\aF_{\alpha, \beta} \circ (\aid, \alpha \af)= (\af, \beta \aid) = (\aid, \beta \ag) \circ \af.
$$
Set $a=\nu_0(\alpha)=(\nu_0(\alpha_1), \ldots, \nu_0(\alpha_d))$, and analogously $b=\nu_0(\beta)$.
This corresponds to taking the tropical modification  $\pi$  of $\nT^d$ along $\tf_i + a_i$, the tropical modification $\eta$ of $\nT^d$ along $\tg_i + b_i$, and
$\tF(\tx, \ty) = (\ty - a, \tx+b)$.
\end{rmk}

Sometimes the \emph{monomialization} procedure may  not be unique for less trivial reasons as well. 
We say a collection of tropical divisors of functions $f_1, \dots , f_k$ intersects properly in $\nT^d$ if
$$
\on{codim} \bigcap_{i \in I} \on{div}(\tf_i) = \abs{I}
$$
for all subsets $I \subset \{1, \dots, d\}$.
If the divisors of $\tf$ do not intersect properly, then there is more freedom in choosing a monomialization than in the example above. 
Suppose we have just two functions $\af_1, \af_2$, and let $\tf_1$, $\tf_2$ denote their tropicalizations. Let $\pi_1:\tX_1 \to \nT^d$ be the modification along the function $f_1$. We then have $\trp(\Graph{\af_1}) = \tX_1$. To carry out the next modification $\pi_2: \tX_2 \to X_1$ so that $\tX_2$ is the tropicalization of the graph along both functions, we would need to know the image of the tropicalization of $(\alg{id} \times \af_1)|_{\{\af_2=0\}}$ in $\tX_1$. 
In general, this tropical subvariety may not be the divisor of a tropical polynomial function, but of a rational function.
In such a case, rational tropical modifications are required. 

If the divisors intersect properly then the tropical modifications $\pi: X \to \nT^d$ along the functions $\tf_i$ are determined up to translation of $X$ in $\nT^d$.
Otherwise, the proof of Theorem \ref{thm:intromonomialization} makes a canonical choice of the liftings of the divisors at each step, in this case all divisors are defined by tropical polynomial functions.
These modifications are always realizable as the graph of a generic function $\af:\nA^d \to \nA^d$ tropicalizing to $\tf$ by Corollary \ref{cor:realMultiGraph}.

\begin{ex}\label{ex:notgenericmatrix}
Consider the following tropical matrix on the left and the lift of it  over $\nK$ on the right 
$$
\tA = \left(\begin{array}{ccc}-1 & 0 & 0 \\\infty & 0 & 0 \\\infty & \infty  & 0\end{array}\right) \qquad \qquad \qquad \aA =  \left(\begin{array}{ccc}a & b & c \\ 0& d & e \\ 0 & 0  & f\end{array}\right).
$$
Here the valuation $\nu_0$ of all non-zero coefficients of the right matrix are $0$ except for $a$ which has valuation $-1$. The matrices represent linear maps $\tf$ and $\af$ of $\nT^3$ and $\nA^3$ respectively.
Moreover, $A$ is non-degenerate, since its tropical determinant is $-1$ which is obtained by only $1$ monomial of the expression.
The divisors of  $\tf_2=x_1 \wedge x_2$ and $\tf_3 = x_1 \wedge x_2 \wedge x_3$ do not intersect transversally, so the tropicalization of the graph of $\af$ cannot be determined from just $\tf$.
This amounts to the fact that the inverse of $\aA$ contains the expression $be - dc$ in one of the entries and for appropriate choices of coefficients we can obtain $\nu_0(be-dc) > 0$.  
\end{ex}

\subsection{Hopf manifolds and contracting germs}

Here we are interested in the case when the above monomialization procedure produces an automorphism.
This means that given a contracting germ $\tf: \nT^d \to \nT^d$, the monomialization procedure provides a modification $\pi: \tX \to \nT^d$ and an automorphism $\tF: \nT^{2d} \to \nT^{2d}$ satisfying $\pi \circ \tF = \tf \circ \pi$ on the strict transform $\Strict{\pi}$ and such that $\tF(\tX) = \tX$.
By Remark \ref{rmk:uniquenessMono} we have some freedom to translate $\tX \subset \nT^{2d}$ and adjust the maps accordingly.
We infer the following criterium on $\tf$ to get an automorphism by Theorem \ref{thm:intromonomialization}. 

\begin{prop}\label{prop:Hopfconditiondivisors}
Let $f: (\nT^d, \infty)  \to (\nT^d, \infty)$ be a weakly non-degenerate tropical germ.
Then there exists a modification $\pi:\tX \subseteq \nT^n \to \nT^d$ and a tropical contracting automorphism $\tF: \nT^n \to \nT^n$ such that $\tF(\tX)=\tX$ and $\pi \circ \tF = \tf \circ \pi$ on the strict transform  $\Strict{\pi}$ in a neighborhood of $\infty$ if and only if $\displaystyle\sum_{i=1}^d \on{div}(\tf_i) = \sum_{i=1}^d \on{div}(\tg_i)$, where $\tf^{-1} = \tg = (\tg_1, \dots, \tg_d)$ is the canonical formal inverse of $\tf$.
\end{prop}
\begin{proof}
The proof follows by the proof of Theorem \ref{thm:intromonomialization}, and from the fact that the condition $\on{div}(\tphi) = \on{div}(\tpsi)$ is equivalent to $\tphi = \tpsi + b$ for some constant $b \in \nR$.
\end{proof}

\begin{figure}
\begin{minipage}[t]{.6\columnwidth}
	\def\svgwidth{\columnwidth}
	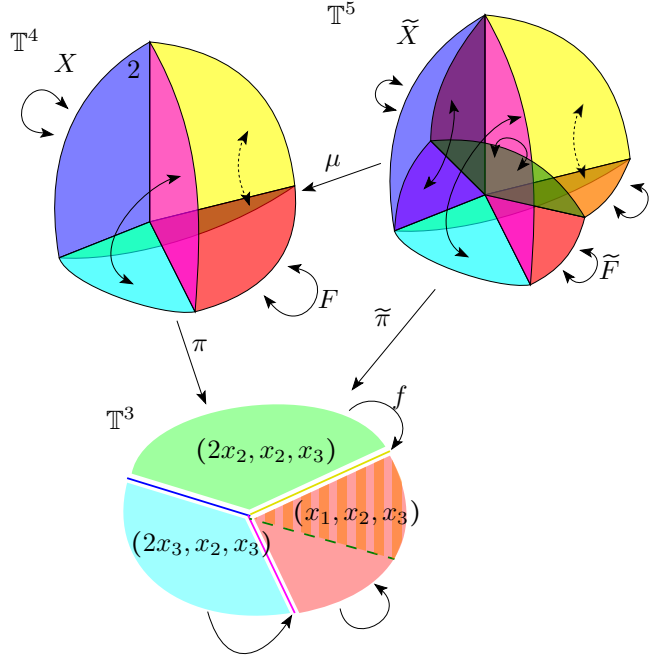
\end{minipage}
\caption{Section of the monomializations for $\tf=(\tx_1\wedge 2\tx_2 \wedge 2\tx_3, \tx_2,\tx_3)$.}\label{fig:monomialweights}
\end{figure}

Even if the monomialization procedure produces an automorphism as outlined in the above proposition,
we may still not obtain a non-singular  tropical Hopf data from a germ $\tf: \nT^d \to \nT^d$ for $d>2$, as the next example illustrates.

\begin{ex}\label{ex:weight2monom}
Consider the tropicalization of the germ $\af(\ax_1, \ax_2, \ax_3) = (\ax_1 + \ax_2^2  - \ax_3^2, \ax_2, \ax_3)$. The tropical divisor of the function $\tf_1=\tx_1 \wedge 2\tx_2 \wedge 2\tx_3$ has $3$ faces:
\begin{equation*}
F_1 := \{x_2 = x_3 \leq x_1 / 2 \}, \qquad
F_2 := \{x_1 = 2x_3 \leq 2x_2 \}, \qquad 
F_3 := \{x_1 = 2x_2 \leq 2x_3\}.
\end{equation*}
The face $F_1$ must be equipped with weight $2$, the others are of weight $1$. 
The only codimension $1$ face of $\on{div}(\tf_1)$ is given by 
$x_1 = 2x_2 = 2x_3$.  
Let $\pi: \tX \to \nT^3$ be the modification along $\tf_1$. Then $\tX$ will also have a face of weight $2$ (depicted in blue in Figure \ref{fig:monomialweights}), and thus $\tX$ is not the germ of a non-singular tropical variety. 

Recall that by Proposition \ref{lem:realGraph}, $\tX$ is the tropicalization of the graph of $\af_1(\ax_1, \ax_2, \ax_3) = (\ax_1+\ax_2^2-\ax_3^2)$.
Consider the embedding $i:\nA^3 \to \nA^5$ given by  
$$
i(\ax_1, \ax_2, \ax_3) = (\ax_1, \ax_2, \ax_3, \f_1(\ax_1, \ax_2, \ax_3), \ax_2 - \ax_3).
$$ 
It can be calculated that  $\wt{\tX}:=\trp(i(\nA^3))\subset \nT^5$ consists of $10$ faces of dimension $3$ all of which are of weight $1$.     
Moreover, by Theorem \ref{thm:intromonomialization} we can lift $\tf$ to an automorphism $\tF:\tX \to \tX$ and also to an automorphism $\wt{\tF}:\wt{\tX} \to \wt{\tX}$, whose action is described in Figure \ref{fig:monomialweights} (some intersections between faces in the picture associated to $\wt{\tX}$ do not exist; they  are only apparent due to the drop of dimension in the drawing).

Notice that a formal inverse of $\tf$ is again $\tf$.
In this case $\tf$ is not contracting, but it suffices to post-compose $\tf$ with the translation by $(1,1,1)$ to get a contracting germ with the same properties.
\end{ex}

\subsubsection*{Hopf manifolds and tropical germs in dimension $2$}

In dimension two, the study of weakly non-degenerate germs is simpler, due to the explicit normal forms that we may deduce from Proposition \ref{prop:normalform}.

\begin{prop}\label{prop:normalforms2d}
Any weakly non-degenerate germ $\tf : (\nT^2,\infty) \to (\nT^2,\infty)$ can be written (in a suitably small neighborhood $U$ of $\infty$) in the form
$$
\tf(\tx_1,\tx_2)=\big( (a_1 + p_1 \tx_1) \wedge (a_2+p_2 \tx_2), (b_1+ q_1 \tx_1) \wedge (b_2+ q_2 \tx_2)\big),
$$
with $a_1, a_2, b_1, b_2 \in \nR$, $p_1, p_2, q_1, q_2 \in \nN^* \cup \{+\infty\}$, and $p_1 \wedge p_2 = q_1 \wedge q_2 = p_1 \wedge q_1 = p_2 \wedge q_2 = 1$.
By convention, $a+p\tx \equiv \infty$ if $p = \infty$.

The linear part of $\tf$ is represented by $L=\begin{pmatrix}a_1+\delta_{p_1}^1 & b_1+\delta_{q_1}^1\\a_2+\delta_{p_2}^1 & b_2+\delta_{q_2}^1\end{pmatrix}$, where $\delta_p^1=\begin{cases}0 &\text{if }p=1\\ \infty &\text{if }p\neq 1\end{cases}$ is the tropicalization of the delta function.
Then $\tf$ is contracting if and only if $\on{det}_\nT(L)>0$ and $\on{tr}_\nT(L) > 0$.
\end{prop}
\begin{proof}
Since $\tf$ is weakly non-degenerate, its linear part $L$ needs to be weakly non-degenerate, meaning that either $(1,0) \in \VertNewton{\tf_1}$ and $(0,1) \in \VertNewton{\tf_2}$, or $(1,0) \in \VertNewton{\tf_2}$ and $(0,1) \in \VertNewton{\tf_1}$.
Consider the first case (the second is completely analogous). We have that $\partialNewton{\tf_1}$ is either the point $(1,0)$, or a segment with endpoints $(1,0)$ and $(0,p_2)$, with $p_2 \geq 1$.
It follows that either $\VertNewton{\tf_1}=\{(1,0)\}$ or $\VertNewton{\tf_1}=\{(1,0), (0,p_2)\}$.
Analogously, either $\VertNewton{\tf_2}=\{(0,1)\}$ or $\VertNewton{\tf_2}=\{(0,1), (q_1,0)\}$.
By Proposition \ref{prop:normalform} we conclude.
\end{proof}

The normal forms given above can also be stated  in a more explicit way.
\begin{cor}\label{cor:normalforms2d}
Any weakly non-degenerate germ $\tf : (\nT^2,\infty) \to (\nT^2,\infty)$ belongs to one of the following classes.
\begin{trivlist}
\item[The \emph{linear} case:] when $p_1, p_2, q_1, q_2 \in \{1,\infty\}$. In this case we can just write
\begin{equation}\label{eqn:nf2dlinear}
\tf(\tx_1,\tx_2)=\big( (a_1 + \tx_1) \wedge (a_2+\tx_2), (b_1+\tx_1) \wedge (b_2+\tx_2)\big),
\end{equation}
where here $a_1, a_2, b_1, b_2 \in \nT$, and $a_1+b_2 \neq \infty$ or $a_2+b_1 \neq \infty$ to have a weakly non-degenerate germ.

Denote by $L=\begin{pmatrix}a_1 & b_1\\a_2 & b_2\end{pmatrix}$ the matrix representing $\tf$.
Then $\on{tr}_\nT(L)=a_1 \wedge b_2$ and $\on{det}_\nT(L)=(a_1+b_2) \wedge (a_2+b_1)$.
It follows that $\tf$ is contracting if and only if $a_1,b_2,a_2+b_1 > 0$.

\item[The \emph{positive non-linear} case:] where $p_1=q_2=1$ and $p_2$ and $q_1$ are not both in $\{1,\infty\}$. In this case we have
\begin{equation}\label{eqn:nf2dpositive}
\tf(\tx_1,\tx_2)=\big( (a_1 + \tx_1) \wedge (a_2+p_2 \tx_2), (b_1+q_1 \tx_1) \wedge (b_2+\tx_2)\big),
\end{equation}
whose linear part $L$ is triangular.
In particular $\on{tr}_\nT(L)=a_1 \wedge b_2$ and $\on{det}_\nT(L)=(a_1+b_2)$, and $\tf$ is contracting if and only if $a_1, b_2 > 0$.

\item[The \emph{negative non-linear} case:] where $p_2=q_1=1$ and $p_1$ and $q_2$ are not both in $\{1,\infty\}$. In this case we have
\begin{equation}\label{eqn:nf2dnegative}
\tf(\tx_1,\tx_2)=\big( (a_1 + p_1\tx_1) \wedge (a_2+\tx_2), (b_1+\tx_1) \wedge (b_2+q_2\tx_2)\big),
\end{equation}
whose linear part $L$ is anti-triangular (at least $1$ of the two elements in the diagonal is $\infty$).
In particular $\on{tr}_\nT(L)=(a_1 + \delta_{p_1}^1) \wedge (b_2 + \delta_{q_2}^1)$ and $\on{det}_\nT(L)=(a_2+b_1)$. Moreover $\tf$ is contracting if and only if $a_1 + \delta_{p_1}^1), b_2 + \delta_{q_2}^1, a_2 +  b_1 > 0$.
\end{trivlist}
\end{cor}

Following the proof of Theorem \ref{thm:intromonomialization}, and to apply Proposition \ref{prop:Hopfconditiondivisors}, we need to compute formal inverses of the germs given by Corollary \ref{cor:normalforms2d}.

\begin{prop}\label{prop:inversenf2d}
Let $\tf: (\nT^2, \infty) \to (\nT^2, \infty)$ be a weakly non-degenerate germ, in its normal form as given by Corollary \ref{cor:normalforms2d}.
\begin{trivlist}
\item[The \emph{linear} case:] assume $\tf$ is linear, given by case \eqref{eqn:nf2dlinear} of Corollary \ref{cor:normalforms2d}, hence represented by the matrix $L=\begin{pmatrix}a_1&b_1\\a_2&b_2\end{pmatrix}$.
Then the canonical formal inverse $\tg$ of $\tf$ is the linear map represented by
$$
L^{-1}=\left(\begin{array}{cc}b_2 - \Delta & b_1 - \Delta\\ a_2 -\Delta & a_1 - \Delta \end{array}\right) \qquad 
\text{where} \quad 
\Delta = (a_1 + b_2) \wedge (a_2 + b_1) \text{ is the tropical determinant of } L.
$$
Moreover, $\tg$ is the unique formal inverse of $\tf$ if and only if $\tf$ is non-degenerate (i.e., if and only if $a_1+b_2 \neq a_2+b_1$).

\item[The \emph{positive non-linear} case:] $\tf$ is given as in Corollary \ref{cor:normalforms2d} by \eqref{eqn:nf2dpositive}.
Then $\tf$ has a unique formal inverse $\tg$ given by
$$
\tg(\ty_1, \ty_2) = \Big( \big(\ty_1 \wedge (a_2 + p_2\ty_2-p_2 b_2)\big) - a_1, \big(\ty_2 \wedge (b_1 + q_1\ty_1 - q_1 a_1)\big) - b_2 \Big).
$$

\item[The \emph{negative non-linear} case:] $\tf$ is given as in Corollary \ref{cor:normalforms2d} by \eqref{eqn:nf2dnegative}.
Then $\tf$ has a unique formal inverse $\tg$ given by
$$
\tg(\ty_1, \ty_2) = \Big( \big(\ty_2 \wedge (b_2 + q_2\ty_1-q_2 a_2)\big) - b_1, \big(\ty_1 \wedge (a_1 + p_1\ty_2 - p_1 b_1)\big) - a_2 \Big).
$$
\end{trivlist} 
\end{prop}
\begin{proof}
In order to have a non-unique formal inverse, the divisors of the function must intersect non-properly in a neighborhood of $\infty$. In dimension $2$, this means that the divisors of $\tf_1$ and $\tf_2$ must coincide, which only happens in the linear case when the matrix of $A$ is weakly non-degenerate but not non-degenerate, 
i.e. $a_1 + b_2 =a_2 + b_1$. 
The formal inverses are found directly, by taking a generic realization $\af$ of $\tf$, computing its inverse $\ag$, and tropicalizing to obtain $\tg=\trp(\ag)$.
\end{proof}

\begin{cor}\label{cor:dynsys2d}
Let $\tf: (\nT^2, \infty) \to (\nT^2, \infty)$ be a weakly non-degenerate contracting germ.
Then the monomialization procedure applied to $\tf$ produces a dynamical system if and only if $\tf$ is of the form
$$
\tf(\tx_1, \tx_2) = \big((a_1 + \tx_1) \wedge (a_2+p_2 \tx_2), (b_1+q_1 \tx_1) \wedge (b_2+\tx_2)\big),
$$
with either:
\begin{enumerate}[(a)]
\item $p_2, q_1 \in \{1, \infty\}$ (i.e., $\tf$ is linear), and either $a_2=b_1= \infty$ ($\tf$ is diagonal), or $a_1 = b_2 \in \nT$;
\item $p_2=\infty$, $q_1 \in \nN^*\setminus\{1\}$, and $b_2=q_1a_1$;
\item $q_1=\infty$, $p_2 \in \nN^*\setminus\{1\}$, and $a_1=q_2b_2$.
\end{enumerate}
\end{cor}

\begin{proof}
To prove the above statement we compare the divisors of $\tf_i$ and $\tg_i$, $i=1,2$, where $\tf=(\tf_1, \tf_2)$ is given in its normal form in Corollary \ref{cor:normalforms2d}, and $\tg=(\tg_1,\tg_2)$ is the canonical formal inverse of $\tf$ as given by Proposition \ref{prop:inversenf2d}.

For the linear case, we deduce that the statement is equivalent to having $\{a_2-a_1, b_2-b_1\} = \{a_2-b_2, a_1-b_1\}$.
If $a_2=b_1=\infty$, the condition is clearly satisfied.
Assume not both $a_2$ and $b_1$ are infinity. Then the condition is satisfied if and only if $a_1=b_2$.

Suppose $f$ is in the positive non-linear case, as in \eqref{eqn:nf2dpositive}.
In the case $p_2=\infty$, we have
\begin{align*}
\tf(\tx_1, \tx_2) &= \big(a_1 + \tx_1, (b_1+q_1 \tx_1) \wedge (b_2+\tx_2)\big),\\
\tg(\ty_1, \ty_2) &= \big(\ty_1 - a_1, \big(\ty_2 \wedge (b_1 + q_1\ty_1 - q_1 a_1)\big) - b_2 \big),
\end{align*}
where $\tg=\tf^{-1}$ is the canonical formal inverse of $\tf$.
The tropical modifications $\pi$ and $\eta$ given by the monomialization Theorem \ref{thm:intromonomialization} are obtained in this case along the tropical divisors supported on the affine lines
$$
\tx_2=q_1\tx_1 +b_1-b_2\qquad \text{and}\qquad \ty_2=q_1\ty_1+b_1-q_1a_1 \quad \text{respectively}.
$$
Since $a_1, b_1, b_2 \neq \infty$, we deduce that $\pi$ and $\eta$ coincide if and only if $b_2=q_1a_1$, and we get case (b).

The case $q_1=\infty$ is completely analogous, and gives case (c).

Assume that both $p_2$ and $q_1$ are different from $\infty$, and not both are equal to $1$.
In this case one can check that the tropical modifications $\pi$ and $\eta$ given by the monomialization Theorem \ref{thm:intromonomialization} are obtained respectively along the lines:
$$
\begin{cases}
x_2=\dfrac{x_1+a_1-a_2}{p_2},\\
x_2=q_1 x_1 + b_1 - b_2,
\end{cases}
\qquad
\begin{cases}
y_2=\dfrac{y_1+p_2b_2-a_2}{p_2},\\
y_2=q_1 y_1 + b_1 - q_1a_1.
\end{cases}
$$
These divisors correspond (in the sense of Proposition \ref{prop:Hopfconditiondivisors}) if and only if $a_1=p_2 b_2$ and $b_2=q_1 a_1$.
Hence $a_1=p_2 q_1 a_1$. Since $p_2 q_1 \geq 2$, this implies that either $a_1=0$ or $a_1 = \infty$, in contradiction with the fact that $\tf$ is contracting and non-degenerate, see Corollary \ref{cor:normalforms2d}.

Suppose finally that $f$ is in the negative non-linear case, as in \eqref{eqn:nf2dnegative}.
As before, one can check that the tropical modifications $\pi$ and $\eta$ given by the monomialization Theorem \ref{thm:intromonomialization} are obtained respectively along the lines:
$$
\begin{cases}
x_2=p_1 x_1 + a_1 - a_2,\\
x_2=\dfrac{x_1+b_1-b_2}{q_2},
\end{cases}
\qquad
\begin{cases}
y_2=q_2 y_1 + b_2 - q_2a_2,\\
y_2=\dfrac{y_1+p_1b_1-a_1}{p_1},
\end{cases}
$$
where we do not perform the modification along the line containing $p_1$ in its equation if $p_1=\infty$, and analogously for $q_2$, while they cannot both be $\infty$.
It follows that these divisors correspond if and only if $p_1=q_2 \in \nN^*\setminus \{1\}$, and $(q_2-1)a_2=b_2-a_1=-(p_1-1)b_1$.
It follows that $a_2+b_1 = 0$, contradicting that $\tf$ is  contracting.
\end{proof}

In the cases of Corollary \ref{cor:dynsys2d},  the modifications $\tX, \tY$ of $\nT^2$ can be chosen to be equal so $\tX = \tY$ and then monomialization of  the tropical polynomial function $f: \nT^2 \to \nT^2$ acts as an automorphism  $F: X \to X$. Therefore, the monomializations of these tropical polynomial functions produce Hopf data of dimension two as described in Definition \ref{cor:realMultiGraph}.
There remain cases when the modifications $\pi$ and $\eta$ do not coincide. In order to ``monomialize'' these germs in a dynamical way, an infinite number of modifications are necessary. This is postponed until Section \ref{ssec:noncompact}.

\subsubsection*{Hopf manifolds and tropical germs in higher dimensions}

Our aim is to describe an analogue of Corollary \ref{cor:dynsys2d} in higher dimensions.
By the characterization of Hopf manifolds given in Proposition \ref{prop:Hopfdatacone}, the description of normal forms given by Proposition \ref{prop:normalform}, plus the characterization of weakly non-degenerate matrices given by Definition \ref{def:matrixnondeg}, we infer the following.

\begin{prop}\label{prop:compactHopfalldim}
Let $\tf=(\tf_1, \ldots, \tf_d):(\nT^d,\infty) \to (\nT^d,\infty)$ be a contracting weakly non-degenerate germ.
Assume the monomialization procedure described in Theorem \ref{thm:intromonomialization} induces a dynamical system $\tF:\tX \to \tX$, where $\tX$ is a modification of $\nT^d$.
Then for every $j$, there exists an affine hyperplane $H_j$  in $\nR^d$ such that:
\begin{enumerate}[(a)]
\item for all $j=1, \ldots, d$, $\VertNewton{\tf_j}\subset H_j$;
\item all $H_j$ are parallel to one another;
\item $\bigcup_{j=1}^d \VertNewton{\tf_j}$ contains the canonical basis $\{e_1, \ldots, e_d\}$ of $\nN^d$.
\end{enumerate}
\end{prop}
\begin{proof}
By hypothesis, $(\tX, \tF)$ defines a (possibly virtually regular) Hopf data. By Proposition \ref{prop:Hopfdatacone}, all $1$-dimensional faces of $X$ of sedentarity $\emptyset$ are parallel to a vector $w \in \nR^d$.
By the proof of Theorem \ref{thm:intromonomialization}, the variety $\tX$ is obtained from $\nT^d$ by modification along the divisors $\on{div}(\tf_j)$.
By projection on $\nT^d$, all (sedentarity $\emptyset$) $1$-dimensional faces of $\on{div}(\tf_j)$ are parallel to the projection of  $w$.
By duality, this means that all (bounded) $1$-codimensional faces of $\partialNewton{\tf_j}$ are parallel to the hyperplane $H$ orthogonal to $v$.
By definition of $\VertNewton{\tf_j}$, assertions \textit{(a)} and \textit{(b)} follow.

If by contradiction there exists $i \in \{1, \ldots, d\}$ so that $e_i \not \in \bigcup_{j=1}^d \VertNewton{\tf_j}$, then the $i$-th column of the linear part $A$ of $\tf$ would be $(\infty, \ldots, \infty)$.
It follows that $\on{det}_\nT(A) = \infty$, which is a contradiction since $\tf$ is weakly non-degenerate. 
\end{proof}

\begin{rmk}
Proposition \ref{prop:compactHopfalldim} does not hold without the contracting assumption, not even in dimension $d=2$.
As an example, consider the germ $\tf(\tx_1, \tx_2)=(\tx_1 \wedge p_2 \tx_2, q_1 \tx_1 \wedge \tx_2)$, with $p_2, q_1 \in \nN$ and $p_2q_1 \geq 2$.
The linear part of $\tf$ is $\begin{pmatrix}0 & \infty \\ \infty & 0\end{pmatrix}$, hence $\tf$ is not contracting.
By Proposition \ref{prop:inversenf2d}, the canonical formal inverse of $\tf$ is $\tf^{-1} = \tf$.
It follows that the monomialization procedure of Theorem \ref{thm:intromonomialization} induces a dynamical system $\tF:\tX \to \tX$, in a neighborhood of $\infty$ but the conclusions of Proposition \ref{prop:compactHopfalldim} do not hold for $\tf$, since we modified along two non-parallel lines $\{\tx_1 = p_2 \tx_2\}$ and $\{\tx_2 = q_1 \tx_1\}$.
\end{rmk}

As we have already seen in dimension $2$ by Corollary \ref{cor:dynsys2d}, the conditions given by Proposition \ref{prop:compactHopfalldim} are not sufficient to guarantee we have a Hopf data defined by the monomialization process.
In higher dimensions, explicit necessary and sufficient conditions are very complicated to formulate, as suggested by the following examples.

\begin{ex}
Consider the linear map $\tf:\nT^d \to \nT^d$ defined by the matrix
$$
A=
\begin{pmatrix}
a & b & \cdots & b\\
b & \ddots & \ddots & \vdots \\
\vdots & \ddots & \ddots & b \\
b & \cdots &  b & a
\end{pmatrix},
$$
with $a,b \in \nT$.
By direct computation, its tropical inverse is given by
$$
A^{-1}=
\begin{pmatrix}
a' & b' & \cdots & b'\\
b' & \ddots & \ddots & \vdots \\
\vdots & \ddots & \ddots & b' \\
b' & \cdots &  b' & a'
\end{pmatrix}
- \on{det}_\nT(A),
\qquad
\begin{array}{l}
a'=(d-1)(a \wedge b),\\
b'=b+(d-2)(a \wedge b),\\
\on{det}_\nT (A) = d(a \wedge b).
\end{array}
$$

It follows that the monomialization of $\tf$ defines a dynamical system if and only if $a \leq b$.
\end{ex}

\begin{ex}\label{ex:PDhigherdim}
Let $\af:(\nA^d,0) \to (\nA^d,0)$ be a contracting (local) automorphism.
Denote by $\nu_0$ the $t$-adic valuation on $\nK=\nk(t^\nR)$, and $\abs{\,\cdot\,}_0=e^{-\nu_0(\cdot)}$ the induced norm on $\nK$.
Up to conjugacy, we may assume that $\af$ is in Poincar\'e-Dulac normal form (see \cite{sternberg:localcontractions,rosay-rudin:holomorphicmaps,berteloot:methodeschangementechelles}):
$$
\af(\ax_1, \ldots, \ax_d)=\big(\alambda_1 \ax_1, \alambda_2 \ax_2 + \aP_2(\ax_1), \ldots, \alambda_d \ax_d + \aP_d(\ax_1, \ldots, \ax_{d-1})\big), 
$$
where $1 > \abs{\alambda_1}_0\geq \cdots \geq \abs{\alambda_d}_0 > 0$.
The polynomials $\aP_j$ have only \emph{resonant monomials}: a monomial $\displaystyle \prod_{h=1}^{j-1} \ax_h^{i_h}$ appears in $\aP_j$ only if $\alambda_j=\displaystyle \prod_{h=1}^{j-1} \alambda_h^{i_h}$.

The tropicalization of $\af$ is of the form
$$
\tf(\tx_1, \ldots, \tx_d)=\big(\tlambda_1 + \tx_1, (\tlambda_2 + \tx_2) \wedge \tP_2(\tx_1), \ldots, (\tlambda_d + \tx_d) \wedge \tP_d(\tx_1, \ldots, \tx_{d-1})\big), 
$$
where $0 < \tlambda_1 \leq \cdots \leq \tlambda_d < \infty$.
The resonance condition says that a monomial $\displaystyle \sum_{h=1}^{j-1} i_h \tx_h$ appears in $\tP_j$ only if $\tlambda_j = \displaystyle \sum_{h=1}^{j-1} i_h \tlambda_h$.

Denote by $\wt{\tf}_{j} : (\nT^j, \infty) \to (\nT^j, \infty)$ the map given by the first $j$ coordinates of $\tf$.
A calculation shows that the canonical formal inverse of $\tf$ is given by
$$
\tf^{-1}(\ty_1, \ldots, \ty_d)=\Big(\ty_1-\tlambda_1, \big(\ty_2 \wedge \tP_2(\wt{\tf}_1^{-1}(\ty_1))\big)-\tlambda_2, \ldots, \big(\ty_d \wedge \tP_d\big(\wt{\tf}_{d-1}^{-1}(\ty_1, \ldots, \ty_{d-1})\big)-\tlambda_d\Big).
$$
Denote by $\tg_1, \ldots, \tg_d$ the coordinate functions  of $\tg=\tf^{-1}$.
By induction, one can show that $\VertNewton{\tf_j}=\VertNewton{\tg_j}$ for all $j$.
But in general $\sum_{j=1}^d \on{div}(\tf_j) \neq \sum_{j=1}^d \on{div}(\tg_j)$.

As an example, consider the case $d=3$, and  $\tlambda_3=\tlambda_2=2\tlambda_1$. We denote $\tlambda=\tlambda_1$. Then we have
\begin{align*}
\tf(\tx_1, \tx_2, \tx_3)&=\big(\tx_1+\tlambda, (\tx_2+2\tlambda) \wedge (\eps + 2\tx_1), (\tx_3 + 2\tlambda) \wedge (\eps_1 + \tx_2) \wedge (\eps_0 + 2\tx_1)\big),\\
\tg(\ty_1, \ty_2, \ty_3)&=\Bigg(\ty_1-\tlambda, \big(\ty_2 \wedge (\eps + 2\ty_1-2\tlambda))\big)-2\tlambda, \\
&\qquad\qquad\Big(\ty_3 \wedge \big((\eps_1 + \ty_2 -2\lambda) \wedge (\eps_1+\eps + 2\ty_1-4\tlambda)\big) \wedge \big(\eps_0 + 2\ty_1-2\lambda\big)\Big)-2\tlambda\Bigg).
\end{align*}
Hence we get that $\on{div}(\tf_j)=\on{div}(\tg_j)$ for $j=1, 2$, while $\on{div}(\tf_3)=\on{div}(\tg_3)$ if and only if $\eps_0 \leq \eps+\eps_1-2\lambda$.
\end{ex}

\begin{ex}
Let $\tf:(\nT^3, \infty) \to (\nT^3, \infty)$ be a contracting weakly non-degenerate germ of the form
$$
\tf(\tx_1, \tx_2, \tx_3)=\big((\tx_1 + a) \wedge (\tx_2 + b), (\tx_1+c) \wedge (\tx_2 + a), (\tx_3 + e) \wedge (u\tx_1 + \eps_1) \wedge (u\tx_2 + \eps_2)\big).
$$
Here $a,b+c,e > 0$, $\Delta:= 2a \wedge (b+c) \in (0,\infty)$, and $u \in \nN^*$.

The canonical formal inverse $\tg=\tf^{-1}$ of $\tf$ is given by
\begin{align*}
\tg(\ty_1, \ty_2, \ty_3)&=\Bigg(\big((\ty_1 + a) \wedge (\ty_2 + b)\big)-\Delta, \big((\ty_1+c) \wedge (\ty_2 + a)\big)-\Delta,\\
&\hspace{-1cm} \Big(\ty_3 \wedge \Big(u\ty_1-u\Delta + \big((au + \eps_1)\wedge(cu+\eps_2)\big)\Big) \wedge \Big(u\ty_2-u\Delta + \big((bu + \eps_1)\wedge(au+\eps_2)\big)\Big)\Big) -e \Bigg).
\end{align*}

We deduce that $\on{div}(\tf_j)=\on{div}(\tg_j)$ for $j=1,2$ (it also follows from Corollary \ref{cor:dynsys2d}), while $\on{div}(\tf_3)=\on{div}(\tg_3)$ if and only if
\begin{equation}\label{eqn:exmixedresonance}
\begin{cases}
\eps_1-e+u\Delta = (au + \eps_1)\wedge(cu+\eps_2), \\
\eps_2-e+u\Delta = (bu + \eps_1)\wedge(au+\eps_2).
\end{cases}
\end{equation}
For example, if $cu+\eps_2 \leq au+\eps_1$ and $bu+\eps_1 \leq au+\eps_2$, then $cu-au +\eps_2 \leq \eps_1 \leq au-bu+\eps_2$, and $\Delta = b+c$.
In this case the condition \eqref{eqn:exmixedresonance} is satisfied if and only if $2e=u(b+c)$, i.e, if $e/u$ is a (tropical) eigenvalue for the matrix $A=\begin{pmatrix}a&c\\b&a\end{pmatrix}$, which represents the linear part of $\tf$ restricted to the first two coordinates.
This situation is analogous to a classical dynamical situation of resonances. See \cite{ruggiero:rigidgerms} for other examples of resonances where the linear part is not in diagonal or triangular form.
\end{ex}

\subsection{Non-compact tropical Hopf manifolds}\label{ssec:noncompact}

In the previous subsections we have seen how, starting from a contracting germ $\tf:(\nT^d,\infty) \to (\nT^d,\infty)$, we can construct in a canonical way its monomialization.
In fact, Theorem \ref{thm:intromonomialization} provides two modifications $\pi : \tX \to \nT^d$ and $\eta : \tY \to \nT^d$, and an isomorphism 
$\tF: \nT^{2d} \to \nT^{2d}$ which is an isomorphism restricted to suitable neighborhoods of $X$ and $Y$ at $\infty$ which lifts 
  $\tf$ in the sense  that $\eta \circ \tF = \tf \circ \pi$ on the strict transform $\Strict{\pi}$.
When the two modifications coincide (up to translation by a scalar), we say that this data is a ``dynamical monomialization'' of $\tf$, we get a tropical Hopf manifold.

If the two modifications $\pi$ and $\eta$ do not coincide,   we can sometimes still find a dynamical monomialization of $\tf$, by performing an infinite (countable) number of monomialization.
As the proof of Theorem \ref{thm:intromonomialization} suggests, to do this we need $\tf$ to be obtained as the tropicalization of a \emph{global} automorphism of $\nA^d$ (which is contracting at $\infty$).
As  the next example shows, the latter condition is not invariant up to equivalence of tropical germs.

\begin{ex}\label{ex:twiceHenon}
We have seen that the tropical germ $\tf:(\nT^2,\infty)\to (\nT^2,\infty)$ considered in Example \ref{ex:monomializationmadness} is not the tropicalization of a global automorphism.
Nevertheless, $\tf$ is equivalent to 
$$
\wt{\tf}(\tx_1, \tx_2)=\Big((1+\tx_1) \wedge p \big(\tx_2 \wedge (q\tx_1-1)\big), q\tx_1 \wedge (\tx_2+1)\Big),
$$
which is the tropicalization of a global automorphism $\wt{\af}:\nA^2 \to \nA^2$.
In fact, $\wt{\af}$ can be taken as the composition $\aH_2 \circ \aH_1$ of the two H\'enon maps (see \cite{friedland-milnor:dynpolyauto}):
$$
\aH_1(\ax_1, \ax_2)= (t\ax_2 + \ax_1^q, \tx_1), \qquad \aH_2(\ax_1, \ax_2)=(t\ax_2 + t^{-p}\ax_1^p, \tx_1).
$$
\end{ex}

We now describe the strategy to prove Theorem \ref{thm:intrononcompact}.
Assume we are in the situation described above. 
To obtain a tropical Hopf data  in the sense of Definition \ref{def:tropHopfdata} we wish to find a modification of $\nT^d$ dominating both $\tX$ and $\tY$, i.e.~$\mu:\tW \to \nT^d$ and a monomial function $\thh: \tW \to \tW$ such that:
$$
\xymatrix@R20pt@C20pt{
& {\tW}\ar[dl]_{\wt{\eta}}\ar[dr]^{\wt{\pi}}\ar@(r,u)[]_>>>>>{\thh} &\\
{\tX}\ar[rr]^\tF\ar[dr]_{\pi} & & {\tY}\ar[dl]^{\eta}\\
& {\nT^d}\ar@(dr,ur)[]_>>>>>{\tf} &
}
$$

A dominating modification always exists, namely we can take the modification along the union of the divisors of the modification $\pi$ and $\eta$. 
Of course such a $\tW$ is not unique; any modification of $\tW$ could also produce such a commutative diagram.
There is a minimal such $\tW$.
Assuming the divisors are all disjoint, then $\mu: W \to \nT^d$ is the minimal such $\tW$.
If some of the divisors happen to coincide  then we only require modification once along each distinct divisor.

Notice that $\thh: \tW \to \tW$ is not an automorphism.
It is neither injective nor surjective, as some faces of $\tW$ are contracted to edges in the image and some faces of $\tW$  are not in the image of $\thh$.  
To rectify this, the idea is to continue to modify $\tW$ to $\tW^{(2)}$, adding two types of faces.
The first type of face is added to $\tW$ so that the contracted faces of $\tW$ have somewhere to land, and another kind of face is added so that they are sent to the faces of $\tW$ which are not in the image of $\thh$.
However, upon passing to this space $\tW^{(2)}$ this problem persists but now for different faces, and it requires us to modify again and again.
Ultimately an infinite polyhedral complex is obtained (thus non-compact) $\tW^{(\infty)}$ in order to monomialize such a map.

\begin{proof}[Proof of Theorem \ref{thm:intrononcompact}]
Let $\tf:\nT^d \to \nT^d$ be a weakly non-degenerate polynomial, obtained as the tropicalization of a contracting global automorphism $\af:\nA^d \to \nA^d$.
Denote by $\tf^n=(\tf_1^{(n)}, \ldots, \tf_d^{(n)})$ the $n$-th iterate of $\tf$, with $n \in \nZ$ (in the case of $n < 0$, $\tf^n = \tg^{-n}$ where $\tg$ is the canonical formal inverse of $\tf$).
We may pick $\af$ generic so that $\tf^n$ is the tropicalization of $\af^n$ for all $n \in \nZ$.

The monomialization result from Theorem \ref{thm:intromonomialization} is obtained by taking the modifications $\pi:\tX \to \nT^d$ along $\tf_1^{(1)}, \ldots, \tf_d^{(1)}$ and $\eta:\tY \to \nT^d$ along $\tf_1^{(-1)}, \ldots, \tf_d^{(-1)}$.
The monomial map is given by $\tF(\tx_1, \ldots, \tx_d, \tu_1, \ldots, \tu_d) = (\tu_1, \ldots, \tu_d,\tx_1, \ldots, \tx_d)$. To simplify notation we write $\tF(\tx,\tu)=(\tu,\tx)$.

The modification $\mu:\tW \to \nT^d$ that dominates both $\pi$ and $\eta$ is given by modifying along $\tf_j^{(1)}$ and $\tf_j^{(-1)}$, with $j=1, \ldots, d$ so that $W \subset \nT^{3d}$.
Denote the coordinates on $\nT^{3d}$ by
$$
(\tu^{(-1)}_1, \dots , \tu^{(-1)}_d, \tx_1, \dots, \tx_d, \tu^{(1)}_1, \dots \tu^{(1)}_d)
$$
and for simplicity by $(\tu^{(-1)}, \tx, \tu^{(1)})$.
We will suppose that the modification $\tW$ is given by taking the graph of $\tf_i^{(-1)}$ in the $\tu^{(-1)}_i$ coordinate and the graph of $\tf_i^{(1)}$ in the $\tu^{(1)}_i$ coordinate.

Set $\tW^{(1)}=\tW$, then the action $\thh^{(1)}:\tW^{(1)} \to \tW^{(1)}$ is given by $\thh(\tu^{(-1)},\tx, \tu^{(1)})=(\tx,\tu^{(1)},\tf(\tu^{(1)}))$.
Notice that $\thh^{(1)}$ is not an automorphism.
We can iterate this procedure to obtain modifications $\mu^{(n)}:\tW^{(n)} \to \nT^d$ and maps $\thh^{(n)}:\tW^{(n)} \to \tW^{(n)}$, where $\tW^{(n)} \subset \nT^{d(2n+1)}$ is obtained by modifying along $\tf_i^{(k)}$ for all $1\leq i \leq d$ and $\abs{k} \leq n$.
The map $\thh^{(n)}$ is of the form $\thh^{(n)}(\tu^{(-n)},\ldots \tu^{(n)})=(\tu^{-(n-1)},\ldots, \tu^{(n)},\tf(\tu^{(n)}))$.
Taking an inverse limit for $n \to \infty$, we get $\thh^{(\infty)}:\tW^{(\infty)}\to\tW^{(\infty)}$ defined by a shift on the coordinates, that is an isomorphism.
\end{proof}

\begin{rmk}
Alternatively, it is possible to  construct the infinite polyhedral complex from above by a gluing procedure. Notice that the map $\thh$ still defines a bijection  $\thh:A \to B$, where $A, B \subset \tW$ denote the strict transforms associated to the modifications $\wt{\eta}$ and $\wt{\pi}$ respectively. 
We then consider an infinite number of copies $W_n$, $A_n$, $B_n$, $n \in \nZ$, and consider the space
$$
\wt{W}=\bigsqcup_{n \in \nZ} W_n/\sim,
$$
where $W_n \supset A_n \ni \tu_n \sim \tu_{n+1} \in B_{n+1} \subset W_{n+1}$ if $\tu_{n+1} = \thh(\tu_n)$.

The action $\wt{\thh}:\wt{W} \to \wt{W}$ given by the shift $W_n \ni \tu \mapsto \tu \in W_{n+1}$ then corresponds to the action $\thh^{(\infty)}:\tW^{(\infty)}\to\tW^{(\infty)}$.
\end{rmk}

\begin{rmk}
The assumption that $\tf$ is  the tropicalization of a \emph{global} automorphism is necessary.
The same procedure above can be applied for any weakly non-degenerate germ $\tf:(\nT^d,\infty)\to (\nT^d,\infty)$,
but in this case, the map $\thh^{(1)}$ would be defined only on the preimage by $\mu^{(1)}$ of a suitable neighborhood $\tU$ of $\infty \in \nT^d$.
By iterating the process, $\thh^{(n)}$ would be defined only on the preimage by $\mu^{(n)}$ of $\tf^{n-1}(\tU)$, which tends to $\{\infty\}$, since $\tf$ is contracting.

In dimension $2$ nevertheless, using the normal forms given by Corollary \ref{cor:normalforms2d}, it is easy to show that any weakly non-degenerate map $\tf:(\nT^2,\infty) \to (\nT^2,\infty)$ is equivalent to another map $\wt{\tf}:(\nT^2,\infty)\to (\nT^2, \infty)$ which is the tropicalization of a global automorphism (see Example \ref{ex:twiceHenon}).

For any contracting germ $\tf:(\nT^2, \infty) \to (\nT^2, \infty)$ we can hence always construct dynamical monomializations.
To construct the non-compact tropical Hopf surface associated to these data, the global automorphism $\af$ tropicalizing to a tropical map equivalent to $\tf$ does not need to be globally attracting (as in the automorphism $\tilde{\af}$ Example \ref{ex:twiceHenon}).
It suffices to consider the action of $\thh^{(\infty)}$ on a small (punctured) neighborhood of $\infty$.
\end{rmk}

We conclude this section with a specific example of the dynamical monomialization by an infinite number of modifications.

\begin{ex}\label{ex:noncompactHopf}
Let $\tf(\tx_1,\tx_2)=(a_1+\tx_1, (b_1+\tx_1) \wedge (b_2+\tx_2) )$ with $a_1 > b_2 > 0$.
Let $\pi : X \to \nT^2$ and $\eta : Y \to \nT^2$ be the modifications constructed in the previous section and $\tF:\tX \to \tY$ be a lift of $\tF$.
Both modifications are along a single divisor and they are both depicted in Figure \ref{fig:lineartriang}.
The modification $\eta: \tW \to \nT^2$ dominating both $\pi$ and $\eta$ is depicted in Figure \ref{fig:noncompact1}. 
The color coding of this map shows which faces of $\tW$ are contracted to edges by the map $\thh$ (here depicted in light blue) and which faces of $\tW$ are not in the image of $\thh$ (depicted in white).

\begin{figure}
\begin{minipage}[t]{.5\columnwidth}
	\def\svgwidth{\columnwidth}
	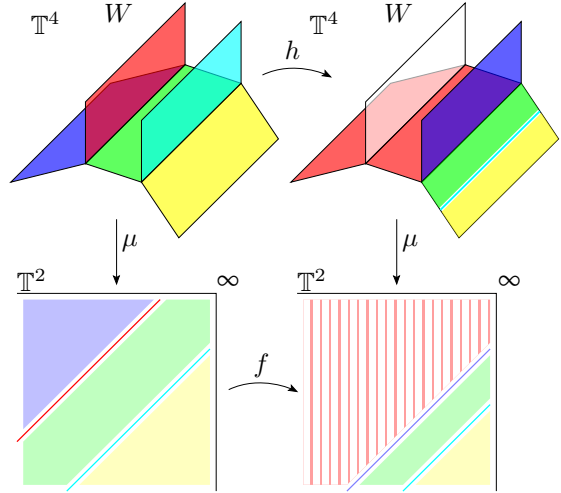
\end{minipage}
\caption{Construction of $\tW$ for $\tf(\tx_1,\tx_2)=(a_1+\tx_1, (b_1+\tx_1) \wedge (b_2+\tx_2))$, with $a_1>b_2>0$.} \label{fig:noncompact1} 
\end{figure}

In Figure \ref{fig:noncompact2}, we modify along the image by $\thh$ of the contracted face (the line depicted in light blue), and the preimage of the line obtained as the intersection between the face not contained on the image of $\thh$ and $\thh(\tW)$, obtaining a modification $\mu^{(2)}:\tW^{(2)} \to \nT^2$ dominating $\mu$, and a map $\thh^{(2)}:\tW^{(2)}\to\tW^{(2)}$.
This map satisfies $\xi^{(2)} \circ \thh^{(2)} =\thh \circ \xi^{(2)}$ on the strict transform $\Strict{\xi^{(2)}}$, where $\mu^{(2)}=\mu \circ \xi^{(2)}$.
The color coding shows how also in this case we have a map $\thh^{(2)}$ which contracts a face (depicted in brown) and does not contain a face in its image (depicted in white).

\begin{figure}
\begin{minipage}[t]{.6\columnwidth}
	\def\svgwidth{\columnwidth}
	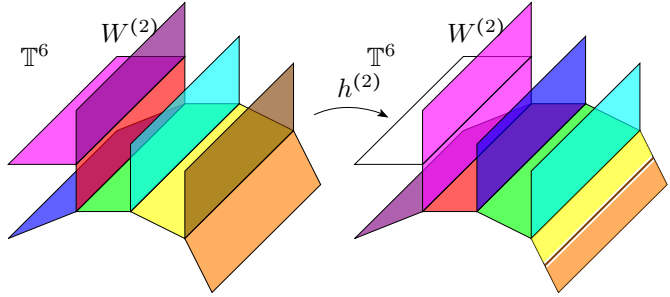
\end{minipage}
\caption{Construction of $\tW^{(2)}$ for $\tf(\tx_1,\tx_2)=(a_1+\tx_1, (b_1+\tx_1) \wedge (b_2+\tx_2))$, with $a_1>b_2>0$.} \label{fig:noncompact2} 
\end{figure}

We iterate this procedure, obtaining an infinite complex $\tW^{(\infty)}$ and an action $\thh^{(\infty)}:\tW^{(\infty)}\to\tW^{(\infty)}$, that acts as a shift on the faces of $\tW^{(\infty)}$.
In this case, topologically $\tW^{(\infty)}$ is a cone at $\infty$ over an infinite tree as in Figure \ref{fig:noncompact3}.
The action of $\thh^{(\infty)}$ restricted to the tree is given by a shift of the faces in the tree. 
The color coding describes the projection $\mu^{(\infty)}:\tW^{(\infty)} \to \nT^2$.
Notice that there is an infinite number of faces (the red ones) projecting to a single line (given by $\tx_2=\tx_1+b_1-b_2$).

\begin{figure}
\begin{minipage}[t]{.6\columnwidth}
	\def\svgwidth{\columnwidth}
	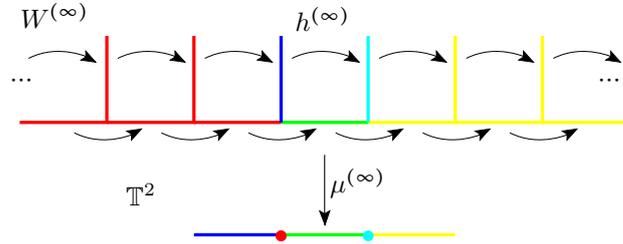
\end{minipage}
\caption{Section of $\tW^{(\infty)}$ for $\tf(\tx_1,\tx_2)=(a_1+\tx_1, (b_1+\tx_1) \wedge (b_2+\tx_2))$, with $a_1>b_2>0$.} \label{fig:noncompact3} 
\end{figure}

\end{ex}

\section{Tropical Hopf data and Analytification}\label{sec:analytification}

This section is devoted to studying the links between tropical Hopf data, tropical Hopf manifolds and the analogous objects in rigid-analytic geometry in the sense of Tate and Berkovich (see \cite{voskuil:nonarchhopf}).
We start by recalling some basics on Berkovich and valuation spaces (see \cite{berkovich:book}, \cite{favre-jonsson:valtree}, \cite{jonsson:berkovich}), and their connection to tropical spaces (see \cite{payne:analitlimittrop}).
We then study which (compact) tropical Hopf data are realized by classical Hopf data (see Definition \ref{def:realizeHopf}).
At last, we study skeleta of rigid-analytic Hopf manifolds, and their relations to tropical Hopf manifolds.

\subsection{Berkovich spaces and tropicalizations}\label{ssec:berktrop}

In the following, $\nK$ will denote an algebraically closed field of characteristic $0$, endowed with a non-trivial rank-one valuation $\nu_0$, whose value group is $\nR$.
Notice that the latter condition can be always achieved up to field extensions.
As an example, take the field of transfinite series $\nK=\nC((t^\nR))$, i.e., series of the form $\sum_{k \in \nR} \phi_k t^k$, where $\phi_k \in \nC$, and the set $\{k \in \nR\ |\ \phi_k \neq 0\}$ is well ordered, endowed with the $t$-adic valuation $\nu_0$, uniquely determined by $\nu_0(t)=1$.
We refer the reader back to Section \ref{sec:basics} for the conventions and definitions regarding tropicalizations.

\newcommand{\Ad}{{\nA^{\!d}}}

The analytification $\nA^d_{an}$ (in the sense of Berkovich) of the affine space $\Ad$ is given by the set of seminorms $\abs{\,\cdot\,}$ on $\nK[\ax_1, \ldots, \ax_d]$ extending $\abs{\,\cdot\,}_0=e^{-\nu_0(\cdot)}$.
We denote by $\mc{V}(\Ad)$ the space $\nA^d_{an}$ seen as semi-valuations $\nu:\nK[\ax_1, \ldots, \ax_d] \to (-\infty, +\infty]$ extending $\nu_0$,
the correspondence given by $\nu(\cdot) = - \log \abs{\,\cdot\,}$.

It was shown in \cite{payne:analitlimittrop} that the analytification of $\Ad$ is homeomorphic to the limit of an inverse system of tropicalizations. 
Here we outline the construction we need in our setting. We refer to Payne's work for details.

For each embedding $i: \Ad \to \nA^n$, there is a map $\rho_i: \mc{V}(\Ad) \to \mc{T}_i(\Ad)$ given by:
$$
\rho_i(\nu)=\big(\nu(\ay_1 \circ i), \ldots, \nu(\ay_n \circ i)\big),
$$
where $(\ay_1, \ldots, \ay_n)$ are the natural coordinates in $\nA^n$.
The set of affine embeddings of $\Ad$, denoted by $\mc{I}(\Ad)$, forms a poset with the order relation induced by equivariant morphisms $\nA^n \to \nA^m$. 
Define
$$
\mc{T}(\Ad) = \lim_{\substack{\longleftarrow\\i \in \mc{I}(\Ad)}} \mc{T}_i(\Ad).
$$
The homeomorphism from \cite{payne:analitlimittrop} is then given by, 
$$
\rho =\!\! \lim_{\substack{\longleftarrow\\i \in \mc{I}(\Ad)}}\!\! \rho_i\ : \mc{V}(\Ad) \to \mc{T}(\Ad).
$$

The map $\rho_i : \mc{V}(\Ad)  \to \mc{T}_i(\Ad)$ is related to the usual tropicalization map in the following way. 
There is the standard inclusion $\Ad \to \mc{V}(\Ad)$, which sends a point $\ax \in \Ad$ to the semi-valuation defined by $\nu_{\ax}(\phi) = \nu_0(\phi(\ax))$ where $\phi \in \nK[\ax]$.
Then the following diagram commutes, 
$$\xymatrix@R30pt@C30pt{
{\Ad\ }\ar@{^{(}->}[d]^i \ar@{^{(}->}[r] & {\mc{V}(\Ad)}\ar[d]^{\rho_i}\\
\nA^n \ar[r]^{\trp} & {\mc{T}_i(\Ad)}.
}
$$

\begin{ex}\label{ex:valline}
Consider the Berkovich line over $\nK$, in its valuation incarnation $\mc{V}(\nA^1_\nK)$, see e.g., \cite[Section 3]{jonsson:berkovich}, \cite[Chapter 4]{favre-jonsson:valtree}.
It is possible to describe all  valuations in $\mc{V}(\nA^1_\nK)$ in the following way. 
Let $\aalpha= \sum_{k \in \nR} a_k t^k \in \nK$, and $k \in \nR \cup \{+\infty\}$.
Now construct a valuation  $\nu_{\aalpha,r}$ by first decomposing  a polynomial $P \in \nK[\az]$ as a product of linear factors of the form $\az-\abeta$, for $\abeta \in \nK$.
We set $\nu_{\aalpha,r}|_\nK^* \equiv \nu_0$,
$$
\nu_{\aalpha,r}(\az-\abeta)=\min\{\nu_0(\aalpha-\abeta),r\},
$$
and we extend by linearity to any polynomial in $\nK[\az]$.
This valuation corresponds to the ball of center $\aalpha$ and radius $e^{-r}$ with respect to the non-archimedean norm $\abs{\,\cdot\,}_0 = e^{-\nu_0(\cdot)}$.
This representation is not unique, and in fact $\nu_{\aalpha,r}=\nu_{\abeta,r}$ if and only if $r \leq \nu_0(\aalpha - \abeta)$.

With respect to classical notations for Berkovich lines, the valuation $\nu_{\aalpha,r}$ is of type $1$ if $r = +\infty$, and of type $2$ if $r \in \nR$.
Notice that there are no points of type $3$, since $\nu_0(\nK^*)=\nR$, and no points of type $4$, since $\nK$ is spherically complete (the intersection of any decreasing family of non-empty balls is non-empty).

We can define a partial order on $\mc{V}(\nA^1_\nK)$ by saying that $\nu_{\aalpha,r} \leq \nu_{\abeta,s}$ whenever $r \leq s$ and $\nu_{\aalpha,r} = \nu_{\abeta,r}$.
This partial order induces a structure of $\nR$-tree on $\mc{V}(\nA^1_\nK)$, see again \cite{favre-jonsson:valtree}. All points of type $1$ are ends of the tree, and all other points are branching points.

Given a subset $S \subset \mc{V}(\nA^1_\nK)$, we denote by $\mc{V}(S)$ the convex hull of $S$, given by the union of all segments between elements in $S$.
\end{ex}

\begin{ex}\label{ex:valprojline}
We can consider what is sometimes called the \emph{Berkovich projective line} over $\nK$, denoted $\mc{V}(\nP^1)$, by adding a point $\nu_\infty$, corresponding to the limit of $\nu_{\aalpha,r}$ for $r \to -\infty$, for any $\aalpha \in \nK$.
Depending on the context, this valuation will be denoted in one of the following ways  $\nu_\infty = \nu_{\infty, +\infty} = \nu_{\aalpha,-\infty}$ for all $\aalpha \in \nK$. 
The definitions of the partial order, tree structure, segments and convex hulls naturally extend to this setting.

Consider the case of a finite set $S \subset \nP^1$. We define the convex hull of $S$ as the convex hull of $\{\nu_{\aalpha, +\infty}\ |\ \aalpha \in S\}$.
Then we can define a retraction $r_S:\mc{V}(\nP^1) \to \mc{V}(S)$ as follows.

Let $\nu=\nu_{\abeta,t} \in \mc{V}(\nP^1)$, with $\abeta \in \nK$ and $t \in [-\infty,+\infty]$.
Set $s = \max\{\nu_0(\abeta-\aalpha)\ |\ \aalpha \in S\}$, where $\nu_0(\abeta-\infty) = - \nu_0 (\abeta)$.
Then $r_S(\nu)=\nu_{\aalpha, s}$, where $\aalpha \in S$ is so that $\nu_0(\abeta-\aalpha)$ realizes the maximum in the expression for $s$.
Notice that $r_S$ is the identity on $\mc{V}(S)$.
\end{ex}

\subsection{Realizing compact tropical Hopf surfaces}

Starting with a tropical variety $\tX \subset \nT^n$ it may already be difficult to determine when it arises as the tropicalization of some variety $\aX \subset \nA^n$. 
Here we are also asking about the realization of a map between tropicalizations $\tX \subset \nT^n$ and $\tY \subset \nT^m$.
Recall that such a map is given by the extension of an integer affine map $\tF: \nR^n \to \nR^m$  which preserves the sedentarity. 
One way of asking for a realization of $\tF: \tX \to \tY$ is to seek embeddings $i :\nA^d \to \nA^n$ and $j: \nA^d \to \nA^m$ and a map $\aF: \nA^n \to \nA^m$ such that  $\mc{T}_i(\nA^d) = \tX$, $\mc{T}_j(\nA^d) = \tY$ and $\tF \circ \trp = \trp \circ \aF$.
Suppose $\tF: \tX \to \tY$ is the restriction of $F: \nT^n \to \nT^m$ which is a linear invertible monomial map. 
If $\tX = \mc{T}_i(\nA^d)$ for some $i: \nA^d \to \nA^n$, then clearly $\text{Trop}(\aF(i(\nA^d))) = \tY$, and $\aF(i(\nA^d))$ is the image of an embedding $j=\aF \circ i : \nA^d \to \nA^m$. 

Alternatively, we may begin with fixed affine embeddings so that $\mc{T}_i(\nA^d) = \tX \subset \nT^n$ and $\mc{T}_j(\nA^d) = \tY \subset \nT^m$, and given a tropical map $\tF: \tX \to \tY$ we may ask if there is a lift $\aF$ of $\tF$ such that $\aF(i(\nA^d)) = j(\nA^d)$. This turns out to be a more delicate question, which better corresponds to the dynamical setting.

Indeed, in the dynamical situation, we start with an automorphism  $\tF: \tX \to \tX$ of a tropical variety $\tX \subset \nT^n$, where $\tF$ is an integer affine mapping. 
In this case, we want to find an embedding $i:\nA^d \to \nA^n$ satisfying $\tX=\mc{T}_i(\nA^d)$, and a lifting $\aF:\nA^n \to \nA^n$ of the map $\tf$, satisfying the property that $\aF(i(\nA^d)) =i(\nA^d)$.
We can think of this as an automorphism $\af: \nA^d \to \nA^d$, defined by the relation $i \circ \af = \aF \circ i$.

\begin{defi}\label{def:realizeHopf}
A \emph{realization} of a tropical Hopf data  $(\tX, \tF)$ is a pair $(i, \aF)$ where  $i: \nA^d \to \nA^n$ is an embedding such that $\mc{T}_i(\nA^d) = X$, $\trp(\aF) = \tF$ and $\aF(i(\nA^d)) = i(\nA^d)$. If such a pair $(i,\aF)$ exists we say that $(\tX, \tF)$ is \emph{realizable}. 
\end{defi}

\begin{ex}\label{ex:realizeHopf1}
We start from $\tX\subset \nT^3$ which is the hypersurface defined by the tropical polynomial $\tu \wedge (\tx_1 + b) \wedge (\tx_2 + a)$, where on $\nT^3$ we consider the coordinates $(\tx_1,\tx_2,\tu)$.
In other words it is a single modification of $\nT^2$ along the function $(\tx_1 + b) \wedge (\tx_2 + a)$.
It follows that the possible embeddings $i:\nA^2 \to \nA^3$ such that $\tX = \mc{T}_i(\nA^2)$ are of the form
$$
i(\ax_1,\ax_2)=(\ax_1,\ax_2,\gamma\ax_1 + \delta \ax_2),
$$
with $\nu_0(\gamma)=b$ and $\nu_0(\delta)=a$.
The automorphism of $\nT^3$ defined by
$$
\tF(\tx_1, \tx_2, \tu) = (\tx_1 + a , \tu, \tx_2 + 2a)
$$ 
leaves the space $\tX$ invariant.
The space $\tX$ has three top dimensional faces, and one is left invariant while the other two are exchanged (the situation is similar to the one described by Figure \ref{fig:noncompact1}, with $a_1=b_2=a$ and $b_1=b$).
Moreover $\tF$ is a contracting automorphism. 
The tropical space $\tX$ and the map $\tF$ come from the dynamical monomialization of the tropical germ $\tf(\tx_1, \tx_2) = (a+\tx_1, b+\tx_1 \wedge a+ \tx_2)$. 
We can ask if the above tropical Hopf data $(\tX, \tF)$ is realizable by a Hopf data over $\nK$.  

Since the map $\tF$ is integer affine linear, any $\aF$ such that $\trp(\aF) = F$ has the form: 
$$
\aF(\ax_1, \ax_2, \au) = (M_1 \ax_1, M_2 \au, M_3 \ax_2),
$$ 
where $\nu_0(M_1) = a, \nu_0(M_2) = 0$ and $\nu_0(M_3) = 2a$. 

So we are seeking coefficients $M_i$ and $\gamma$, $\delta$, so that $\aF(i(\nA^2)) = i(\nA^2)$.
It turns out that there is not much choice.     
By direct computation,
$$
\aF(i(\ax_1, \ax_2)) = (M_1 \ax_1, M_2(\gamma \ax_1 + \delta \ax_2), M_3 \ax_2),
$$
whose image can be written as
$$
\left\{(\ax_1, \ax_2, \au) \ | \ \au = \frac{M_3}{M_2 \delta}\ax_2 - \frac{M_3 \gamma}{M_1 \delta}\ax_1\right\}.
$$
This coincides with $i(\nA^2)$ if and only if $M_1 = - \delta M_2$.

The embedding $i$ and the map $\aF$ induce an action $\af:\nA^2 \to \nA^2$ given by
$$
(\ax_1, \ax_2) \mapsto (M_1 \ax_1, \gamma M_2 \ax_1 -M_1 \ax_2).
$$
Hence the automorphism $\af$ of $\nA^2$ which realizes the above tropical model is of a specific conjugacy class: a linear map with eigenvalues $M_1$ and $-M_1$.

\end{ex}

\medskip

The next proposition shows that only very specific tropical Hopf data in dimension two  are realizable. The point is that if a $2$-dimensional tropical Hopf data $(\tX, \tF)$ is realizable by a pair $(i, \aF)$ there must exist curves in $\nA^2$ with periodic orbits under the action of the germ $\f$ induced by $i$ and $\aF$.

\begin{prop}\label{prop:realizinggerms}
Given a $2$-dimensional Hopf data $(\tX, \tF)$ where $\tX \subset \nT^n$ and $\tF$ is an invertible contracting germ.
Suppose $\tX$ is not homeomorphic to $\nT^2$ and that $(\tX, \tF)$ is realizable by a Hopf surface defined over $\nK$ by the germ $\af$.
Then $\af$ is conjugate to $\af(\ax_1, \ax_2) = (\alambda \ax_1, \amu \ax_2)$ where $\alambda^k = \amu^l$ for some $k, l \in \nN^*$. 
\end{prop}

\begin{proof}
If $(\tX, \tF)$ is realizable, then the boundary of the quotient $\partial S(\tX, \tF)$ are tropical curves which are realized by curves in $S(\aX, \aF)$.
They correspond to $\tf$-periodic curves in $\nA^2$.
The assumption of $\tX$ not being homeomorphic to $\nT^2$ assures that there are at least $3$ such curves.

By the Poincar\'e-Dulac classification of contracting germs $\af: (\nA^2, 0) \to(\nA^2, 0)$, we have
$$
\af(\ax_1, \ax_2) = (\alambda \ax_1, \amu \ax_2 + \aeps \ax_1^m),
$$
where $0 < \nu_0(\alambda) \leq \nu_0(\amu) < +\infty$ and $\aeps(\amu - \alambda^m) = 0$.
If $\aeps \neq 0$ there is only $1$ curve in $\nA^2$ with periodic orbit under the action of $\af$, it is defined by $\ax_1 =0$, and it is invariant under the action of $\af$. Similarly if $\alambda^k \neq \amu^l$ for all positive integers $k, l$, there are $2$ curves in $\nA^2$  periodic under the action of $\af$. They are defined by $\ax_1 =0$ and $\ax_2 =0$, and again they are invariant.  In either case we obtain a contradiction, which completes the proof.  
\end{proof}

From Proposition \ref{prop:realizinggerms} we directly infer:

\begin{cor}\label{cor:finitemodels}
A contracting invertible germ $\af: (\nA^2 , 0) \to (\nA^2, 0)$ realizes a compact tropical Hopf data $(\tX, \tF)$ if and only if it is conjugate to a diagonal linear map. 
\end{cor}

We can go further and describe the tropical Hopf data $(\tX, \tF)$ of dimension $2$ which are realizable.
Recall that by Proposition \ref{prop:Hopfdatacone}, the tropical variety $\tX$ is the cone over a tree $T$, and all $1$-valent vertices of $T$ are adjacent to unbounded edges.

\begin{thm}\label{thm:listfinitetrophopf}
For a germ $\af(\ax_1, \ax_2) = (\alambda \ax_1, \amu \ax_2)$ let $D$ be the minimal positive integer such that there exists coprime $k, l \in \nN^*$  with $\alambda^{kD} = \amu^{lD}$. If no such $D$ exists set $D =0$. 
Then a tropical Hopf data $(\tX, \tF)$ of dimension $2$ is realizable by $\af$ (up to conjugacy) if an only if the underlying tree $T$ and action $h: T \to T$ induced by $F$ satisfy one  of the following: 

\begin{enumerate}[(a)]
\item $D = 0$ and $T = [-\infty, \infty]$ with $h = id$;  
\item $D = 1$ and $T $ is any metric tree with $h = id$;
\item $D \geq 2$ and there exists a maximal segment $I \subseteq T$ (possibly a point) where $h|_I = \on{id}_I$  such that every point in $T \setminus I$ has exact order $D$. 
\end{enumerate}
\end{thm}

As an easy consequence of this theorem, we get

\begin{cor}
For any tropical Hopf data $(\tX, \tF)$ of dimension $2$, there exists $n \in \nN^*$ so that $(\tX, \tF^n)$ is realizable.
\end{cor}

Before proceeding with the proof of Theorem \ref{thm:listfinitetrophopf}, we need a lemma.

\begin{lem}\label{lem:newEmbedding}
If $(i, \aF)$ is a realization of a $2$-dimensional  tropical Hopf data $(X, F)$, then there is another embedding $i'$ and monomial map $\aF'$ realizing a tropical Hopf data $(\tX', \tF')$ such that:
\begin{itemize}
\item $\{i'_j = 0\}$ are irreducible curves;
\item $\tX'$ is a cone over a tree $T'$ which is isomorphic to $T$ over which $X$ is a cone;
\item the action $h':T' \to T'$ induced by $\tF'$ is conjugate to the action of $h$ on $T$ induced by $F$. 
\end{itemize}
\end{lem}
\begin{proof}Consider the functions whose divisors are the irreducible components of the divisors of functions $i_j$  and the embedding 
$\wt{i}: \nA^2 \to \nA^N$ extending $i$ given by also embedding along those additional functions.
The tropicalizations $\mc{T}_{\wt{i}}$ and $\mc{T}_{i}$ yield isomorphic metric trees $\wt{T}$ and $T$.
The projection map onto the first $n$ coordinates contracts no leaves of $\wt{T}$.
It also contracts no bounded edges since if it did there would be a vertex of $T \subset \nR^{n-1}$ which is locally reducible, again contradicting the fact that $X$ is non-singular.

Now the projection of $\nA^n$ onto the last $N - n$ coordinates gives an embedding $i'$ satisfying that $i_j = 0$ are irreducible curves. Also $\mc{T}_{i'}$ is non-singular since it can be obtained by modification along a collection of irreducible non-singular tropical curves.  By the same argument as above, since $\mc{T}_{i'}$ is non-singular, the tree $T'$ associated to the tropicalization of $\mc{T}_{i'}$ is again isomorphic to $\wt{T}$ and hence to $T$. 
The induced actions $h, \wt{h}$ and $h'$ all coincide. 
\end{proof}

\begin{proof}[Proof of Theorem \ref{thm:listfinitetrophopf}]
The first case is clear, since if $D = 0$ the only periodic curves under the action of $\af$ are $H_1 = \{x_1 = 0\}$ and $H_2 = \{x_2 = 0\}$. Conversely, every diagonal tropical Hopf surface given by the data $X = \nT^2$ and $F$ a translation is realizable. 

Now suppose that $(\tX, \tF)$ is a tropical Hopf data realized by an embedding $i: \nA^2 \to \nA^n$ and a monomial map $\aF: \nA^n \to \nA^n$ whose action corresponds to that of a germ $\af$ with $D \geq 1$. 
Consider the projection of $X \cap \nR^n$ in the cone direction. This quotient is a tropical curve $T$ in $\nR^{n-1} \cong \nR^n / \nR (v_1, \dots , v_n)$.
The tree $T$ is isomorphic to the compactification of this tropical curve given by adding a point at the end of each unbounded ray.
The fact that $\tX$ is non-singular implies that the tropical curve $T$ is locally irreducible. 

The ends of the leaves of $T$ correspond to the tropicalization of the intersection between $\aX=i(\nA^2)$ and coordinate hyperplanes $\{\ay_j=0\}$. 
They correspond to (not necessarily irreducible) curves in $\nA^2$, given by $\{i_j(\ax_1 ,\ax_2)=0\}$, where $i=(i_1, \ldots, i_n)$.

If the curve defined by $\{i_j = 0\}$ is not irreducible, each component must still be reduced otherwise the tropicalization $\mc{T}_i$ would have weights, implying  that $X$ is not non-singular. 
Notice also that the tropicalization of any two irreducible components of $\{i_j = 0\}$ must be distinct curves in $\nT^n_{\{j\}}$ otherwise again $\mc{T}_i$ would have faces with weights greater than $1$. By Lemma \ref{lem:newEmbedding}, we can assume that the curves defined by $\{i_j = 0\}$ are reduced and irreducible. 

We now proceed with the assumption that $\{i_j = 0\}$ is an irreducible curve for every $j  = 1, \dots , n$.  Since $i$ is an embedding (on a neighborhood of $0$), we must have that there exists $j_1, j_2$ so that $(i_{j_1}, i_{j_2})$ give local coordinates for $\nA^2$ at $0$.
In Proposition \ref{prop:realizinggerms}, we deduced that up to conjugacy $\af(\ax_1,\ax_2)=(\alambda \ax_1, \amu \ax_2)$, and also that there exists a minimal $D \in \nN^*$ for which $\alambda^{kD} = \amu^{lD}$ for some $k,l \in \nN^*$ coprime. 
The only periodic curves in $\nA^2$ for $\af$ are $H_1=\{\ax_1 = 0\}$, $H_2=\{\ax_2 = 0\}$ and $C_\aalpha=\{\ax_2^l = \aalpha \ax_1^k\}$ for $\aalpha \in \nK^*$.
Set
$$
P_\aalpha(\ax_1, \ax_2)=\ax_2^l - \aalpha \ax_1^k, \qquad P_0(\ax_1, \ax_2) = \ax_2, \qquad P_\infty(\ax_1, \ax_2)=\ax_1,
$$
and $\ol{\nK} = \nK \cup \{\infty\}.$
Then for any $j=1, \ldots, n$, we have $i_j$ is equal up to constant to some $P_{\alpha}$ with $\alpha \in \ol{\nK}$.

Notice that if we endow the coordinate $\ax_1$ with the weight $l$ and $\ax_2$ with the weight $k$, then $\ax_1$, $\ax_2$ and $P_\aalpha$ with $\aalpha \in \nK^*$ are weighted homogeneous polynomials of degree $l$, $k$ and $kl$ respectively.
This weighted homogeneous structure induces the cone structure on $\tX=\mc{T}_i(\nA^2)$.

Up to rescaling the coordinates corresponding to $\ax_1$ and $\ax_2$, i.e., up to replacing $\ax_1$ and $\ax_2$ by $\ax_1^k$ and $\ax_2^l$, we may assume that all components of $i$ are weighted homogeneous of degree $kl$. This does not change the structure of $T$ as a metric tree (although it changes the metric itself, the two metrics are equivalent). 

Denote by $\on{pr}:\tX \to T$ the natural projection.
Then $T=\on{pr} \circ \nu(\hat{i}(\nK))$ and $T$, where $\hat{i}=(\hat{i}_1, \ldots, \hat{i}_n)$, and if $i_j = P_\aalpha$, then $\hat{i}_j = \hat{P}_\aalpha$, with
$$
\hat{P}_\infty(\az)=1, \qquad \hat{P}_\aalpha(\az)=\az-\aalpha \text{ for all } \aalpha \in \nK.
$$
Let $S$ be the set of all $\aalpha \in \nP^1$ so that $\hat{i_j} = \hat{P}_\aalpha$ for some $j$, and $\ol{S}=S \cup \{\infty\}$. We now require the following lemma. 

\begin{lem}\label{lem:convexhullolS}
The tropicalization $\nu_0(\hat{i}(\nK))$ is isomorphic (as a metric tree) to the convex hull $\mc{V}(\ol{S})$.
\end{lem}
\begin{proof}
Consider the tropicalization map $\rho_{\hat{i}}:\mc{V}(\nP^1) \to \nT^n$ defined in Section \ref{ssec:berktrop}.
We need to show that:
\begin{enumerate}[(a)]
\item $\rho_{\hat{i}}(\mc{V}(\nP^1))=\rho_{\hat{i}}(\mc{V}(\ol{S}))$, and
\item $\rho_{\hat{i}}$ is injective on $\mc{V}(\ol{S})$.
\end{enumerate}

Let $\nu$ be an element in $\mc{V}(\nP^1)$, and consider the retraction $r_{\ol{S}}:\mc{V}(\nP^1) \to \mc{V}(\ol{S})$, defined in Example \ref{ex:valprojline}.
Then $\rho_{\hat{i}}(\nu)=\rho_{\hat{i}}(r_{\ol{S}}(\nu))$ for all $\nu \in \mc{V}(\nP^1)$, and (a) is verified.

Consider now two valuations $\nu_1 \neq \nu_2 \in \mc{V}(\ol{S})$. Then we are in one of the following situations:
\begin{itemize}
\item $\nu_1= \nu_{\aalpha,r}$ and $\nu_2 = \nu_{\aalpha,s}$, and $r \neq s \in [-\infty, +\infty]$.
In this case, $\nu_1(\az-\aalpha)= r \neq s = \nu_2(\az-\aalpha)$, and $\rho_{\hat{i}}(\nu_1)\neq \rho_{\hat{i}}(\nu_2)$.
\item $\nu_1=\nu_{\aalpha,r}$ and $\nu_2=\nu_{\abeta,s}$, with $\aalpha \neq \abeta \in S \cap \nK$, and $r,s \geq \nu_0(\abeta-\aalpha)=:t$.
In this case $\nu_1(\az-\aalpha)=r$, $\nu_2(\az-\aalpha)=t$ and $\nu_1(\az-\abeta)=t$,  $\nu_2(\az-\abeta)=s$.
The values for these two polynomials coincide for $\nu_1$ and $\nu_2$ if and only if $r=s=t$, which would imply $\nu_1 = \nu_2$.
We infer that $\rho_{\hat{i}}(\nu_1)\neq \rho_{\hat{i}}(\nu_2)$.
\end{itemize}
Hence (b) is satisfied, and the proof of the lemma is complete.
\end{proof}

\begin{prop}\label{prop:convexhullS}
The metric tree $T$ is isomorphic to the convex hull $\mc{V}(S)$.
\end{prop}
\begin{proof}
If $\infty \in S$, we get $\ol{S} = S$ and we are done by Lemma \ref{lem:convexhullolS}.
Suppose $\infty \not \in S$. Set $t = \min\{\nu_0(\aalpha-\abeta)\ |\ \aalpha, \abeta \in S\}$, which corresponds to $-\log \on{diam}(S)$.
Let $I=[\nu_\infty, \nu_{\aalpha,t}]$, where $\aalpha$ is any element in $S$. Notice that $I = \ol{\mc{V}(\ol{S}) \setminus \mc{V}(S)}$.
We want to show that $\on{pr}(\rho_{\hat{i}}(I)) = \{\on{pr}(\rho_{\hat{i}}(\nu_{\aalpha,t}))\}$.
In fact, for any $r \in [-\infty, t]$, we have that $\rho_{\hat{i}}(\nu_{\aalpha,r})=(r, \ldots, r)$, and its projection on $T$ does not depend on $r$.
\end{proof}

Proposition \ref{prop:convexhullS} describes the structure of the tree $T$.
The action $h$ induced on $T$ by $\tf$ corresponds to the action induced by $\af$ on $\mc{V}(\nP^1)$; it 
is given by the action of $\af$ on the curves $\{H_1, H_2, C_\aalpha, \aalpha \in \nK^*\}$.
Notice that $H_1$ and $H_2$ are fixed by $\af$, while $\af(C_\aalpha) = C_{\aalpha'}$, with $\aalpha'=\frac{\amu^l}{\alambda^k}\aalpha$. In particular $C_\aalpha$ is periodic of exact period $D$.

If $D=1$, all curves are fixed, and $h$ is the identity: we get case (b).
If $D \geq 2$, we get case (c).
In fact, the set of fixed points $I$ of $h$ corresponds to the set $\mc{V}(S)\cap [\nu_{H_1},\nu_{H_2}]$, where $\nu_{H_j}$ is the curve valuation associated to $H_j$.

This concludes the description of necessary conditions to be realizable.
Notice that if $h$ and $T$ satisfy one of the conditions (a,b,c), then $(\tX, \tf)$ is clearly realizable. See Examples \ref{ex:GraphA}, \ref{ex:GraphB}, \ref{ex:GraphC}, \ref{ex:GraphD} for explicit constructions.
\end{proof}

\begin{rmk}
Suppose that $D \geq 2$ in the statement of Theorem \ref{thm:listfinitetrophopf}.
If $I \neq [-\infty, \infty]$, then we must have at least one curve of the form $C_{\aalpha}$ which is smooth at $0$, i.e., we must have either $k$ or $l$ equal to $1$.
If moreover $I$ is bounded, we must have two such curves $C_{\aalpha}$ and $C_{\abeta}$ so that $(\ax_2^l-\aalpha\ax_1^k,\ax_2^l-\abeta\ax_1^k)$ are local coordinates of $\nA^2$ at $0$. This implies $k=l=1$.
\end{rmk}

\begin{ex}
Not all tropical Hopf data are realizable.
Consider the Hopf data $(\tX,\tF)$ given by Example \ref{ex:permutation}.
Then $(\tX,\tF)$ is realizable if and only if $\tF$ is the composition of a permutation $\sigma$ and a translation, where $\sigma$ is either the identity, or has at most two fixed points 
and the other periods of $\sigma$ are all of the same size. 

As an example, we may take $n=5$ and $\sigma=(1 2)(3 4 5)$.
Notice that $(\tX, \tF^2)$, $(\tX, \tF^3)$, and also $(\tX, \tF^6)$ are realizable. However, $(\tX, \tF)$ is not. 
\end{ex}

\begin{ex}\label{ex:GraphA}
Consider the germ $\af(\ax_1, \ax_2) = (\alambda \ax_1, \amu \ax_2)$ with $\alambda = t^2$ and $\amu = t^3$, so that $\alambda^3=\amu^2$.
With respect to the notations used in the proof of Theorem \ref{thm:listfinitetrophopf}, we have $k=3,l=2,D=1$.

The curves which are periodic under the action of $\af$ are  $H_1=\{\ax_1=0\}$, $H_2=\{\ax_2 = 0\}$, and $C_\aalpha = \{ \ax_2^2 = \aalpha \ax_1^3\}$. In fact, all of these curves are fixed by $\af$. 
Consider now as an example the embedding $i:\nA^2 \to \nA^{7}$ given by
$$
i(\ax_1, \ax_2)=\left(\ax_1,\ax_2, \ax_2^2 - \ax_1^3, \ax_2^2 -(1+t) \ax_1^3, \ax_2^2 - (1+t+t^3)\ax_1^3, \ax_2^2 - (1-t)\ax_1^3, \ax_2^2 - t\ax_1^3\right).
$$

Set $\tX=\mc{T}_i(\nA^2)$ the tropicalization of $\nA^2$ with respect to $i$. Then $\tX$ is a cone over a tree, depicted in Figure \ref{fig:realizabletreefixed}.
Let $\aF:\nA^7 \to \nA^7$ be the map
$$
\aF(\ay_1, \ldots, \ay_7)=\left(t^2\ay_1, t^3\ay_2, t^6\ay_3, t^6\ay_4, t^6\ay_5, t^6\ay_6, t^6\ay_7\right),
$$
and $\tF$ its tropicalization. Then $(\tX,\tF)$ is a Hopf data, realized by $(i, \aF)$.
\end{ex}

\begin{figure}
\begin{subfigure}{.23\columnwidth}
\begin{minipage}[t]{\columnwidth}
	\def\svgwidth{\columnwidth}
	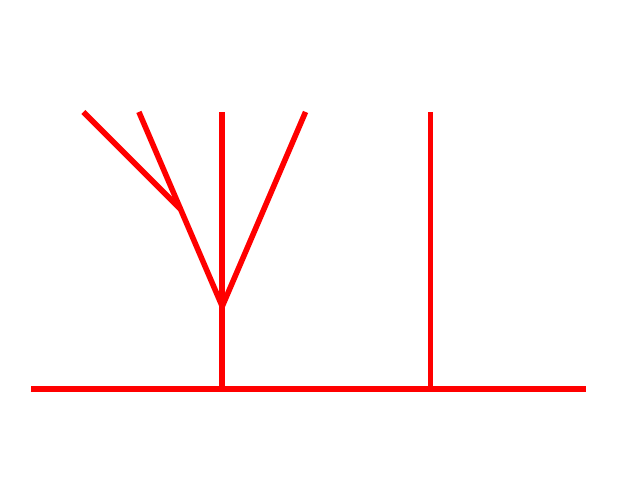
\end{minipage}
\caption{}
\label{fig:realizabletreefixed}
\end{subfigure}
\begin{subfigure}{.23\columnwidth}
\begin{minipage}[t]{1\columnwidth}
	\def\svgwidth{\columnwidth}
	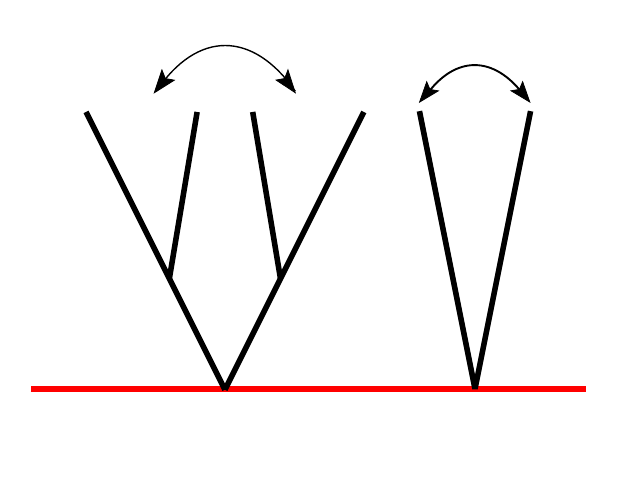
\end{minipage}
\caption{}
\label{fig:realizabletreea}
\end{subfigure}
\begin{subfigure}{.23\columnwidth}
\begin{minipage}[t]{\columnwidth}
	\def\svgwidth{\columnwidth}
	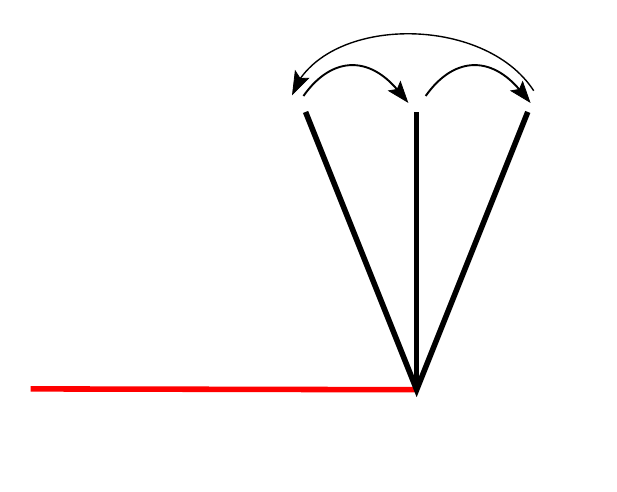
\end{minipage}
\caption{}
\label{fig:realizabletreeb1}
\end{subfigure}
\begin{subfigure}{.23\columnwidth}
\begin{minipage}[t]{1\columnwidth}
	\def\svgwidth{\columnwidth}
	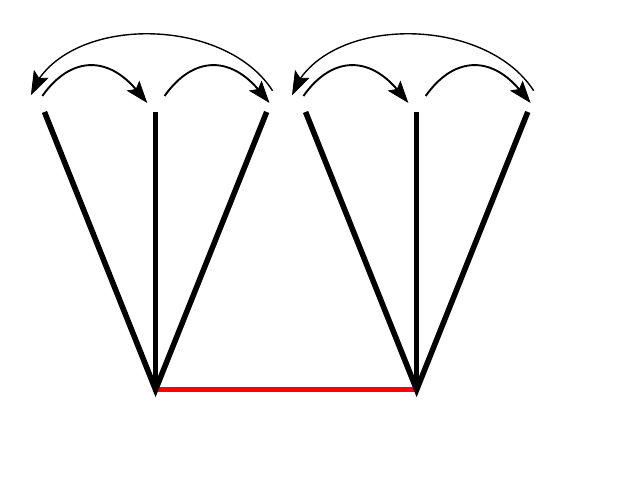
\end{minipage}
\caption{}
\label{fig:realizabletreeb2}
\end{subfigure}
\caption{Graphs over which $\tX$ is the cone for a two dimensional realizable tropical Hopf data $(\tX, \tF)$ from Examples \ref{ex:GraphA}, \ref{ex:GraphB}, \ref{ex:GraphC}, \ref{ex:GraphD}.}
\label{fig:realHopf}
\end{figure}

\begin{ex}\label{ex:GraphB}
Consider the germ $\af(\ax_1, \ax_2) = (\alambda \ax_1, \amu \ax_2)$ with $\alambda = t$ and $\amu = -t^2$, so that $\alambda^4=\amu^2$.
With respect to the notations used in the proof of Theorem \ref{thm:listfinitetrophopf}, we have $k=2,l=1,D=2$.

The periodic curves by $\af$ are $H_1=\{\ax_1=0\}$, $H_2=\{\ax_2 = 0\}$, and $C_\aalpha = \{ \ax_2 = \aalpha \ax_1^2\}$.
Notice that $\af(C_\aalpha)=C_{-\aalpha}$.
Consider now as an example the embedding $i:\nA^2 \to \nA^{8}$ given by
$$
i(\ax_1, \ax_2)=\left(\ax_1,\ax_2, \ax_2 - \ax_1^2, \ax_2 + \ax_1^2, \ax_2 - (1+t)\ax_1^2, \ax_2 + (1+t)\ax_1^2, \ax_2 - t\ax_1^2, \ax_2 + t\ax_1^2\right).
$$

Set $\tX=\mc{T}_i(\nA^2)$ the tropicalization of $\nA^2$ with respect to $i$. Then $\tX$ is a cone over a tree, depicted in Figure \ref{fig:realizabletreea}.
Let $\aF:\nA^8 \to \nA^8$ be the map
$$
\aF(\ay_1, \ldots, \ay_8)=\left(t\ay_1, -t^2\ay_2, -t^2\ay_4, -t^2\ay_3, -t^2\ay_6, -t^2\ay_5, -t^2\ay_8, -t^2\ay_7\right),
$$
and $\tF$ its tropicalization. Then $(\tX,\tF)$ is a Hopf data, realized by $(i, \aF)$.
\end{ex}

\begin{ex}\label{ex:GraphC}
Consider the germ $\af(\ax_1, \ax_2) = (\alambda \ax_1, \amu \ax_2)$ with $\alambda = t$ and $\amu = \zeta t^3$, where $\zeta$ is a primitive $D$-root of unity, $D \geq 2$.
Then $\alambda^{3D} = \amu^D$, and with respect to the notations used in the proof of Theorem \ref{thm:listfinitetrophopf}, we have $k=3,l=1$.

The periodic curves by $\af$ are $H_1=\{\ax_1=0\}$, $H_2=\{\ax_2 = 0\}$, and $C_\aalpha = \{\ax_2 = \aalpha \ax_1^3\}$.
Notice that $\af(C_\aalpha)=C_{\zeta \aalpha}$.

Consider now as an example the embedding $i:\nA^2 \to \nA^{1+D}$ given by
$$
i(\ax_1, \ax_2)=\left(\ax_1,\ax_2 - t^2\ax_1^3, \ax_2 - \zeta t^2\ax_1^3, \ldots, \ax_2 - \zeta^{D-1} t^2\ax_1^3\right).
$$

Set $\tX_i=\mc{T}_i(\nA^2)$ the tropicalization of $\nA^2$ with respect to $i$. Then $\tX$ is a cone over a tree, depicted in Figure \ref{fig:realizabletreeb1}.
Let $\aF:\nA^{1+D} \to \nA^{1+D}$ be the map
$$
\aF(\ay_1, \ldots, \ay_{1+D})=\left(t\ay_1, \zeta t^3 \ay_{1+D}, \zeta t^3 \ay_2, \ldots, \zeta t^3 \ay_D\right),
$$
and $\tF$ its tropicalization. Then $(\tX,\tF)$ is a Hopf data, realized by $(i, \aF)$.
\end{ex}

\begin{ex}\label{ex:GraphD}
Similarly to Example \ref{ex:GraphC}, we consider the germ $\af(\ax_1, \ax_2) = (\alambda \ax_1, \amu \ax_2)$ with $\alambda = t$ and $\amu = \zeta t$, where $\zeta$ is a primitive $D$-root of unity, $D \geq 2$.
In this case $\alambda^{D} = \amu^D$, and we have $k=l=1$.
Consider the embedding $i:\nA^2 \to \nA^{2D}$, given by
$$
i(\ax_1, \ax_2)=\left(\ax_2 - \ax_1, \ax_2 - \zeta \ax_1, \ldots, \ax_2 - \zeta^{D-1}\ax_1, \ax_2 - 2t^u \ax_1, \ax_2 - 2\zeta t^u\ax_1, \ldots, \ax_2 - 2\zeta^{D-1} t^u\ax_1 \right),
$$
where $u \in \nR$.
The tropicalization $\tX=\mc{T}_i(\nA^2)$ is a cone over a tree $T$, depicted in Figure \ref{fig:realizabletreeb2}.
Let $\aF:\nA^{2D} \to \nA^{2D}$ be the map
$$
\aF(\ay_1, \ldots, \ay_{2D})=\left(\zeta t \ay_{D}, \zeta t \ay_1, \ldots, \zeta t \ay_{D-1}, \zeta t \ay_{2D}, \zeta t \ay_{D+1}, \ldots, \zeta t \ay_{2D-1}\right),
$$
and $\tF$ its tropicalization. Then $(\tX,\tF)$ is a Hopf data, realized by $(i, \aF)$.
Notice that the segment $I = \on{Fix}(h)$ is a bounded segment for all $u \in \nR$, and reduces to a point when $u=0$.
\end{ex}

\begin{rmk}
Notice that tropical Hopf manifolds that cannot be realized by primary Hopf manifolds, cannot be realized even by secondary ones.
Secondary Hopf manifolds are obtained by considering the quotient of $\nA^d \setminus \{0\}$ by the cocompact, free and properly discontinuous action of a group $G$.
Assume we have an embedding $i:\nA^d \to \nA^n$, and that the action of $G$ extends to an action on $\nA^n$ that leaves $\aX = i(\nA^d)$ invariant.
By a theorem of Kodaira's, $G$ contains a contraction $\af$, which generates an infinite cyclic subgroup of $G$ of finite index.
In particular $G$ is generated by $\af$ and some torsion elements $\ag$.
The tropicalization $\tg$ of $\ag$ on $\nT^n$ will act as a permutation of coordinates.
In particular, the action induced by $\tg$ is not free.
Denote by $\tX$ the tropicalization of $\aX$. Then the quotient of $\tX \setminus \{\infty\}$ by the action of $G$ cannot produce a tropical manifold. 
\end{rmk}

\begin{rmk}
In this section we used the fact that $\nK$ has characteristic $0$.
In positive characteristic, normal forms of contracting automorphisms $\af:(\nA^d,0) \to (\nA^d,0)$ coincide with the ones given in characteristic $0$ (see \cite[Theorem 1.1]{voskuil:nonarchhopf}, \cite[Theorem 2.7]{ruggiero:rigidgerms}, \cite[Remark 7.2]{ruggiero:superattrdim1charp}).
In dimension $d=2$, the study of periodic curves of a germ  $\af$ defined over a field of positive characteristic is similar to the situation in characteristic $0$.
Notice that a family of periodic curves for $\af$, would induce a family of non-constant meromorphic functions on the Hopf surface induced by $\af$.
By \cite[Theorem 1.3]{voskuil:nonarchhopf}, the new phenomenon in characteristic $p$ is given by the maps $\af(\ax_1, \ax_2)=(\alambda \ax_1, \alambda^m \ax_2 + \aeps \ax_1^m)$, where we get the family of meromorphic functions
$$
Q_\aalpha = \ax_2^p - \aalpha \ax_1^m - \aeps^{p-1} \ax_1^{m(p-1)} \ax_2,
$$
for some $\aalpha \in \nK$.
Denote by $C_\aalpha$ the curve defined by $\{Q_\aalpha=0\}$.
It is easy to show that $\af(C_\aalpha)= C_{\aalpha \alambda^{m(p-1)}}$.
Since $\nu_0(\alambda)>0$, the only curve of this form that is periodic by $\af$ is given by $\aalpha = 0$. 
Therefore, for fields of positive characteristic one can make statements similar to  Theorem \ref{thm:listfinitetrophopf}.
\end{rmk}

\begin{rmk}\label{rmk:noncompactreal}
Realizations can be considered also for non-compact tropical Hopf varieties.
In this case, however, in general there is no cone structure as given by Proposition \ref{prop:Hopfdatacone} (see Example \ref{ex:globalautonocone}), and the geometry of the total space $W^{(\infty)}$ considered in Subsection \ref{ssec:noncompact} is much more complicated than the compact case.

In the case of surfaces, using the normal forms described by Proposition \ref{prop:normalforms2d}, we can notice that if $\tf:\nT^2 \to \nT^2$ is the tropicalization of  a global automorphism $\af$, then the corresponding $W^{(\infty)}$ has a cone structure.
In fact, the only contracting weakly non-degenerate germs which do not come from a global automorphism are of the form \eqref{eqn:nf2dpositive} with $p_2q_1 \geq 2$, or \eqref{eqn:nf2dnegative} with $p_1q_2 \geq 2$, see Example \ref{ex:monomializationmadness}. 

So in dimension $2$ we may apply the same arguments used in the proof of Theorem \ref{thm:listfinitetrophopf} to get a description of the geometry of $W^{(\infty)}$ as a cone over some tree $T$ with infinitely-many leaves.
Leaves of $T$ correspond to a set of irreducible curves, given by (the irreducible components of) the intersection between $i^{(\infty)}(\nA^2)$ and coordinate hyperplanes.
The cone structure implies that there exists weights $l,k$ on $\ax_1, \ax_2$ coordinates in $\nA^2$ so that any such irreducible component is a (irreducible) weighted homogeneous polynomial in $\ax_1, \ax_2$.
Up to multiplying by a constant, they are either $P_\infty=\ax_1$, $P_0=\ax_2$ or of the form $P_\aalpha = \ax_2^k - \aalpha \ax_1^l$ for some $\aalpha \in \nK^*$.
Set $C_\aalpha = \{P_\aalpha = 0\}$ for all $\aalpha \in \ol{\nK}$, and denote by $S$ the set of $\aalpha \in \ol{\nK}$ so that $C_\aalpha$ is one of the curves described above.
As before, the shift $h^{(\infty)}:W^{(\infty)}\to W^{(\infty)}$ induces an automorphism on $S$.

Now, notice that if $\af(\ax_1, \ax_2)=(\alambda \ax_1, \amu \ax_2)$ is diagonal, then for all $\aalpha \in\nK^*$ we get $\af(C_\aalpha) = C_{\aalpha'}$, where $\aalpha'=\frac{\amu^l}{\alambda^k}\aalpha$, while $\af(C_0)=C_0$ and $\af(C_\infty) = C_\infty$.

If $k\nu_0(\alambda)=l \nu_0(\amu)$, we may either have only periodic curves, when $\alambda^k= \amu^l$ as in Theorem \ref{thm:listfinitetrophopf}, or we may have no periodic curves, even though $\nu_0(\aalpha)=\nu_0(\aalpha')$.

If $k\nu_0(\alambda) \neq l \nu_0(\amu)$, we have $\nu_0(\aalpha) \neq \nu_0(\aalpha')$. In this case we need to have $0$ and $\infty$ in $S$, to ensure the existence of the retraction to the convex hull of $S$.

In the non-diagonalizable case $\af(\ax_1, \ax_2)=(\alambda \ax_1, \alambda^m \ax_2 + \aeps \ax_1^m)$, we get $\af(C_\aalpha)= C_{\aalpha'}$ with $\aalpha'=\aalpha + \alambda^{-m} \aeps$ for all $\aalpha \in \nK$, and the only periodic curve is $C_\infty=\af(C_\infty)$.

Considering the iterates of $\af$, we get $\af^n(C_\aalpha)=C_{\aalpha_n}$, where $\aalpha_n = \aalpha + n \alambda^{-m} \aeps$.
Notice that $\nu_0(\aalpha_n)$ is bounded away from $-\infty$, and the convex hull of $S$ is closed in this case.
\end{rmk}

\begin{ex}\label{ex:globalautonocone}
Consider the global automorphism $\af:\nA^3 \to \nA^3$ given by
\begin{equation}
\af(\ax_1,\ax_2,\ax_3)=\big(t \ax_1, t\ax_2+\ax_1^2,t\ax_3 + \ax_1 + \ax_2\big),
\end{equation}
and its tropicalization $\tf$.
The modification $\pi:\tX \to \nT^3$ of Theorem \ref{thm:intromonomialization} is obtained along two divisors, one whose support is  $\{\tx_2+1=2\tx_1\}$, and the other is supported on the union of $\{\tx_3+1=\tx_1\}$, $\{\tx_3+1=\tx_2\}$ and $\{\tx_1=\tx_2\}$. Notice that the direction of the intersection between the last $3$ planes is $(1,1,1)$, which is not tangent to the first plane.
It follows that $\tX$ is not a cone, and consequently $W^{(\infty)}$ is not a cone.
\end{ex}

\begin{rmk}\label{rmk:PDisok}
By Poincar\'e-Dulac theorem, we know that any contracting automorphism is analytically conjugated to a germ $\af:(\nA^d,0) \to (\nA^d,0)$ in Poincar\'e-Dulac normal form. These normal forms are triangular, in the sense that the $k$-th coordinate of $\af$ only depends on the first $k$ coordinates.
In particular, $\af$ defines a global automorphism, and $\af^{-1}$ is also triangular.
Denote by $\alambda_1, \ldots, \alambda_d$ the eigenvalues of the linear part of $\af$, and set
$$
V=\bigcap_{n \in \nZ^d, \alambda^n = 1} \{x \in \nR^d\ |\ \scalprod{n,x}=0\}.
$$
This is a vector space of dimension $\geq 1$ containing a vector $a \in \nN^d$.
Then it is easy to check that every divisor of $\tf_k$ contains the line generated by $a$ for any $a \in V$.
It follows that $W^{(\infty)}$ has a cone structure whenever $\tf$ is the tropicalization of a Poincar\'e-Dulac normal form.
\end{rmk}

\subsection{Connections to rigid analytic Hopf manifolds} 

Here we make a link between tropical Hopf surfaces and non-Archimedean Hopf surfaces which were studied in \cite{voskuil:nonarchhopf}.

An automorphism $\af: \Ad \to \Ad$ induces an automorphism $\af_{\ast} : \mc{V}(\Ad) \to \mc{V}(\Ad)$ by $\af_{\ast} \mu(\aphi) = \mu(\aphi \circ \af)$.
The action of $\af_{\ast}$ on $\mc{V}(\Ad)$ can be connected to tropicalizations as follows. 
Given embeddings $i : \Ad \to \nA^n$ and $j : \Ad \to \nA^n$,  let $\aF : \nA^n \to \nA^n$ be a regular map such that $\aF \circ i  = j \circ \af$. 
Then it is easily shown that the following diagram commutes even if $\aF$ is not equivariant with respect to the natural action of $(\nK^*)^n$, 
$$
\xymatrix@R25pt@C25pt{
{\mc{V}(\Ad)} \ \ar[r]^{\af_{\ast}}\ar[d]_{\rho_i}  & {\mc{V}(\Ad)}\ar[d]^{\rho_j}\\
 {\mc{T}_i (\Ad)} \ar[r]^{\trp(\aF)} &  {\mc{T}_j(\Ad)}
}
$$

In cases of particularly nice tropicalizations, the maps  $\rho_{i}, \rho_j$ admit  sections $s_{i}, s_{j}$. 
For a point $\tx \in \mc{T}_i(\Ad)$, let $m_i(\tx) = m_{i(\Ad)}(\tx)$ denote the multiplicity of the tropicalization $\mc{T}_i$ at $x$ (see Definition \ref{def:tropmult}).
The following theorem is phrased in order to apply it to non-singular tropical varieties described in Definition \ref{def:nonsingular}. In general, the conditions of the intersection of the tropicalization with the boundary strata are weaker.  

\begin{thm}[{\cite[Theorem 10.6]{gubler-rabinoff-werner:skeletons}, \cite[Corollary 8.15]{gubler-rabinoff-werner:tropicalskeletons}}]\label{thm:sectionTrop}
Let $Z \subset \mc{T}_i(\Ad) \subset \nT^n$  be such that the multiplicity of the tropicalization,   $m_i(\tx) = 1$ for all $\tx \in Z$, and that $Z \cap \nR^n_I$ is of the expected dimension for all subsets $I \subset \{1, \dots, n\}$. 
Then the map $\tx \to s_i(\tx)$, where $s_i(\tx)$ is the unique Shilov boundary point of $\rho^{-1}_i(\tx)$, defines a continuous section $Z \to \mc{V}(\Ad)$.
If $Z$ is contained in the closure of its interior in $\mc{T}_i(\Ad) \cap \nR^{n_i}$ then $s_i$ is the unique continuous section of $\rho_i$ defined on $Z$.
\end{thm}

$$
\xymatrix@R25pt@C25pt{
{\mc{V}(\Ad)} \ \ar[r]^{\af_{\ast}}\ar[d]_{\rho_i}  & {\mc{V}(\Ad)}\ar[d]^{\rho_j}\\
{\mc{T}_i (\Ad)} \ar@(l,l)[u]^{s_i} \ar[r]^{\trp(\aF)} &  {\mc{T}_j(\Ad)} \ar@(r,r)[u]_{s_j}
}
$$

When we have such sections, we can consider the action induced by $\af_{\ast}$ on $s_{i}(\mc{T}_i(\Ad))$.
In general however, $\af_\ast$ does not define a dynamical system on $s_i(\mc{T}_i(\Ad))$, not even if $\mc{T}_i(\Ad) = \mc{T}_j(\Ad)$.

\begin{prop}\label{prop:sectionAction}
Let $\af :\Ad \to \Ad$ be a globally contracting automorphism and $i: \Ad \to \nA^n$ be an embedding such that $i(\Ad)$ is not contained in a subspace of $\nA^d$ and there exists a map $\aF: \nA^n \to \nA^n$ satisfying $\aF \circ i = i \circ \af$. 
Suppose that $\rho_i : \mc{V}(\Ad) \to \mc{T}_i(\Ad)$ admits a section $s_i$ on a neighborhood $U$ of $\infty$, given by Theorem \ref{thm:sectionTrop}.  
Then $\af_{\ast}(s_i(U)) \subset s_i(U)$ if and only if $\aF$ is an invertible linear monomial map. 
\end{prop}
\begin{proof}
Suppose that the map  $\aF$ is equivariant, and let $\tF = \text{Trop}(\aF)$.
We can write $\aF(\ax_1, \ldots, \ax_n) = (\alpha_1 \ax_{\sigma(1)}, \ldots, \alpha_n \ax_{\sigma(n)})$.
Since $\rho_i \circ \af_\ast = \tF \circ \rho_i$, we infer $\af_{\ast}(\rho^{-1}_i(\tx)) = \rho_i^{-1}(\tF(\tx))$ for a point $\tx \in U$.
In particular, if $\mu \in \rho_i^{-1}(\tx)$ then
$$
\nu_0(\alpha_k) + \tx_{\sigma(k)} = \af_{\ast} \mu(i_k) = \mu(i_k \circ \af) = \mu(\ax_k \circ \aF \circ i) = \mu(\alpha_k i_{\sigma(k)})
$$
for $1\leq k \leq n$.
Finally if $\mu$ is a Shilov boundary point of $\rho^{-1}(\tx)$, then $\af_{\ast} \mu \in \rho^{-1}(\tF(\tx))$ is a Shilov boundary point of $\rho^{-1}(\tF(\tx))$.
It follows that $\af^{\ast}(s_i(U)) \subset s_i(U)$.

Suppose now that $\af_{\ast}(s_i(U)) \subset s_i(U)$, and set $\tF=\rho_i \circ \af_\ast \circ s_i$, which leaves $\tX=\mc{T}_i(\Ad)$ invariant. Then $F = \trp(\aF)$. 
If $\af_{\ast}(s_i(U)) \subset s_i(U)$ then $\tF = \trp(\aF)$ must be a monomial map.
Otherwise a neighborhood of  $\mc{T}_i(\Ad)$ could not be invariant under $\tF$, since $i(\Ad)$ is not contained in an equivariant subspace of $\nA^n$. 
\end{proof}

\begin{proof}[Proof of Theorem \ref{thm:introantotrop}]
If $\af$ is a \emph{diagonalizable} contracting automorphism, we take coordinates $(\ax_1, \ax_2)$ so that $\af$ is a diagonal linear map on such coordinates.
By applying Proposition \ref{prop:sectionAction} to the trivial embedding $\on{id}:\nA^2 \to \nA^2$, we get a section $s:\nT^2 \to \mc{V}(\nA^2)$. The map $r=s \circ \rho : \mc{V}(\nA^2) \to s(\nT^2)$ is then a retraction, that induces a retraction on the level of Hopf surfaces.

If $\af$ is a \emph{resonant} contracting automorphism, we take coordinates so that $\af$ is in Poincar\'e-Dulac normal form $\af(\ax_1, \ax_2) = (\alambda \ax_1, \alambda^m \ax_2 + \aeps \ax_1^m)$ with $\eps \neq 0$.
Its tropicalization $\tf(\tx_1, \tx_2)=\big(\tlambda + \tx_1, (m \tlambda + \tx_2) \wedge (\eps + m\tx_1)\big)$ satisfies the hypotheses of Corollary \ref{cor:dynsys2d}, so there exists a modification $\pi:\tX \to \nT^2$, and an automorphism $\tF:\nT^3 \to \nT^3$ so that $\tf \circ \pi= \pi \circ \tF$ on $\Strict{\pi}=\{\tu_1=(m \tlambda + \tx_2) \wedge (\eps + m\tx_1)\}$.
Notice that all points in $\tX$ have multiplicity $1$. Explicitly, $\tF(\tx_1, \tx_2, \tu_1)=(\tx_1+\tlambda, \tu_1, \tx_2 + m\tlambda)$.

By Theorem \ref{thm:introrealHopf}, a resonant  germ $\af$ does not realize any tropical Hopf data.
This corresponds to the fact that for any choice of $\aF$ tropicalizing to $\tF$, and of an embedding $i(\ax_1, \ax_2)=(\ax_1, \ax_2, \ax_2 - \aalpha \ax_1^m)$ with $\aalpha \in \nK^*, \nu_0(\aalpha)=\eps-m\tlambda$, we have that $\aF \circ i (\nA^2)$ is not contained in $i(\nA^2)$.
In fact, for any $n \in \nZ$, we get
$$
\af^n(\ax_1, \ax_2)=\big(\alambda^n \ax_1, \alambda^{nm}(\ax_2 + n \aeps \alambda^{-m} \ax_1^m)\big).
$$
Nevertheless we can proceed as in Subsection \ref{ssec:noncompact} and Remark \ref{rmk:noncompactreal}.
We consider the sequence of embeddings $i_n:\nA^2 \to \nA \times \nA^{2n+1}$, $n \in \nN$, given by
$$
i_n(\ax_1, \ax_2)=\big(\ax_1, \ax_2 - n \aeps \alambda^{-m} \ax_1^m, \ldots, \ax_2 + n \aeps \alambda^{-m} \ax_1^m\big).
$$
Denote by $(\ax, \ay_{-n}, \ldots, \ay_{n})$ the coordinates of $\nA \times \nA^{2n+1}$.
For any $n \geq n'$, there are natural projections $\on{pr}_{n,n'}:\nA \times \nA^{2n+1}\to \nA \times \nA^{2n'+1}$ so that $i_{n'}=\on{pr}_{n,n'} \circ i_{n}$.
We may consider the inverse limit $i_\infty:\nA^2 \to \nA \times \nA^\nZ$. If $(\ax, \ay_{n})_{n \in \nZ}$ are coordinates of $\nA \times \nA^{\nZ}$, then $i_\infty$ is defined by $\ax \circ i_\infty (\ax_1, \ax_2)=\ax_1$, and $\ay_n \circ i_\infty(\ax_1, \ax_2)=\ax_2-n \aeps \alambda^{-m} \ax_1^m$ for all $n \in \nZ$.
The map $\aF_\infty: \nA \times \nA^\nZ \to \nA \times \nA^\nZ$ given by $\ax \circ \aF_\infty = \alambda \ax$ and $\ay_n \circ \aF_\infty = \alambda^m \ay_{n+1}$, satisfies $i_\infty \circ \af = \aF_\infty \circ i_\infty$.

Consider $\tX_\infty = \mc{T}_{i_\infty}(\nA^2)$, and let $\rho_\infty : \mc{V}(\nA^2) \to \tX_\infty$ be the tropicalization obtained as the inverse limit of the $\rho_n = \rho_{i_n}$ for $n \to \infty$.
Since all $\tX_n = \mc{T}_{i_n}(\nA^2)$ are regular (all points have multiplicity $1$), by Theorem \ref{thm:sectionTrop}, $\rho_n$ admits a section $s_n:\tX_n \to \mc{V}(\nA^2)$.
Its inverse limit $s_\infty : \tX_\infty \to \mc{V}(\nA^2)$ is a section of $\rho_\infty$.
Again, the map $r_\infty=s_\infty \circ \rho_\infty : \mc{V}(\nA^2) \to s_\infty(\tX_\infty)$ is then a retraction.
We clearly have $s_\infty \circ \tF_\infty = f_\ast \circ s_\infty$, where $\tF_\infty$ is the tropicalization of $\aF_\infty$.
It follows that $r_\infty$ induces a retraction from the Hopf surface associated to $\af$ to the non-compact tropical Hopf surface induced by $(\tX_\infty, \tF_\infty)$.
\end{proof}

\begin{rmk}
The same argument used in the proof of Theorem \ref{thm:introantotrop} can be used to study connections  between analytic and tropical Hopf surfaces in other situations.
For example, if $\af$ is diagonalizable, Theorem \ref{thm:listfinitetrophopf} gives a list of tropical Hopf data $(\tX, \tF)$ that are realizable by $\af$. If $\tX$ has multiplicity $1$, we can use Theorem \ref{thm:sectionTrop} to construct a continuous section $s:\tX \to \mc{V}(\nA^2)$, and conclude as above that the Hopf surface associated to $\af$ contracts to 
the image of the section $s(\tX)$. 
\end{rmk}

\begin{rmk}
Theorem \ref{thm:sectionTrop} and Proposition \ref{prop:sectionAction} hold in all dimensions. To apply them and obtain a retraction of the analytifications of Hopf manifolds to tropical manifolds in higher dimension, some difficulties arise.
\begin{itemize}
\item As noticed in Example \ref{ex:PDhigherdim}, there could be no (compact) tropical Hopf data attached to a contracting \emph{resonant} germ $\af$ in dimension $\geq 3$. This problem can be solved by considering non-compact tropical Hopf data, which always exists, since Poincar\'e-Dulac normal forms are global automorphisms (see Remark \ref{rmk:PDisok}).
\item Even if we can find a dynamical monomialization of a tropicalization $\tf$ of a germ $\af$, we could end up with a tropical space $\tX$ where not all points have multiplicity $1$ (see Example \ref{ex:weight2monom}).
This provides added interest to resolving tropicalizations of multiplicity greater than $1$ by further modifications, as in \cite{cueto-markwig:repairtropplanecurve}. 
\end{itemize}
\end{rmk}

\section{Geometry of tropical Hopf manifolds}\label{sec:invariants}

We finish by discussing the geometry and topology of tropical Hopf manifolds and a comparison to their classical counterparts. 
Firstly, the topology of a tropical Hopf manifold $S(\tX, \tF)$ depends on both $\tX$ and $\tF$, however, it is clear that we have the following: 

\begin{prop}\label{prop:homotopytype}
Every tropical Hopf manifold $S(\tX, \tF)$ has homotopy type $\nS^1$. 
\end{prop}

\subsection{Picard group and divisors}

Complex primary Hopf manifolds have a Picard group isomorphic to $\nC^*$ (see \cite[Theorem 3]{ise:geomhopfmfld}, or \cite[Section V.18]{barth-hulek-peters-vanderven:compactcomplexsurfaces}).
In the setting of rigid geometry (see \cite[Proposition 3.1]{voskuil:nonarchhopf}), an analogous statement holds: the Picard group of a analytic primary Hopf surface over a field $\nK$ is isomorphic to $\nK^*$.
Here we prove the analogous statement in the tropical setting.

We take as the definition of the tropical Picard group $\on{Pic}(S):=H^1(S, \mathcal{O}^{\ast}_S)$.
Here $\mathcal{O}^{\ast}_S$ is the sheaf of groups given by tropical invertible regular functions on $S$. 
For $U \subset \nT^d$, $\mathcal{O}^{\ast}(U)$ consists of functions which coincide with integer affine functions on $U \cap \nR^d \to \nR$ and which do not attain the value $\infty$ when extended to $U$.
Equivalently, if $I$ is the union of sedentarities of all points $x \in U$ then $\mathcal{O}^{\ast}(U)$ consists of integer affine functions given by tropical Laurent monomials not involving $x_i$ for $i \in I$. 

\begin{ex}
Let $U_1, U_2$ be an open cover of $\nT^2 \setminus \{\infty\}$ given by 
$$
U_1 = \{(\tx_1, \tx_2) \ | \ \tx_2 < \tx_1+1\} \text{ and } U_2 = \{ (\tx_1, \tx_2 \ | \ \tx_2 > \tx_1 - 1 \}.
$$
Then 
$$
\mathcal{O}^*(U_1) =   \{ a + k x_2  \ | \ a \in \nR \text{ and } k \in \nZ \}
\text{ and } 
\mathcal{O}^*(U_2) = \{ a + k x_1  \ | \ a \in \nR \text{ and } k \in \nZ \}
$$
$$
\mathcal{O}^*(U_1 \cap U_2) = \{ a + kx_1 + l x_2  \ | \ a \in \nR \text{ and } k, l \in \nZ \}.
$$
\end{ex}

\begin{prop}
Let $S$ be a diagonal tropical Hopf manifold. Then
$$
\on{Pic}(S) \cong \nR.
$$   
\end{prop}
\begin{proof}
Recall that a diagonal tropical Hopf manifold is obtained as the quotient of $\nT^d \setminus \{\infty\}$ by the action of a translation $\tF:\nT^d \to \nT^d$ by a positive vector $a=(a_1, \ldots, a_d) \in (\nR_+^*)^d$. Denote by $\on{pr}:\nT^d\setminus \{\infty\} \to S$ the natural projection.

Any $1$-cocycle $\tau$ on $S$ with coefficients in $\mathcal{O}_S^*$ can be pulled back to a $1$-cocycle $\wt{\tau}$ on $\nT^d \setminus \{\infty\}$ in the standard way.
To do this, let us suppose that $\tau$ is defined on the covering $\{U_\alpha\}$ of $S$ and the value on $U_\alpha \cap U_\beta$ is $\tau_{\alpha\beta}$. 
This covering induces a covering $\{\wt{U}_{\alpha}\}$ of $\nT^d\setminus \{\infty\}$, where $\wt{U}_{\alpha} = \on{pr}^{-1}(U_\alpha)$ for all $\alpha$. 
Then the value of $\wt{\tau}$ on $\wt{U}_{\alpha} \cap \wt{U}_{\beta} \neq \emptyset$ is $\wt{\tau}_{\alpha\beta} = \tau_{\alpha\beta} \circ \on{pr}$.

The $1$-cocycle $\wt{\tau}$ is a coboundary since all bundles on $\nT^n \setminus \{\infty\}$ are trivial. Hence there exists a function $T: \nT^d \setminus \{\infty\} \to  \nT^d \setminus \{\infty\}$ so that $\coboundary T = \wt{\tau}$. The map $T$ is determined up to additive constant, and is unique if we impose $T(0,\dots ,0)=0$.
We now define the map $l$ from $1$-cocycles to $\nR$, as $l(\tau)=T(\tF(\tx)) - T(\tx)$ for any $x \in \nT^d \setminus \{\infty\}$.
Since $\tau$ is a $1$-cocycle in $S$, its lift $\wt{\tau}$ is invariant by $\tF$, hence $\wt{\tau}=\coboundary T=\coboundary (T \circ \tF)$, and $T$ and $T \circ \tF$ differ by a constant, hence $l(\tau)$ does not depend on $\tx$, nor on the choice of $T$.

The map $l:C^1(S, \mc{O}_S^{\ast})\to \nR$ descends to cohomology as an injective morphism, since if $\tau = \coboundary \sigma$ for 
$\sigma = C^0(S, \mc{O}_S^{\ast})$, then $\wt{\tau} = \coboundary \wt{\sigma}$, where $\wt{\sigma}$ is the lift of $\sigma$, satisfying $\wt{\sigma} \circ \tF = \wt{\sigma}$.
Hence $l(\tau)=\wt{\sigma}(\tF(\tx)) - \wt{\sigma}(\tx)=0$.

\begin{figure}
\begin{minipage}[t]{.8\columnwidth}
	\def\svgwidth{\columnwidth}
	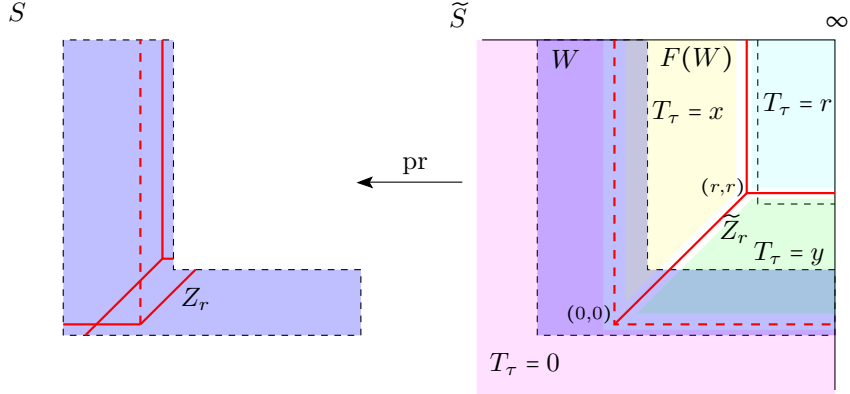
\end{minipage}
\caption{Construction of the cycle $Z_r$ in dimension $d=2$.}\label{fig:picard}
\end{figure}

Lastly, we show that the map $l:H^1(S, \mc{O}_S^{\ast})\to \nR$ is surjective.
To do this, we first construct a tropical cycle $Z'_r$ in $\nT^d$ for any $r \in \nR$, as $Z'_r = \on{div}_{\nT^d}(\Phi_r)$, where
$$
\Phi_r(\tx_1 , \dots \tx_d) = (\tx_1 \wedge \dots \tx_d \wedge r) - (\tx_1 \wedge \dots \tx_d \wedge 0),
$$
see the right side of Figure \ref{fig:picard}, where the values of $\Phi_r$ are explicitly indicated.

Denote by $Z_r=\on{pr}_*Z'_r$ the push forward of $Z'_r$ on $S$ (the left side of Figure \ref{fig:picard}), and let $\tau_r=[Z_r]$ be the $1$-cocycle associated to $Z_r$ on $S$.
We claim that $l(\tau_r)=r$.

In fact, its lift $\wt{\tau}_r$ is the $1$-cocycle associated to the divisor $\sum_{n \in \nZ} (\tF^n)^* Z'_r$. It follows that $\wt{\tau}_r=\coboundary T_r$, where
$
T_r(\tx) = \displaystyle \sum_{n \in \nZ} \big(\Phi_r \circ \tF^n(\tx) - \Phi_r \circ \tF^n(0)\big).
$
Hence
$$
l(\tau_r)= \sum_{n \in \nZ} \big(\Phi_r \circ \tF^{n+1}(\tx) - \Phi_r \circ \tF^n(\tx)\big) = \lim_{n \to +\infty} \Big(\Phi_r\big(\tF^n(\tx)\big)-\Phi_r\big(\tF^{-n}(\tx)\big)\Big) = r-0 = r,
$$
and we are done.
\end{proof}

From the last part of the above proof we obtain the next corollary which says that every element of $\on{Pic}(S)$ can be represented by a codimension $1$-tropical cycle in $S$.
This is not the case for classical Hopf surfaces over a field (see \cite[Theorem 3.1]{dabrowski:modulihopfsurfaces}).
\begin{cor}
Every line bundle on a (diagonal) tropical Hopf manifold $S(\tX, \tF)$ has a tropical section, i.e., can be represented by a tropical divisor. 
\end{cor}

\subsection{$(p, q)$-homology of Hopf manifolds}
Lastly we look at the tropical $(p, q)$-homology groups of a diagonal Hopf manifold $S$ of dimension $d$.
Definitions of tropical $(p, q)$-homology can be found in \cite{itenberg-katzarkov-mikhalkin-zharkov:trophomology}, \cite{mikhalkin-zharkov:eigenwave}, or the more introductory \cite{brugalle-itenberg-mikhalkin-shaw:gokova}.

For tropical (diagonal) Hopf manifolds, we get the following statement.

\begin{prop}\label{prop:homologydiag}
Let $S$ be a tropical diagonal Hopf manifold of dimension $d$.
Then we have
$$
H_{p,q}(S) \cong \begin{cases}
                  \nR & \text{if } (p,q) \in \{(0,0), (0,1), (d,d-1), (d,d)\},\\
                  0 & \text{otherwise}.
                 \end{cases}
$$
\end{prop}
\begin{rmk}
In particular, the numbers $h_{p,q}(S)  = \dim (H_{p, q}(S))$ can be arranged in the diamond shape found below. 
These numbers are the same as the dimensions of Dolbeaut cohomology groups of a complex Hopf manifold (see \cite[Theorem 4(c)]{ise:geomhopfmfld}). 
Because of the lack of symmetry about the vertical axis in the diamond, the tropical Hopf surface provides a simple counter-example to Conjecture 5.3 from \cite{mikhalkin-zharkov:eigenwave}. 
$$
\xymatrix@R0pt@C2pt{
& & & & 1 & & & &\\
& & & 0 & & 1 & & &\\
& & 0 & & 0 & & 0 & &\\
& \udots & \vdots & \ddots & \vdots & \ddots & \vdots & \ddots & \\
0 & \cdots & 0 & & 0 & & 0 & \cdots & 0 \\
& \ddots & \vdots & \ddots & \vdots & \ddots & \vdots & \udots & \\
& & 0 & & 0 & & 0 & &\\
& & & 1 & & 0 & & & \\
& & & & 1 & & & & \\
}
$$
\end{rmk}
\begin{proof}
The manifold $S$ is obtained as a quotient of $\nT^d\setminus \{\infty\}$ by the action of a translation $\tf$ by a positive vector $a \in (\nR_+^*)^d$.
The stratification of $\nT^d$ induces a stratification on $S$.
We now describe all faces of $S$.
Then for any $I \neq \{1, \ldots, d\}$, we have a face $\tau_I$, of dimension $d-\abs{I}$, obtained by quotienting $\nT^d_I \setminus \{\infty\}$ by the action induced by $\tf$.
A stratum $\tau_{I'}$ contains another stratum  $\tau_I$ in its closure if and only if $I' \subseteq I$.
Strata are depicted by Figure \ref{fig:homology3d} when $d=3$.

Denote by $\mc{F}_p(\tau_I)$ the $p$-tangent space of $S$ at the face $\tau_I$, and by $\rho_p(I',I):\mc{F}_p(\tau_{I'}) \to \mc{F}_p(\tau_I)$ the maps between $p$-tangent spaces associated to a stratum $\tau_{I'}$ dominating $\tau_I$.
It is easy to check that
$$
\mc{F}_p(\tau_I) = \Lambda^p\big(\nR^n_I\big) \qquad \text{ and } \quad \rho_p(I',I):\Lambda^p\big(\nR^{J'}\big) \to \Lambda^p\big(\nR^n_I\big)
$$
is the natural projection.
In particular $\mc{F}_p(\tau_I) = 0$ and $\rho_p(I',I) = 0$ whenever $\abs{I} > d-p$.
See Figure \ref{fig:homology2d} for multi-tangent spaces in dimension $d=2$.

For $p=0$, we have $H_{0,q}(S) \cong H_q(S, \nR)$. Since $S$ is homotopically equivalent to $\nS^1$ (see Proposition \ref{prop:homotopytype}), we get $H_{0,q}(S) \cong \nR$ for $q=0,1$, and it is null otherwise.

For $p=d$, we get $\mc{F}_p(\tau_I)=0$ for all $I \neq \emptyset$. It follows that $H_{d,q}(S) \cong H_q(S, \partial S)$.
Notice that $S$ is homeomorphic to $\nS^1 \times \nD^{d-1}$, where $\nD^{d-1}$ is a closed $(d-1)$-dimensional ball. It follows that $\partial S$ is homeomorphic to $\nS^1 \times \nS^{d-2}$.
By a direct computation using the long exact sequence for the homology of pairs, we get $H_{d,q}(S) \cong \nR$ if $q=d,d-1$, and $0$ otherwise.

For $p=1, \ldots, d-1$, we proceed as follows.
Let $\wt{A}$ be the interior of a fundamental domain for $\tF$ in $\nT^d$, and $\wt{B}$ be the translation of $\wt{A}$ by $\ta/2$, where $\tF$ is the translation by $\ta \in \nR^d$.
Let $A$ and $B$ be the images of $\wt{A}$ and $\wt{B}$ by the natural projection $\nT^d\setminus \{\infty\} \to S$.
Then $S = A \cup B$ and $A \cap B$ consist of  two connected components.

Firstly, $A$ and $B$ are two copies of the interior of a fundamental domain  $W^o$ of the action of $F$. Moreover, the two connected components of $A \cap B$ are $\wt{W}^o$ where $\wt{W}$ is, up to translation, the fundamental domain rescaled by some homothety. Therefore, the subsets $A$, $B$ and each connected component of $A \cap B$ all have the same $(p,q)$-homology groups.
Directly calculating the tropical homology groups of $A$ we obtain, 
$H_{p,q}(A) = \nR$  for  $(p,q) = (0,0)$ or $(d, d-1)$, and $H_{p,q}(A) = 0$ otherwise. 
Therefore, we also have $H_{p,q}(B) = H_{p,q}(A \cap B) = 0$ for $p = 0$ or $d$.
It can be checked that these are the same homology groups as for $\nT^d \setminus \{\infty\}$. 

For a fixed $p$ there is a Mayer-Vietoris long exact sequence for the $(p, q)$-homology groups,
$$
\cdots \to H_{p,q}(A \cap B) \to H_{p,q}(A) \oplus H_{p,q}(B) \to H_{p,q}(S) \to H_{p,q-1}(A \cap B) \to \cdots
$$
Using the above long exact sequence and the above calculations we have that 
$H_{p,q}(S) = 0$ for all $q$ when $p \neq 0,d$. 
This completes the calculation of all tropical homology groups. 
\end{proof}

\begin{figure}
\begin{subfigure}{.49\columnwidth}
\begin{minipage}[t]{\columnwidth}	
	\def\svgwidth{\columnwidth}
	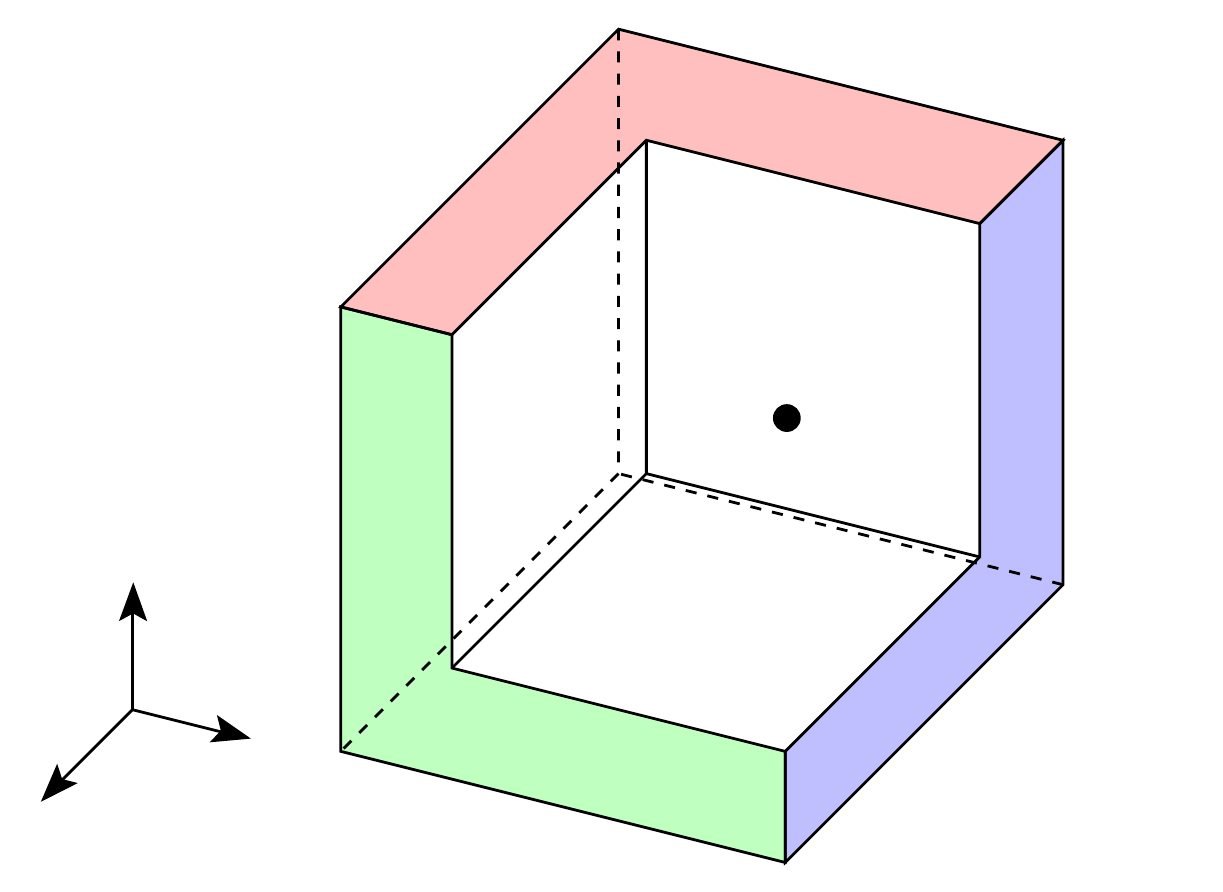
\end{minipage}
\caption{Strata for $d=3$}\label{fig:homology3d}
\end{subfigure}
\begin{subfigure}{.49\columnwidth}
\begin{minipage}[t]{.88\columnwidth}
	\def\svgwidth{\columnwidth}
	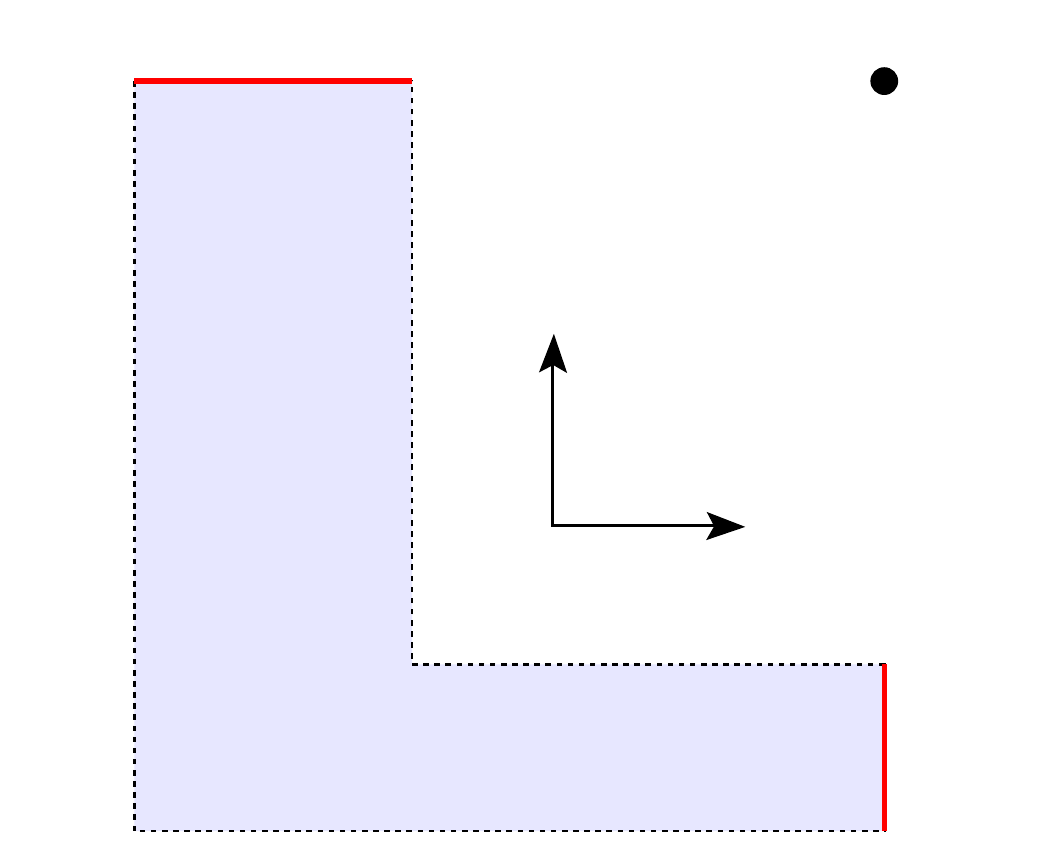
\end{minipage}
\caption{Strata and multi-tangent spaces for $d=2$}\label{fig:homology2d}
\end{subfigure}
\caption{Strata and multi-tangent spaces on a fundamental domain for a tropical diagonal map.\vspace{-1mm}}\label{fig:homology}
\end{figure}

\begin{rmk}
Tropical homology groups remain invariant under modifications of non-singular tropical manifolds, see \cite{shaw:surfaces}. 
Therefore the statement in Proposition \ref{prop:homologydiag} remains valid for Hopf data which are modifications of a diagonal Hopf data. 
We expect that taking finite quotients also does not effect the above calculation, which would imply that all tropical Hopf manifolds have the same tropical Betti numbers. 
\end{rmk}

\bibliographystyle{alpha}
\bibliography{../biblio}

\end{document}